\documentclass{amsart}

\usepackage{amsfonts,mathrsfs}
\usepackage{amsmath,amssymb,amsthm, amscd, bm}
\usepackage{tikz}
\usetikzlibrary{arrows,calc,matrix,topaths,positioning,scopes,shapes,decorations,decorations.markings} 
\usepackage{dsfont}
\usepackage{epsfig}
\usepackage{tikz-cd}
\usepackage{hyperref}
\usepackage{fullpage}
\usepackage[foot]{amsaddr}
\usepackage{adjustbox}
\usepackage{enumitem}

\usepackage[utf8x]{inputenc}

\makeatletter

\renewcommand{\email}[2][]{%
  \ifx\emails\@empty\relax\else{\g@addto@macro\emails{,\space}}\fi%
  \@ifnotempty{#1}{\g@addto@macro\emails{\textrm{(#1)}\space}}%
  \g@addto@macro\emails{#2}%
}
\makeatother

\author{Hiroaki  Karuo$^{(1)}$}
\address{${}^{(1)}$ Gakushuin University, Faculty of Sciences, Department of mathematics, 1-5-1 Mejiro, Toshima-ku, Tokyo 171-8588 Japan}
\email{hiroaki.karuo@gakushuin.ac.jp}
\author{Julien Korinman$^{(2)}$}
\address{${}^{(2)}$ Institut Montpelli\'erain Alexander Grothendieck - UMR 5149 Universit\'e de Montpellier. Place Eug\'ene Bataillon, 34090 Montpellier France}
\email{julien.korinman@gmail.com}
\urladdr{https://sites.google.com/site/homepagejulienkorinman/}

\subjclass{$57$R$56$,  $57$M$25$.}


\def\restriction#1#2{\mathchoice
              {\setbox1\hbox{${\displaystyle #1}_{\scriptstyle #2}$}
              \restrictionaux{#1}{#2}}
              {\setbox1\hbox{${\textstyle #1}_{\scriptstyle #2}$}
              \restrictionaux{#1}{#2}}
              {\setbox1\hbox{${\scriptstyle #1}_{\scriptscriptstyle #2}$}
              \restrictionaux{#1}{#2}}
              {\setbox1\hbox{${\scriptscriptstyle #1}_{\scriptscriptstyle #2}$}
              \restrictionaux{#1}{#2}}}
\def\restrictionaux#1#2{{#1\,\smash{\vrule height .8\ht1 depth .85\dp1}}_{\,#2}}

\newcommand{\quotient}[2]{{\raisebox{.2em}{$#1$}\left/\raisebox{-.2em}{$#2$}\right.}}

\newcommand{\sslash}{\mathbin{/\mkern-6mu/}}
\newcommand{\Hom}{\operatorname{Hom}}
\newcommand{\tr}{\operatorname{tr}}
\newcommand{\Tr}{\operatorname{Tr}}
\newcommand{\qtr}{\operatorname{qtr}}
\newcommand{\SL}{\operatorname{SL}}
\newcommand{\id}{id}
\newcommand{\Span}{\operatorname{Span}}
\newcommand{\End}{\operatorname{End}}

\newcommand{\PGL}{\operatorname{PGL}}
\newcommand{\Vect}{\operatorname{Vect}}
\newcommand{\Specm}{\operatorname{MaxSpec}}
\newcommand{\Mod}{\operatorname{Mod}}
\newcommand{\qdim}{\operatorname{qdim}}

\newcommand{\Mat}{\operatorname{Mat}}
\newcommand{\skein}{\mathcal{S}_A}
\newcommand{\Aut}{\operatorname{Aut}}
\newcommand{\Rib}{\operatorname{Rib}}
\newcommand{\Rker}{\operatorname{Rker}}
\newcommand{\LKer}{\operatorname{Lker}}
\newcommand{\col}{\operatorname{col}}

\newcommand{\qbinom}[2]{
\left[ 
\begin{array}{l} #1 \\ #2 \end{array}
\right]
}

\newcommand{\Crosspos}{
\tikz[baseline=-0.4ex,scale=0.5,>=stealth]{	
\draw [fill=gray!45,gray!45] (-.6,-.6)  rectangle (.6,.6);
\draw[line width=1.2,-] (-0.4,-0.52) -- (.4,.53);
\draw[line width=1.2,-] (0.4,-0.52) -- (0.1,-0.12);
\draw[line width=1.2,-] (-0.1,0.12) -- (-.4,.53);
}}

\newcommand{\Crossneg}{
\tikz[baseline=-0.4ex,scale=0.5,>=stealth]{	
\draw [fill=gray!45,gray!45] (-.6,-.6)  rectangle (.6,.6);
\draw[line width=1.2,-] (-0.4,0.53) -- (.4,-.52);
\draw[line width=1.2,-] (-0.4,-0.52) -- (-0.1,-0.12);
\draw[line width=1.2,-] (0.1,0.12) -- (.4,.53);
}}

\begin{document}

\theoremstyle{plain}
\newtheorem{theorem}{Theorem}[section]
\newtheorem{proposition}[theorem]{Proposition}
\newtheorem{corollary}[theorem]{Corollary}
\newtheorem{lemma}[theorem]{Lemma}
\theoremstyle{definition}
\newtheorem{notations}[theorem]{Notations}
\newtheorem{convention}[theorem]{Convention}
\newtheorem{problem}[theorem]{Problem}
\newtheorem{definition}[theorem]{Definition}
\theoremstyle{remark}
\newtheorem{remark}[theorem]{Remark}
\newtheorem{conjecture}[theorem]{Conjecture}
\newtheorem{example}[theorem]{Example}
\newtheorem{strategy}[theorem]{Strategy}
\newtheorem{question}[theorem]{Question}

\title[Azumaya loci of skein algebras]{Azumaya loci of skein algebras}
%
%
%

\date{}
\maketitle


\begin{abstract} 
We compute the Azumaya loci of Kauffman-bracket skein algebras of closed surfaces at odd roots of unity and provide partial results for open surfaces as well. 
As applications, we give an alternative definition of the projective representations of the Torelli groups derived from non-semisimple TQFTs and we strengthen a result by Frohman-Kania Bartoszynska-L\^e about the dimensions of some quotients of the skein modules of closed 3-manifolds. 
\end{abstract}


\section{Introduction}

\subsection{Background on skein algebras and their representations}\label{sec_background}

Let $\Sigma_g$ be an oriented connected closed genus $g$ surface, $N\geq 3$ an odd integer,  $A\in \mathbb{C}^*$ a root of unity of order either $N$ or $2N$ and $\mathcal{S}_A(\Sigma_g)$ the associated Kauffman-bracket skein algebra. A representation $r: \skein(\Sigma_g)\to \End(V)$ is a \textit{weight representation} if $V$ is semisimple as a module over the center of $\skein(\Sigma_g)$. The purpose of this paper is to make progresses towards the 

\begin{problem}\label{problem_classification}
Classify all  finite dimensional weight representations of  $\skein(\Sigma_g)$. 
\end{problem}
Our study is based on several recent results that we now summarize briefly and review in details in Section \ref{sec_skein}. 

\vspace{2mm}
\par \underline{\textit{Chebyshev-Frobenius morphisms:}}
Write $\varepsilon = A^N$, so $\varepsilon=+1$  if $A$ has order $N$ and $\varepsilon = -1$ if $A$ has order $2N$.
Bonahon and Wong defined in \cite{BonahonWong1} an injective morphism of algebras
$$ Ch_A : \mathcal{S}_{\varepsilon}(\Sigma_g) \hookrightarrow \mathcal{Z}\left( \skein(\Sigma_g) \right)$$
named \textit{Chebyshev-Frobenius morphism}, from the (commutative) skein algebra at $A=\varepsilon$ into the center of the skein algebra at root of unity of odd order. It is proved in \cite{FrohmanKaniaLe_UnicityRep} that $Ch_A$ is onto. 
A classical result of Bullock \cite{Bullock, PS00, ChaMa} combined with the result in \cite{Barett} shows the existence of a (non-canonical when $\varepsilon=+1$) isomorphism
$$ \Psi : \mathcal{S}_{\varepsilon}(\Sigma_g) \xrightarrow{\cong} \mathcal{O}[\mathcal{X}_{\SL_2}(\Sigma_g)]$$
between the skein algebra at $A=\varepsilon$ and the ring of regular functions of the character variety defined by the GIT quotient
$$ \mathcal{X}_{\SL_2}(\Sigma_g) := \Hom(\pi_1(\Sigma_g), \SL_2) \sslash \SL_2.$$
Every indecomposable weight representation of $\skein(\Sigma_g)$ induces a character over the center of $\skein(\Sigma_g)$, thus a closed point in $ \mathcal{X}_{\SL_2}(\Sigma_g) $ named its \textit{classical shadow}.

\vspace{2mm}
\par \underline{\textit{Azumaya locus:}} Let $[\rho]\in  \mathcal{X}_{\SL_2}(\Sigma_g) $ be the class of a representation $\rho : \pi_1(\Sigma_g) \to \SL_2$ seen as a closed point in $\mathcal{X}_{\SL_2}(\Sigma_g)$ and denote by $\chi_{[\rho]}$ the corresponding induced character over the center of   $\skein(\Sigma_g)$. Write $\mathcal{I}_{[\rho]}
\subset \skein(\Sigma_g)$ the ideal generated by elements  $\chi_{[\rho]}(z)-Ch_A(z)$ for all $z\in \mathcal{S}_{\varepsilon}(\mathbf{\Sigma})$ and consider the finite dimensional algebra
$$ \left( \skein(\Sigma_g) \right)_{[\rho]} := \quotient{ \skein(\Sigma_g)}{\mathcal{I}_{[\rho]}}.$$
The \textit{Azumaya locus} of $\mathcal{S}_A(\Sigma_g)$ is defined as the subset
$$ \mathcal{AL}:= \left\{ [\rho] \in \mathcal{X}_{\SL_2}(\Sigma_g) \mbox{ such that } \left( \skein(\Sigma_g) \right)_{[\rho]} \cong \Mat_{N^{3g-3}}(\mathbb{C}) \right\} \subset \mathcal{X}_{\SL_2}(\Sigma_g).$$
Therefore, for every $[\rho]\in \mathcal{AL}$, there exists a unique indecomposable weight representation $ r_{[\rho]} : \skein(\Sigma_g)\to \End(V_{[\rho]})$ with classical shadow $[\rho]$: this representation corresponds to the unique irreducible representation of $\Mat_{N^{3g-3}}(\mathbb{C}) $, so $r_{[\rho]}$ is irreducible and has dimension $N^{3g-3}$. Therefore, the representation theory of $\skein(\Sigma_g)$ over its Azumaya locus is trivial.
 A classical result of De Concini-Kac \cite{DeConciniKacRepQGroups}, further extended in \cite{BrownGoodearl, Brown_AL_discriminant, FrohmanKaniaLe_UnicityRep}, asserts that if a complex algebra is $(i)$ affine, $(ii)$ prime and $(iii)$ of finite rank over its center, then its Azumaya locus is a Zariski open dense set. This result was applied to $\skein(\Sigma_g)$ in \cite{FrohmanKaniaLe_UnicityRep} to prove that $\mathcal{AL}\subset \mathcal{X}_{\SL_2}(\Sigma_g)$ is open dense.
 
 \vspace{2mm}
\par \underline{\textit{Poisson order:}} 
The character variety $\mathcal{X}_{\SL_2}(\Sigma_g)$ has a Poisson structure defined by Goldman in \cite{Goldman86}. When $g\geq 2$, it decomposes as 
$$ \mathcal{X}_{\SL_2}(\Sigma_g) = \mathcal{X}^{(0)} \sqcup \mathcal{X}^{(1)} \sqcup \mathcal{X}^{(2)}$$
where:
\begin{itemize}
\item 
$\mathcal{X}^{(0)}$ is the smooth locus which corresponds to the classes of irreducible representations $\rho : \pi_1(\Sigma_g) \to \SL_2$, 
\item
$\mathcal{X}^{(1)}$ is the smooth locus of $\mathcal{X}_{\SL_2}(\Sigma_g) \setminus \mathcal{X}^{(0)} $ and corresponds to the classes of representations which are diagonal but not central, 
\item $\mathcal{X}^{(2)}\cong \mathrm{H}^1(\Sigma_g ;  \mathbb{Z}/2\mathbb{Z})$ is the finite locus of central representations.
\end{itemize}
Both $\mathcal{X}^{(1)}$ and $\mathcal{X}^{(0)}$ are smooth symplectic varieties and the symplectic leaves of $\mathcal{X}_{\SL_2}(\Sigma_g)$ are the sets  $\mathcal{X}^{(1)}$, $\mathcal{X}^{(0)}$ and the singletons $\{[\rho]\}$ for $[\rho]\in \mathcal{X}^{(2)}$. 
Ganev-Jordan-Safronov noticed in \cite{GanevJordanSafranov_FrobeniusMorphism} that skein algebras admit  structures of Poisson orders over $\mathcal{X}_{\SL_2}(\Sigma_g)$. The theory of Poisson orders was initiated by De Concini-Kac in \cite{DeConciniKacRepQGroups} and fully developed in \cite{BrownGordon_PO}. The main result of the theory implies that whenever $[\rho_1]$ and $[\rho_2]$ belong to the same symplectic leaf, then 
$$ \left( \skein(\Sigma_g) \right)_{[\rho_1]} \cong \left( \skein(\Sigma_g) \right)_{[\rho_2]}.$$
In particular, if a leaf intersects non-trivially the Azumaya locus, then it is included in it. Since both the Azumaya locus and $\mathcal{X}^{(0)}$ are dense, they intersect non-trivially, therefore one has (this is part of  \cite[Theorem $1.1$]{GanevJordanSafranov_FrobeniusMorphism}):
\begin{equation}\label{eq_AL1}
 \mathcal{X}^{(0)} \subset \mathcal{AL}.
 \end{equation}
An alternative proof of (a weaker version of) this result can also be found in \cite{DetcherrySantharoubane_UnpuncturedQTrace}.

\par \underline{\textit{WRT representations:}} When $[\rho_1], [\rho_2] \in \mathcal{X}^{(2)}$, we easily see that the two algebras $ \left( \skein(\Sigma_g) \right)_{[\rho_1]}$ and $ \left( \skein(\Sigma_g) \right)_{[\rho_2]}$ are isomorphic so either all the points of $\mathcal{X}^{(2)}$ belong to the Azumaya locus or none of them does.
The Witten-Reshetikhin-Turaev TQFTs in \cite{Wi2, RT} induce representations $r^{WRT}: \skein(\Sigma_g) \to \End( V(\Sigma_g))$ which are known to be irreducible \cite{GelcaUribe_SU2} and to have classical shadows in $\mathcal{X}^{(2)}$ \cite{BonahonWong4}. Moreover, their dimensions are given by the so-called Verlinde formula which satisfies $1<\dim (V(\Sigma_g))<N^{3g-3}$, so these representations are not induced by  representations of $\Mat_{N^{3g-3}}(\mathbb{C}) $. Therefore the existence of the WRT representations proves that 
\begin{equation}\label{eq_AL2}
 \mathcal{AL}\cap \mathcal{X}^{(2)} = \emptyset.
 \end{equation}

\subsection{Main results}

By Equations \eqref{eq_AL1} and \eqref{eq_AL2}, the only remaining open question regarding the Azumaya locus of $\skein(\Sigma_g) $ is how it intersects the leaf $\mathcal{X}^{(1)}$. The goal of the present paper it to answer this question.
 By the theory of Poisson orders, we are left with the alternative: 
$$ \mbox{either one has  } \mathcal{X}^{(1)} \subset \mathcal{AL} \mbox{, or } \mathcal{X}^{(1)} \cap \mathcal{AL} = \emptyset.$$

\begin{theorem}\label{main_theorem}
One has $ \mathcal{X}^{(1)} \subset \mathcal{AL}$, therefore the Azumaya locus of  $\skein(\Sigma_g) $ is the locus of non-central representations.
\end{theorem}

This theorem disproves a conjecture of L\^e and Yu in \cite{LeYu_Survey} who conjectured that the Azumaya locus of  $\skein(\Sigma_g) $ was the smooth locus of $\mathcal{X}_{\SL_2}(\Sigma_g)$.
\vspace{2mm}
\par 
The strategy to prove Theorem \ref{main_theorem} goes as follows. By a theorem of Posner, (see \cite[Theorem III.I.6]{BrownGoodearl}) if $[\rho]$ does not belong to the Azumaya locus, then every irreducible representations of  $\left( \skein(\Sigma_g) \right)_{[\rho]}$ have dimension strictly smaller than $N^{3g-3}$, so we get the following alternative characterization of the Azumaya locus:
$$\mathcal{AL}= \{ [\rho] \mbox{ s.t. }[\rho]\mbox{ is the classical shadow of an irreducible representation of (maximal) dimension } N^{3g-3} \}.$$
In particular if  $r: \skein(\Sigma_g)\to \End(V)$ is a central representation of dimension $N^{3g-3}$ whose classical shadow is $[\rho]$ then $[\rho]$ belongs to the Azumaya locus if and only if $r$ is irreducible. 
\par For every $[\rho] \in \mathcal{X}^{(1)}$, there exist two central representations of dimension $N^{3g-3}$ whose classical shadow is $[\rho]$: 
\begin{enumerate}
\item the representations defined by Bonahon and Wong in \cite{BonahonWong3} using quantum Teichm\"uller theory and 
\item the representations defined by Blanchet-Costantino-Geer-Patureau Mirand in \cite{BCGPTQFT} derived from non-semisimple TQFTs.
\end{enumerate}

The skein representations defined in \cite{BCGPTQFT} are part of a larger algebraic structure named (non-semisimple) \textit{extended topological quantum field theory} well studied in \cite{BCGPTQFT, DeRenzi_NSETQFT}. This additional structure permits one to define interesting bases and non-degenerate pairings which permit one to perform explicit computations.
Theorem \ref{main_theorem} will be derived from the following:

\begin{theorem}\label{theorem2} The skein representations derived from non-semisimple TQFTs with classical shadows in $\mathcal{X}^{(1)}$ are irreducible.
\end{theorem}

Note that, by the theory of Poisson orders, either every such representation is irreducible or none of them is. We will prove that the latter alternative leads to a contradiction with the non-vanishing of certain $6j$-symbols.

\begin{corollary}\label{coro_intro}
 Any two skein representations from \cite{BonahonWong3} and \cite{BCGPTQFT} respectively having the same classical shadow in $\mathcal{X}^{(1)}$ are isomorphic.
\end{corollary}

Therefore, the only remaining question to complete Problem \ref{problem_classification} for closed surfaces is to classify all indecomposable modules of the finite dimensional algebras  $\left( \skein(\Sigma_g) \right)_{[\rho]}$ for $[\rho] \in \mathcal{X}^{(2)}$ (these algebras are pairwise isomorphic).
\vspace{2mm}
\par A first consequence of Theorem \ref{main_theorem} is that  it permits one to give an alternative very simple definition of the projective representations of the Torelli groups defined in \cite{BCGPTQFT} (see Section \ref{sec_MCG}). A second consequence will be stated in Corollary \ref{coro2} below. 
\par 
\vspace{2mm} \par Let $\Sigma_{g,n}$ be a genus $g$ surface with $n\geq 1$ boundary components. If $Z$ denotes the center of $\mathcal{S}_A(\Sigma_{g,n})$, then $\widehat{\mathcal{X}}(\Sigma_{g,n}):= \Specm(Z)$ is a finite branched covering of the character variety $\mathcal{X}_{\SL_2}(\Sigma_{g,n})$. Its closed points are $n+1$-tuples $([\rho], z_1, \ldots, z_n)$ with $[\rho]\in \mathcal{X}_{\SL_2}(\Sigma_{g,n})$ and $z_i\in \mathbb{C}$ such that if $\gamma_{i}$ is the peripheral curve encircling the $i$-th boundary component once, then $T_N(z_i)=-\tr(\rho(\gamma_i))$ (here $T_N$ is the Chebyshev polynomial of the first kind).
\par The Azumaya locus $\mathcal{AL}\subset \widehat{\mathcal{X}}(\Sigma_{g,n})$ is still a dense open subset and we call \textit{fully Azumaya locus} the subset $\mathcal{FAL}\subset \mathcal{X}_{\SL_2}(\Sigma_{g,n})$ of classes $[\rho]$ such that $([\rho], z_1, \ldots, z_n) \in \mathcal{AL}$ for every $z_i$ (satisfying $T_N(z_i)=-\tr(\rho(\gamma_i))$). Using Poisson orders, we see that the  fully Azumaya locus is a union of symplectic leaves. Let $\nu : \mathcal{X}_{\SL_2}(\Sigma_{g,n}) \to \mathbb{C}^n$, $\nu([\rho]):= (\tr(\rho(\gamma_1)), \ldots, \tr(\rho(\gamma_n)))$. The pull-backs $\mathcal{X}_{\SL_2}(\Sigma_{g,n}, \mathbf{z}):= \nu^{-1}(\mathbf{z})$ are called \textit{relative} or \textit{sliced character varieties} and the smooth loci of these sliced character varieties are symplectic leaves of $\mathcal{X}_{\SL_2}(\Sigma_{g,n})$. In particular, unlike the closed case, no symplectic leaf is open dense. However, by analyzing the quotient algebras $\quotient{\mathcal{S}_A(\Sigma_{g,n})}{(\gamma_i -z_i)}$, Frohman-Kania Bartoszynska-L\^e proved in \cite{FKL_GeometricSkein} that the smooth loci of sliced character varieties are included in the fully Azumaya loci. The computation of these smooth loci is not known in general though many partial results are presented in  \cite{FKL_GeometricSkein}.
\par  For $\Sigma_{1,1}$ and $\Sigma_{0,4}$, Takenov exhibited in \cite{Takenov_Azumaya} some open dense subsets which are included in the Azumaya loci of their skein algebras. 
 Large families of  irreducible representations of $\mathcal{S}_A(\Sigma_{1,1})$ were defined by Havl\'i\v{c}ek-Po\v{s}ta
in \cite{HavlicekPosta_Uqso3} (after remarking that  $\mathcal{S}_A(\Sigma_{1,1})\cong U'_q\mathfrak{so}_3$), and Takenov's result for $\Sigma_{1,1}$  is derived from this work. A classical result of Fricke asserts that $\mathcal{X}_{\SL_2}(\Sigma_{1,1})\cong \mathbb{C}^3$ through the isomorphism sending $[\rho]$ to $\{\tr(\rho(\alpha_i)\}_{i=1,2,3}$ where  $\alpha_1, \alpha_2, \alpha_3$ are described in Figure~\ref{fig_holed_torus}. 
The \textit{quaternionic representation} is the point $[\rho_0]$ sent to $(0,0,0)$ by this isomorphism.

 Fix  $A^{1/2}\in \mathbb{C}^*$ and let $q$ denote $A^{2}$.
\begin{theorem}\label{theorem_opensurfaces} Suppose that $\chi(\Sigma_{g,n}) <0$. Let $x=([\rho], z_1, \ldots, z_n)\in \widehat{\mathcal{X}}_{\SL_2}(\Sigma_{g,n})$ and write $c_i:=\tr(\rho(\gamma_i))$. 
\begin{enumerate}
\item If $(g,n)\neq (1,1)$ and there exists $i$ such that $z_i=-q-q^{-1}$ then $x\notin \mathcal{AL}$. So if there exists $i$ such that $c_i=2$, then $[\rho]\notin \mathcal{FAL}$. 
\item Both $\mathcal{AL}$ and $\mathcal{FAL}$ are invariant under the natural $\mathrm{H}^1(\Sigma_{g,n}; \mathbb{Z}/2\mathbb{Z})$ action. So if $n\geq 2$ and $c_i=\pm 2$ for some $i$, then $[\rho]\notin \mathcal{FAL}$.
\item If $[\rho]$ is central and $z_i=-q^{n_i} - q^{-n_i}$ with $n_i \in \{1, \ldots, N-1\}$, then $x \notin \mathcal{AL}$. So no central representation is fully Azumaya.
\item If $g\geq 1$ and $[\rho]$ is diagonal non-central and for all $1\leq i \leq n$ either $c_i \neq \pm 2$ or $z_i= \pm 2$, then $([\rho], z_1, \ldots, z_n) \in \mathcal{AL}$. In particular, if $[\rho]$ is diagonal non-central and $c_i \neq \pm 2$ for all $i$, then $[\rho]\in \mathcal{FAL}$.  
\item  Let $x=([\rho],z) \in \widehat{\mathcal{X}}_{\SL_2}(\Sigma_{1,1})$.
\begin{itemize}
\item If $[\rho]$ is central and $z\neq -2$, $x$ does not belong to the Azumaya locus.
\item If $[\rho]=[\rho_0]$ is the quaternionic representation, then $x$ belongs to the Azumaya locus if and only if $z=-2$.
\item If $[\rho]$ is neither central nor quaternionic, then $x$ belongs to the Azumaya locus.
\end{itemize}

\end{enumerate}
\end{theorem}
The first two points are immediate. The third follows from the existence of the Witten-Reshetikhin-Turaev representations whereas the fourth follows from our study of the representations in  \cite{BCGPTQFT}. The  last point follows by combining the works in   \cite{FKL_GeometricSkein} and  \cite{HavlicekPosta_Uqso3}. In this case, we could not decide whether the four points $([\rho], -2)$ for $[\rho]$ central belong to the Azumaya locus or not. 
Soon after the prepublication of the present paper, Yu proved in \cite{Yu_ALSkein_SmallSurfaces} that these four points belong to the Azumaya locus and compute the Azumaya locus of $\mathcal{S}_A(\Sigma_{0,4})$ as well.

\vspace{2mm}
\par For $M$ a $3$-manifold, the Chebyshev-Frobenius morphism $Ch_A: \mathcal{S}_{\varepsilon}(M)\to \mathcal{S}_A(M)$ turns $\mathcal{S}_A(M)$ into a $\mathcal{S}_{\varepsilon}(M)$-module; equivalently it defines a coherent sheaf  $\mathscr{L}^M\to \mathcal{X}_{\SL_2}(M)$. For $[\rho]\in \mathcal{X}_{\SL_2}(M)$ corresponding to a maximal ideal  $\mathfrak{m}_{[\rho]} \subset \mathcal{S}_{\varepsilon}(M)$, the \textit{reduced skein module} is the quotient:
$$ \mathcal{S}_A(M)_{[\rho]} := \quotient{\mathcal{S}_A(M)}{Ch_A(\mathfrak{m}_{[\rho]})\mathcal{S}_A(M)}. $$
Equivalently, it is the space of sections of the fiber $\restriction{\mathscr{L}^M}{[\rho]}$ of $\mathscr{L}^M$ over $[\rho]$. 
In  \cite{FKL_GeometricSkein}, the authors proved that for $M$ a closed $3$-manifold and $[\rho]\in \mathcal{X}_{\SL_2}(M)$ the class of an irreducible representation then $ \mathcal{S}_A(M)_{[\rho]}$
is one-dimensional as a consequence of the fact that irreducible characters are in the Azumaya locus of closed surfaces $\Sigma_g$ and of punctured disks $\Sigma_{0,g+1}$.
 As explained in Section \ref{sec_onedim}, their argument extends word-by-word to deduce from Theorem \ref{main_theorem} and \ref{theorem_opensurfaces}, the following stronger statement.

\begin{theorem}\label{coro2} Let $M$ be a closed oriented $3$ manifold and let $[\rho] \in \mathcal{X}_{\SL_2}(M)$ be the class of a non-central representation.  Then $ \mathcal{S}_A(M)_{[\rho]}$
is one-dimensional.
\end{theorem}

So $\mathscr{L}^M$  is a line bundle over the locus of non-central representations. For applications of Theorem \ref{coro2}, see \cite{DetcherryKalfagianniSikora_SkeinDim, Detcherry_NSTQFT}. 
 Let $\mathcal{S}_v(M) $ be the skein module over the ring $\mathbb{Q}[v^{\pm 1}]$, so $v$ is generic formal parameter.
We say that the skein module $\mathcal{S}_v(M)$ is $A$-\textit{tame} if is a direct sum of a cyclic $\mathbb{Q}[v^{\pm 1}]$-modules which does not contain any summand of the form $\quotient{\mathbb{Q}[v^{\pm 1}]}{(\phi_N(v))}$ if $\varepsilon = +1$ or of the form $\quotient{\mathbb{Q}[v^{\pm 1}]}{(\phi_{2N}(v))}$ if $\varepsilon=-1$, where $\phi_N(X)$ is the cyclotomic polynomial; see \cite{DetcherryKalfagianniSikora_SkeinDim} for a slightly different definition of $A$-tame. 
Using the results of \cite{DetcherryKalfagianniSikora_SkeinDim}, we easily deduce the following:

\begin{theorem}\label{theorem_central_intro} Suppose that $M$ is a closed $3$-manifold such that $(1)$ $\mathcal{X}_{\SL_2}(M)$ is finite and reduced and $(2)$ the skein module $\mathcal{S}_v(M)$  is $A$-tame. Then for any central representation $[\rho]$, $ \mathcal{S}_A(M)_{[\rho]}$ is one-dimensional. So $\mathscr{L}^M \to \mathcal{X}_{\SL_2}(M)$ is a line bundle.
\end{theorem}

It is very hard to decide when $M$ satisfies hypotheses $(1)$ and $(2)$; see however  \cite{DetcherryKalfagianniSikora_SkeinDim} for large classes of examples. 
 Two years after the prepublication of the present paper, the second author proved in \cite{Koju_SkeinLineBundle} that 
 $\mathscr{L}^M\to \mathcal{X}_{\SL_2}(M)$  is a line bundle for any closed $3$-manifold $M$ without assumptions $(1)$ or $(2)$ needed.

\subsection{Plan of the paper}
In Section \ref{sec_skein}, we review the results stated in Subsection \ref{sec_background}, mainly the Chebyshev-Frobenius morphism, the Azumaya locus, the theory of Poisson orders and the WRT representations.
 Sections \ref{sec_category} and \ref{sec_representations} are devoted to the study of the skein representations derived from non-semisimple TQFTs. In the original construction in \cite{BCGPTQFT}, the authors considered $A$ to be a root of unity of even order, whereas our paper is concerned about odd roots of unity, so we will need to adapt the construction in \cite{BCGPTQFT} to our setting.  In Section \ref{sec_category} we describe the category underlying the construction of these non-semisimple TQFTs and in Section \ref{sec_representations} we construct the skein representations derived from these TQFTs and  prove Theorem \ref{theorem2}. 
In Section \ref{sec_opensurfaces} we extend the previous results to open surfaces and prove Theorem \ref{theorem_opensurfaces}.
In Section \ref{sec_MCG} we give a simple alternative construction of the representations of the Torelli groups arising in non-semisimple TQFTs.
In Section \ref{sec_onedim} we prove Theorems \ref{coro2} and \ref{theorem_central_intro}. 
In the appendices, we prove some technical results concerning the multiplicity modules and $6j$-symbols in the category of representations of the unrolled quantum group.

\textit{Acknowledgments.} The authors thank S.Baseilhac, D.Calaque,  F.Costantino,  R.Detcherry, M. De Renzi, C.Frohman,   J. March\'e, T.L\^e and R.Santharoubane for valuable conversations and the anonymous referee for her/his careful analysis and important suggestions which improved the quality and exactness of the paper. J.K. acknowledges support from the Japanese Society for Promotion of Sciences, from the Centre National de la Recherche Scientifique and from the European Research Council (ERC DerSympApp) under the European Union’s Horizon 2020 research and innovation program (Grant Agreement No. 768679).
 Part of this paper was accomplished while he was hosted by Waseda University.
 H.K. was supported by JSPS KAKENHI Grant Numbers JP22K20342, JP23K12976.

\section{Preliminary results on Kauffman-bracket skein algebras}\label{sec_skein}

\subsection{Kauffman-bracket skein algebras}

\begin{definition}
Let $M$ be a compact oriented $3$-manifold, $k$ be a (unital, associative) commutative ring and $A\in k^{\times}$ an invertible element. The \textit{Kauffman-bracket skein module} $\mathcal{S}_A(M)$ is the quotient of the $k$-module freely generated by embedded framed links $L\subset M$ by the ideal generated by elements $L-L'$, for $L,L'$ two isotopic links, and by the following skein relations: 
  	\begin{equation*} 
\begin{tikzpicture}[baseline=-0.4ex,scale=0.5,>=stealth]	
\draw [fill=gray!45,gray!45] (-.6,-.6)  rectangle (.6,.6)   ;
\draw[line width=1.2,-] (-0.4,-0.52) -- (.4,.53);
\draw[line width=1.2,-] (0.4,-0.52) -- (0.1,-0.12);
\draw[line width=1.2,-] (-0.1,0.12) -- (-.4,.53);
\end{tikzpicture}
=A
\begin{tikzpicture}[baseline=-0.4ex,scale=0.5,>=stealth] 
\draw [fill=gray!45,gray!45] (-.6,-.6)  rectangle (.6,.6)   ;
\draw[line width=1.2] (-0.4,-0.52) ..controls +(.3,.5).. (-.4,.53);
\draw[line width=1.2] (0.4,-0.52) ..controls +(-.3,.5).. (.4,.53);
\end{tikzpicture}
+A^{-1}
\begin{tikzpicture}[baseline=-0.4ex,scale=0.5,rotate=90]	
\draw [fill=gray!45,gray!45] (-.6,-.6)  rectangle (.6,.6)   ;
\draw[line width=1.2] (-0.4,-0.52) ..controls +(.3,.5).. (-.4,.53);
\draw[line width=1.2] (0.4,-0.52) ..controls +(-.3,.5).. (.4,.53);
\end{tikzpicture}
\hspace{.5cm}
\text{ and }\hspace{.5cm}
\begin{tikzpicture}[baseline=-0.4ex,scale=0.5,rotate=90] 
\draw [fill=gray!45,gray!45] (-.6,-.6)  rectangle (.6,.6)   ;
\draw[line width=1.2,black] (0,0)  circle (.4)   ;
\end{tikzpicture}
= -(A^2+A^{-2}) 
\begin{tikzpicture}[baseline=-0.4ex,scale=0.5,rotate=90] 
\draw [fill=gray!45,gray!45] (-.6,-.6)  rectangle (.6,.6)   ;
\end{tikzpicture}
.
\end{equation*}
When $M=\Sigma \times [0,1]$ is a thickened surface, then $\mathcal{S}_A(\Sigma):=\mathcal{S}_A(\Sigma\times [0,1])$ has a structure of associative unital algebra where 
the product of two classes links  $[L_1]$ and $[L_2]$ is defined by  isotoping $L_1$ and $L_2$  in $\Sigma\times (1/2, 1) $ and $\Sigma\times (0, 1/2)$ respectively and then setting $[L_1]\cdot [L_2]:=[L_1\cup L_2]$. 
\end{definition}

A \textit{multi-curve} $\gamma \subset \Sigma$ is a one-dimensional compact submanifold without contractible component. It defines a framed link through the inclusion $\gamma \subset \Sigma \times \{1/2\} \subset \Sigma \times (0,1)$ with blackboard framing. A \textit{curve} is a connected multi-curve. 

\begin{theorem}\label{theorem_basic}
\begin{enumerate}
\item (Przytycki \cite{Przytycki_skein}) The set 
$$\mathcal{B} := \{ [\gamma]\mid  \gamma \mbox{ is a multi-curve in }\Sigma\}$$
is a basis of $\skein(\Sigma)$.
\item (Bullock \cite{BullockGeneratorsSkein}, see also \cite{AbdielFrohman_SkeinFrobenius, FrohmanKania_SkeinRootUnity, SantharoubaneSkeinGenerators}) There exists a finite set of curves whose images generate $\skein(\Sigma)$ as an algebra.
\item (Przytycki-Sikora \cite{PrzytyckiSikora_SkeinDomain} for closed surfaces, Bonahon-Wong \cite{BonahonWongqTrace} for open surfaces) $\skein(\Sigma)$ is a domain.
\end{enumerate}
\end{theorem}

When $A=\pm 1$, the skein algebra $\mathcal{S}_{\pm 1}(\Sigma)$ is commutative: this is a straightforward consequence of the following transparency relation: 
$$\Crosspos
= \pm 
\begin{tikzpicture}[baseline=-0.4ex,scale=0.5,>=stealth] 
\draw [fill=gray!45,gray!45] (-.6,-.6)  rectangle (.6,.6)   ;
\draw[line width=1.2] (-0.4,-0.52) ..controls +(.3,.5).. (-.4,.53);
\draw[line width=1.2] (0.4,-0.52) ..controls +(-.3,.5).. (.4,.53);
\end{tikzpicture}
 \pm
\begin{tikzpicture}[baseline=-0.4ex,scale=0.5,rotate=90]	
\draw [fill=gray!45,gray!45] (-.6,-.6)  rectangle (.6,.6)   ;
\draw[line width=1.2] (-0.4,-0.52) ..controls +(.3,.5).. (-.4,.53);
\draw[line width=1.2] (0.4,-0.52) ..controls +(-.3,.5).. (.4,.53);
\end{tikzpicture}
= \Crossneg.$$
\par The $n$-th Chebyshev polynomial of first kind is the polynomial  $T_n(X) \in \mathbb{Z}[X]$ defined by the recursive formulas $T_0(X)=2$, $T_1(X)=X$ and $T_{n+2}(X)=XT_{n+1}(X) -T_n(X)$ for $n\geq 0$. It is characterized by the formula
$$T_n(\tr(M)) = \tr(M^n) \mbox{, for all }M\in \SL_2(\mathbb{C}).$$

\begin{theorem}\label{theorem_chebyshev}
Let $N\geq 3$ an odd integer.
Suppose that $A\in \mathbb{C}^*$ is a root of unity of  order  $N$ or $2N$. Write $\varepsilon = +1$ if the order of $A$ is $N$ and $\varepsilon = -1$ if the order of $A$ is $2N$.
\begin{enumerate}
\item (Bonahon-Wong \cite{BonahonWong1}) There exists an embedding 
$$Ch_A: \mathcal{S}_{\varepsilon}(\Sigma) \hookrightarrow \mathcal{Z}\left(\skein(\Sigma) \right)$$
into the center of $\skein(\Sigma)$ characterized by the formula $Ch_A([\gamma]) := T_N([\gamma])$ for $\gamma\subset \Sigma$ a curve.
\item (Frohman-Kania Bartoszynska-L\^e \cite{FrohmanKaniaLe_UnicityRep}) For $\Sigma$ a closed surface $Ch_A$ is onto, so provides an isomorphism between the skein algebra at $A=\varepsilon$ and the center $Z_{\Sigma}$ of $\skein(\Sigma)$. 
Moreover, $\mathcal{S}_A(\Sigma)$ is finitely generated over $Z_{\Sigma}$. 
\end{enumerate}
\end{theorem}

\begin{proposition}\label{prop_Barett}(Barrett \cite{Barett}) Let $S$ be a spin structure on $\Sigma$ with Johnson quadratic form $w_S: \mathrm{H}_1(M; \mathbb{Z}/2\mathbb{Z}) \to \mathbb{Z}/2\mathbb{Z}$. There is an algebra isomorphism $\mathcal{S}_{-A}(\Sigma) \xrightarrow{\cong} \mathcal{S}_A(\Sigma)$ sending $[\gamma]$ to $(-1)^{w_S([\gamma])} [\gamma]$. 
\end{proposition}

\subsection{Azumaya locus}

\begin{definition} A $\mathbb{C}$-algebra $\mathcal{A}$  is \textit{almost Azumaya} if: 
\begin{itemize}
\item[(i)] $\mathcal{A}$ is affine , i.e. $A$ is a finitely generated $\mathbb{C}$-algebra; 
\item[(ii)] $\mathcal{A}$ is prime, i.e. any $a,b\in A$ satisfy that if $arb=0$ for all $r\in A$ then $a=0$ or $b=0$;
\item[(iii)]$ \mathcal{A}$ has finite rank over its center $Z$.
\end{itemize}
\end{definition}

Let $S:= Z\setminus \{0\}$ so that $K:=ZS^{-1}$ is a field. A theorem of Posner-Formanek (see \cite[Theorem $I.13.3$]{BrownGoodearl}) asserts that the localization $\mathcal{A}\otimes_Z K$ is a central simple algebra over $K$, therefore there exists a finite extension $\overline{K}$ of $K$ such that $\mathcal{A}\otimes_Z \overline{K} \cong \Mat_D(\overline{K})$ is a matrix algebra. The integer $D$ is called the \textit{PI-degree} of $\mathcal{A}$ and is characterized by the formula:
$$ \dim_{\overline{K}} \mathcal{A}\otimes_Z \overline{K} = D^2.$$
By Theorem \ref{theorem_chebyshev}, $\mathcal{S}_A(\Sigma)$ is almost Azumaya. 

\begin{theorem}(Frohman-Kania Bartoszynska-L\^e \cite{FrohmanKaniaLe_DimSkein}) \label{theorem_PIDeg}
The algebra $\mathcal{S}_A(\Sigma_g)$ has PI-degree $N^{3g-3}$ if $g\geq 2$ and $N$ if $g=1$.
\end{theorem}


Let $\mathcal{A}$ be an almost Azumaya algebra with center $Z$,  $\mathcal{X}:=\Specm(Z)$ and for each  maximal ideal $x \in \mathcal{X}$,  consider the finite dimensional algebra 
$$\mathcal{A}_x:= \quotient{\mathcal{A}}{x\mathcal{A}}.$$
Note that when $\mathcal{A}$ is a domain, by the Artin-Tate lemma, $Z$ is finitely generated so $\mathcal{X}$ is an irreducible affine variety. 

\begin{definition}
Let $\rho: \mathcal{A}\to \End(V)$ be a representation. 
\begin{enumerate}
\item 
$\rho$ is a  \textit{central representation}  if 
it sends central elements  to  scalar operators. In this case, it induces a character over the center $Z$, so a  point in $\mathcal{X}$ named its \textit{classical shadow}.
\item $\rho$ is a \textit{weight representation} if $V$ is semisimple as a $Z$-module, i.e. if $V$ is a direct sum of central representations.
\end{enumerate}
\end{definition}

Indecomposable weight representations are clearly central. A central representation  $\rho: \mathcal{A}\to \End(V)$ with classical shadow $x$ induces by quotient a representation $\rho : \mathcal{A}_x \to \End(V)$. 

\begin{definition} The \textit{Azumaya locus} of $\mathcal{A}$ is the subset
$$\mathcal{AL}:= \{ x\in \mathcal{X}, \mbox{ such that } \mathcal{A}_x \cong \Mat_D(\mathbb{C}) \}.$$
\end{definition}

Note that $\Mat_D(\mathbb{C})$ is simple with only simple module  the standard module $\mathbb{C}^D$.
So if $x\in \mathcal{AL}$, there exists a unique indecomposable representation of $\mathcal{A}$ with classical shadow $x$ (up to isomorphism); this representation is irreducible and has dimension $D$.
The following theorem is essentially due to De Concini-Kac in their original work on quantum enveloping algebras \cite{DeConciniKacRepQGroups} and was further generalized by several authors including \cite{DeConciniLyubashenko_OqG, Brown_AL_discriminant, BrownGoodearl, FrohmanKaniaLe_UnicityRep}.
\begin{theorem}\label{theorem_AL_dense} Let $\mathcal{A}$ be an almost Azumaya algebra with PI-degree $D$.
\begin{enumerate}
\item (\cite[Theorem III.I.7]{BrownGoodearl}) The Azumaya locus is a Zariski open dense subset.
\item  (\cite[Theorem III.I.6]{BrownGoodearl}) If $\rho: \mathcal{A}\to \End(V)$ is an irreducible representation whose classical shadow is not in the Azumaya locus, then $\dim(V)<D$. So the Azumaya locus admits the following alternative description:
$$\mathcal{AL}=\{ x \in \mathcal{X} | x \mbox{ is induced by an irreducible representation of maximal dimension }D\}.$$
  Therefore, if $\rho:\mathcal{A} \to \End(V)$ is a $D$-dimensional central representation inducing $x\in \mathcal{X}$, then $\rho$ is irreducible if and only if $x\in \mathcal{AL}$.
\end{enumerate}
\end{theorem}

\begin{corollary}\label{coro_unicity}(Frohman-Kania Bartoszynska-L\^e \cite{FrohmanKaniaLe_UnicityRep})
The Azumaya locus of $\skein(\Sigma)$ is a Zariski open dense subset.
\end{corollary}

\begin{corollary}\label{coro_roro}
If $\rho: \skein(\Sigma_g) \to \End(V)$ is a central representation with classical shadow $x$ such that $\dim(V)=N^{3g-3}$ when $g\geq 2$ and $\dim(V)=N$ when $g=1$, then 
$$ \rho \mbox{ is irreducible if and only if } x\in \mathcal{AL}.$$
\end{corollary}

Note that these results were extended to stated skein algebras and their reduced versions in \cite{KojuAzumayaSkein}.

\subsection{Character varieties and deformation quantization}

The commutative algebra $\mathcal{S}_{\varepsilon}(\Sigma)$ has a natural Poisson structure defined as follows. Let $\mathcal{S}_{\hbar}(\Sigma)$ be the skein algebra associated to the ring $k=\mathbb{C}[[\hbar]]$ of formal power series in $\hbar$ with parameter $A_{\hbar}:= \varepsilon \exp(\hbar/2)$. Since the set $\mathcal{B}$ of classes of multicurves is a basis for both the algebras $\mathcal{S}_{\hbar}(\Sigma)$ and $\mathcal{S}_{\varepsilon}(\Sigma)$, we can define a $\mathbb{C}[[\hbar]]$-linear isomorphism $\Psi: \mathcal{S}_{\varepsilon}(\Sigma)\otimes_{\mathbb{C}}\mathbb{C}[[\hbar]] \xrightarrow{\cong} \mathcal{S}_{\hbar}(\Sigma)$ characterized by the fact that it sends a basis elements $[\gamma]\in \mathcal{B}$ to itself. Let $\star$ be the pull-back in $\mathcal{S}_{\varepsilon}(\Sigma)\otimes_{\mathbb{C}}\mathbb{C}[[\hbar]]$ of the product in $\mathcal{S}_{\hbar}(\Sigma)$. 

\begin{definition}\label{def_DQSkein} The Poisson bracket $\{\cdot, \cdot\}$ on $\mathcal{S}_{\varepsilon}(\Sigma)$ is defined by the formula
$$ x \star y - y \star x = \hbar \{x, y\} \pmod{\hbar^2}.$$
\end{definition}

In particular, the associativity of $\star$ implies the Jacobi identity for $\{\cdot, \cdot\}$. The Poisson algebra $(\mathcal{S}_{\varepsilon}(\Sigma), \{\cdot, \cdot\})$ admits the following geometric interpretation.

\begin{definition} Suppose that $\Sigma$ is connected and fix a base point $v\in \Sigma$.
\begin{enumerate}
\item The \textit{representation variety} $\mathcal{R}_{\SL_2}(\Sigma)$ is the space of representations $\rho: \pi_1(\Sigma, v) \to \SL_2$ equipped with its natural structure of affine variety. The group $\SL_2$ acts algebraically on $\mathcal{R}_{\SL_2}(\Sigma)$ by conjugacy. The \textit{character variety} is the algebraic (GIT) quotient 
$$\mathcal{X}_{\SL_2}(\Sigma):= \mathcal{R}_{\SL_2}(\Sigma) \sslash \SL_2.$$
Note that $\mathcal{X}_{\SL_2}(\Sigma)$ does not depend on the base point $v$ up to canonical isomorphism.
\item For $\gamma \in \pi_1(\Sigma, v)$, the \textit{trace function} is the regular map $\tau_{\gamma} \in \mathcal{O}[\mathcal{X}_{\SL_2}(\Sigma)]$ defined by 
$$ \tau_{\gamma}([\rho]) := \tr (\rho(\gamma)) \mbox{, for all }\gamma \in \pi_1(\Sigma, v), \rho: \pi_1(\Sigma)\to \SL_2.$$
\item Let $\varphi : \mathbb{C}^* \hookrightarrow \SL_2$ be the diagonal embedding sending $z$ to $\begin{pmatrix} z & 0 \\ 0 & z^{-1} \end{pmatrix}$.
\\ A representation $\rho : \pi_1(\Sigma, v) \to \SL_2$ is said 
\begin{itemize}
\item  \textit{irreducible} if $\rho$ is irreducible; 
\item \textit{central} if it factorizes as $\rho: \pi_1(\Sigma, v) \to \{-1, +1\} \xrightarrow{\subset} \mathbb{C}^* \xrightarrow{\varphi} \SL_2$;
\item \textit{diagonal} if it is not central and it is conjugate to a representation which factorizes as $\pi_1(\Sigma, v) \to \mathbb{C}^* \xrightarrow{\varphi} \SL_2$.
\end{itemize}
For $\Sigma=\Sigma_g$, we denote by $\mathcal{X}^{(0)}$, $\mathcal{X}^{(1)}$, $\mathcal{X}^{(2)}$ the subvarieties of $\mathcal{X}_{\SL_2}(\Sigma_g)$ of classes of representations which are irreducible, diagonal and central respectively, so we get the decomposition
$$ \mathcal{X}_{\SL_2}(\Sigma_g) = \left\{ 
\begin{array}{ll} 
\mathcal{X}^{(0)} \bigsqcup \mathcal{X}^{(1)} \bigsqcup \mathcal{X}^{(2)} 
& \mbox{when }g\geq 2; \\
\mathcal{X}^{(1)} \bigsqcup \mathcal{X}^{(2)} 
& \mbox{when }g=1.
\end{array} \right.$$
\end{enumerate}
\end{definition}
Note that $\mathcal{X}^{(2)} \cong \mathrm{H}^1(\Sigma; \mathbb{Z}/2\mathbb{Z})$ is finite  and that, when $g\geq 2$,  the locus of diagonal representations is defined by the following algebraic equations which say that the trace of the commutator of $\rho(\alpha)$ and $\rho(\beta)$ is $+2$ (see e.g.  \cite{MarcheCharVarSkein} for details): 
$$ \tau_{\alpha}^2 + \tau_{\beta}^2 + \tau_{\alpha\beta}^2 -\tau_{\alpha}\tau_{\beta}\tau_{\alpha\beta} - 4 = 0  \mbox{, for }\alpha, \beta \in \pi_1(\Sigma, v); $$
so  $\mathcal{X}^{(0)}$ is indeed open in the Zariski topology.
For $\rho: \pi_1(\Sigma, v) \to \SL_2$, the stabilizer for the $\SL_2$ action by conjugacy is equal to $\{ \pm \mathds{1}_2\}$, $\varphi(\mathbb{C}^*)$ and $\SL_2$ when $\rho$ is irreducible, diagonal and central respectively. It follows that, when $g\geq 2$,  $\mathcal{X}^{(0)}$ is the smooth locus of $\mathcal{X}_{\SL_2}(\Sigma_g)$ and that $\mathcal{X}^{(1)}$ is the smooth locus of $\mathcal{X}_{\SL_2}(\Sigma_g)\setminus \mathcal{X}^{(0)}$. In the genus one case, $\mathcal{X}^{(0)} =\emptyset$ and $\mathcal{X}^{(1)}$ is the smooth locus of $\mathcal{X}_{\SL_2}(\Sigma_1)$.

\vspace{2mm}\par 
The character variety has a Poisson structure defined by Goldman in \cite{Goldman86} and described as follows. By a theorem of Procesi (\cite{Procesi87}), the algebra of regular functions of the character variety is generated by curve functions, so we need to define the Poisson bracket of two such functions $\tau_{\alpha}$ and $\tau_{\beta}$. Clearly, a trace function $\tau_{\alpha}$ only depends on the isotopy class of $\alpha$ seen as an immersed curve in $\Sigma$. 

 Suppose that $\alpha$ and $\beta$ are two embedded curves which intersect transversally with at most double intersection points. For $v\in \alpha \cap \beta$, let $\alpha_v, \beta_v \in \pi_1(\Sigma, v)$ be representatives of $\alpha$ and $\beta$ and $\rho_v: \pi_1(\Sigma, v) \to \SL_2$ a representative of $[\rho]$. 
  Let $\epsilon(v)\in \{ - 1, +1\}$ be equal to $+1$ if the set $(T_v\alpha, T_v\beta)$ of tangent vectors of $\alpha$ and $\beta$ at $v$ forms an oriented basis of $T_v\Sigma$ and set $\epsilon(v)=-1$ else. 
 Let $\left(\cdot , \cdot \right) : \mathfrak{sl}_2 \otimes \mathfrak{sl}_2 \to \mathbb{C}$ be the non-degenerate Ad-invariant symmetric form normalized so that $(A,B)=\tr(AB)$ for $A,B \in \mathfrak{sl}_2$ (i.e. $\frac{1}{4}$ times the Killing form). For $f\in \mathcal{O}[\SL_2]$ a conjugacy invariant function (for instance $f=\tau$ the trace function), let 
    $X_{f, \alpha}(v) \in \mathfrak{sl}_2$ be the unique vector satisfying 
    $$(X_{f, \alpha}, Y)= \restriction{\frac{d}{dt}}{t=0} f(\rho_v(\alpha_v) e^{tY}),  \mbox{ for all } Y\in \mathfrak{sl}_2.$$
     Extend $(\cdot, \cdot)$ to a bilinear form on $\mathfrak{sl}_2^{\otimes 2}$ by $(x\otimes y, z \otimes t):= (x, z) (y,t)$.

  \begin{definition}\label{GoldmanFormula}(Goldman \cite{Goldman86}) The Poisson bracket $\{\cdot, \cdot \}$ on $\mathcal{O}[\mathcal{X}_{\SL_2}(\mathbf{\Sigma})]$ is defined by the following formula
  \begin{eqnarray*} \{ \tau_{\alpha} , \tau_{\beta} \}([\rho])  &=&  2\sum_{v\in c_1\cap c_2} \epsilon(v) \left( X_{\tau, \alpha}(v), X_{\tau, \beta}(v) \right).
 \end{eqnarray*}
  \end{definition}
  
  Goldman gave in \cite{Goldman86} an explicit formula which shows that $\{\tau_{\alpha}, \tau_{\beta}\}$ is a linear combination of trace functions, so the Poisson bracket of two regular functions is a regular function. Goldman original work concerns closed surfaces, though it was extended to open surfaces as well in \cite{Lawton_PoissonGeomSL3, KojuTriangularCharVar}.
   When $\Sigma$ is closed, it is proved in \cite{Goldman_symplectic, Goldman86} that the smooth locus of the character variety is isomorphic to the smooth part of the moduli space of flat $\SL_2$ structures on $\Sigma$ equipped with its Atiyah-Bott symplectic structure defined in \cite{AB}. 
   In particular, one has:
  
  \begin{theorem}\label{theorem_smooth_symplectic}(Goldman \cite{Goldman_symplectic, Goldman86}) When $\Sigma$ is closed, the smooth locus of $\mathcal{X}_{\SL_2}(\Sigma)$  is symplectic. \end{theorem}.
  
  \begin{lemma}\label{lemma_X1} When $\Sigma$ is closed, the diagonal locus $\mathcal{X}^{(1)}$ is symplectic. \end{lemma}

\begin{proof}
Consider the symplectic torus $\mathcal{X}_{\mathbb{C}^*}(\Sigma)$ defined as follows. As an affine variety $$\mathcal{X}_{\mathbb{C}^*}(\Sigma):= \Specm \left( \mathbb{C}[\mathrm{H}_1(\Sigma; \mathbb{Z})] \right).$$ 
 Here $\mathbb{C}[\mathrm{H}_1(\Sigma; \mathbb{Z})] $ denotes the group algebra
with  basis consisting of $X_{[\gamma]}$ for $[\gamma] \in  \mathrm{H}_1(\Sigma; \mathbb{Z})$ and relations $X_{[\gamma_1]} X_{[\gamma_2]} = X_{[\gamma_1]+[\gamma_2]}$.
The symplectic structure is given by the formula $$\{X_{[\alpha]}, X_{[\beta]} \} := ([\alpha], [\beta]) X_{[\alpha+\beta]},$$ where $(\cdot, \cdot)$ is the intersection form. Since $\Sigma$ is closed, the intersection form has trivial kernel so the Poisson structure is indeed symplectic. Note that the set of closed points of $\mathcal{X}_{\mathbb{C}^*}(\Sigma)$ is identified with $\mathrm{H}^1(\Sigma; \mathbb{C}^*)$. The diagonal inclusion $\varphi: \mathbb{C}^* \hookrightarrow \SL_2$ induces a regular morphism $\Phi: \mathcal{X}_{\mathbb{C}^*}(\Sigma) \to \mathcal{X}_{\SL_2}(\Sigma)$ sending $\chi: \mathrm{H}_1(\Sigma; \mathbb{Z}) \to \mathbb{C}^*$ to $\Phi(\chi):= [\varphi \circ \chi]$. Said differently, $\Phi^*: \mathcal{O}[\mathcal{X}_{\SL_2}(\Sigma)] \to \mathbb{C}[\mathrm{H}_1(\Sigma; \mathbb{Z})] $ sends $\tau_{\gamma}$ to $X_{[\gamma]}+X_{-[\gamma]}$. Clearly $\Phi: \mathcal{X}_{\mathbb{C}^*}(\Sigma) \to \mathcal{X}^{(1)}\cup \mathcal{X}^{(2)}$ is a double branched covering branched along $\mathcal{X}^{(2)}$.
\par Let us prove that $\Phi$ is Poisson. Let $[\rho]\in \mathcal{X}^{(1)}\cup \mathcal{X}^{(2)}$, consider two simple closed curves $\alpha, \beta$ which intersects transversally as before and let $v\in \alpha \cap \beta$. Let 
 $\rho_v : \pi_1(\Sigma, v) \xrightarrow{\lambda} \mathbb{C}^* \xrightarrow{\varphi} \SL_2$ be a diagonal representation representing the class $[\rho]$. For $Y= \begin{pmatrix}a& b\\ c & -a\end{pmatrix} \in \mathfrak{sl}_2$, then 
 $$ \tau (\rho_v(\alpha_v)Y)= \Tr \left( \begin{pmatrix} \lambda(\alpha_v) & 0 \\ 0 & \lambda(\alpha_v)^{-1} \end{pmatrix} \begin{pmatrix} a & b\\ c & -a \end{pmatrix}\right) = (\lambda(\alpha)-\lambda(\alpha)^{-1})a =(\lambda(\alpha)-\lambda(\alpha)^{-1}) (\frac{H}{2}, Y).$$
 Therefore $X_{\tau, \alpha}(v)= (\lambda(\alpha)-\lambda(\alpha)^{-1}) \frac{H}{2}$ so 
 $$ (X_{\tau, \alpha}(v), X_{\tau, \beta}(v) ) = (\lambda(\alpha)-\lambda(\alpha)^{-1})(\lambda(\beta)-\lambda(\beta)^{-1}) \frac{(H,H)}{4} = \frac{1}{2}(\tau_{[\alpha +\beta]} - \tau_{[\alpha - \beta]})(\rho_v).$$
 So Goldman's formula in Definition \ref{GoldmanFormula} reads:
$$ \{ \tau_{[\alpha]} , \tau_{[\beta]} \} ([\rho])=  \left( \sum_{v \in \alpha \cap \beta} \epsilon(v) \right)( \tau_{[\alpha+ \beta]}-\tau_{[\alpha-\beta]})([\rho])= ([\alpha], [\beta]) (\tau_{[\alpha+\beta]}-\tau_{[\alpha-\beta]})([\rho]).$$
On the other hand, one has 
\begin{multline*}
 \{ \Phi^* ( \tau_{\alpha}), \Phi^{*}(\tau_{\beta}) \}= \{ X_{[\alpha]} + X_{[\alpha]}^{-1}, X_{[\beta]} + X_{[\beta]}^{-1}\} = ([\alpha], [\beta]) (X_{[\alpha +\beta]} + X_{[\alpha +\beta]}^{-1} - X_{[\alpha- \beta]} - X_{[\alpha - \beta]}^{-1}) \\ = ([\alpha], [\beta]) (\Phi^*(\tau_{[\alpha+\beta]} + \tau_{[\alpha - \beta]})).\end{multline*}
Therefore $\{ \Phi^*(\tau_{[\alpha]}), \Phi^*(\tau_{[\beta]}) \}= \Phi^* ( \{\tau_{[\alpha]}, \tau_{[\beta]} \} )$ so $\Phi^*$ is Poisson. Since $\Phi: \mathcal{X}_{\mathbb{C}^*}(\Sigma) \to \mathcal{X}^{(1)}\cup \mathcal{X}^{(2)}$ is a double branched covering, regular at  $\mathcal{X}^{(1)}$, the fact that $\mathcal{X}_{\mathbb{C}^*}(\Sigma)$ is symplectic implies that $\mathcal{X}^{(1)}$ is symplectic as well.
\end{proof}

\begin{theorem}\label{theorem_skein+1}(\cite{Bullock, PS00, Barett, Turaev91})
Let $S$ be a spin structure on $\Sigma$ and denote by $w_S: \mathrm{H}_1(\Sigma; \mathbb{Z}/2\mathbb{Z}) \to \mathbb{Z}/2\mathbb{Z}$ its associated Johnson quadratic form. There exists a Poisson isomorphisms
$$\Psi : (\mathcal{S}_{-1}(\Sigma), \{\cdot, \cdot \} ) \xrightarrow{\cong} (\mathcal{O}[\mathcal{X}_{\SL_2}(\Sigma)], \{ \cdot, \cdot \} )
\  \mbox{ and } \ 
\Psi^S : (\mathcal{S}_{+1}(\Sigma), \{\cdot, \cdot \} ) \xrightarrow{\cong} (\mathcal{O}[\mathcal{X}_{\SL_2}(\Sigma)], \{ \cdot, \cdot \} )$$
characterized by 
$$ \Psi([\gamma]):= -\tau_{\gamma}, \quad \Psi^S([\gamma]):= (-1)^{w_S([\gamma])+1} \tau_{\gamma}.$$
\end{theorem}

From now on, we fix once and for all a spin structure on $\Sigma$ if $\varepsilon=+1$  and identify $\Specm( \mathcal{S}_{\varepsilon}(\Sigma))$ with $\mathcal{X}_{\SL_2}(\Sigma)$. In particular, we consider classical shadows as closed points in $\mathcal{X}_{\SL_2}(\Sigma)$.

\subsection{Poisson orders}\label{sec_PO}

The theory of Poisson orders was created by DeConcini-Kac in their original work on quantum  algebras \cite{DeConciniKacRepQGroups} and generalized by Brown-Gordon in \cite{BrownGordon_PO}.

\begin{definition}
A \textit{Poisson order} is a $4$-tuple $(\mathcal{A}, {X}, \phi, D)$ where: 
\begin{enumerate}
\item $\mathcal{A}$ is an (associative, unital) affine prime $\mathbb{C}$-algebra finitely generated over its center $\mathcal{Z}$; 
\item ${X}$ is a Poisson affine $\mathbb{C}$-variety; 
\item $\phi: \mathcal{O}[{X}] \hookrightarrow \mathcal{Z}$ a finite injective morphism of algebras; 
\item $D: \mathcal{O}[{X}] \to Der (\mathcal{A}) : z\mapsto D_z$ a linear map such that for all $f,g \in  \mathcal{O}[{X}]$, we have
$$ D_f(\phi(g))= \phi(\{f,g\}).$$
\end{enumerate}
\end{definition}

Let $X$ be an affine complex Poisson variety. 
\begin{itemize}
\item Define a first partition $X=X^0\sqcup \ldots \sqcup X^n$ where $X^0$ is the smooth locus of $X$ and for $i=0, \ldots, n-1$, $X^{i+1}$ is the smooth locus of $X\setminus (X^0 \sqcup \ldots \sqcup X^{i})$. Then each $X^i$ is a smooth complex affine variety that can be seen as an analytic Poisson variety. 
\item  Define an equivalence relation $\sim$ on $X^i$ by 
  writing $x\sim y$ if there exists a finite sequence $x=p_0, p_1, \ldots, p_k=y$ and functions $h_0, \ldots, h_{k-1} \in \mathcal{O}[X^i]$ such that $p_{i+1}$ is obtained from $p_i$ by deriving along the Hamiltonian flow of $h_i$. Write $X^i= \bigsqcup_j X^{i,j}$ the orbits of this relation.   
  \end{itemize}
  
  \begin{definition} 
  The elements $X^{i,j}$ of the (analytic) partition $X= \bigsqcup_{i,j} X^{i,j}$ are called the \textit{symplectic leaves} of $X$.
  \end{definition}

The symplectic leaves are analytic subsets. One can also define a partition into algebraic subsets as follows. An ideal $\mathcal{I}\subset \mathcal{O}[X]$ is a \textit{Poisson ideal} if $\{ \mathcal{I}, \mathcal{O}[{X}] \} \subset \mathcal{I}$. Since the sum of two Poisson ideals is a Poisson ideal, every maximal ideal $\mathfrak{m}\subset \mathcal{O}[{X}]$ contains a unique maximal Poisson ideal $P(\mathfrak{m}) \subset \mathfrak{m}$. Define an equivalence relation $\sim$ on ${X}=\mathrm{Specm}(\mathcal{O}[{X}])$ by $\mathfrak{m} \sim \mathfrak{m}'$ if $P(\mathfrak{m})=P(\mathfrak{m}')$. The equivalence class of $\mathfrak{m}$ is called its \textit{symplectic core} and is denoted by $C(\mathfrak{m})$.
 
 \begin{definition} The partition ${X}=\bigsqcup_{\mathfrak{m}\in {X/\sim}} C(\mathfrak{m})$ is called the \textit{symplectic cores partition} of ${X}$.
 \end{definition} 
 
If $x, y\in {X}$ belong to the same smooth strata of ${X}$ and are such that $y$ can be obtained from $x$ by deriving along some Hamiltonian flow, Brown and Gordon proved that $C(x)=C(y)$ from which they deduce the

\begin{proposition}(Brown-Gordon \cite[Proposition $3.6$]{BrownGordon_PO}) The symplectic leaves partition is a refinement of the symplectic cores partition.
\end{proposition}

 If $X^{i,j}$ is a symplectic leaf contained in a symplectic core $C$, then $C$ is the smallest algebraic subset containing $X^{i,j}$.

\begin{theorem}\label{theorem_PO}(Brown-Gordon \cite[Theorem $4.2$]{BrownGordon_PO}) Let $(\mathcal{A}, {X}, \phi, D)$ be a Poisson order. If $x,y\in{X}$ belong to the same symplectic core, then $\mathcal{A}_x\cong \mathcal{A}_y$.
\end{theorem}

Note that Theorems \ref{theorem_AL_dense} and \ref{theorem_PO} imply that if $X$ contains a symplectic leaf (or core) which is dense, then it is included into the Azumaya locus. 
\par Here is our main source of examples of Poisson orders; we closely follow \cite{BrownGordon_PO} to which we refer for further details. 
  Let $\mathcal{A}_q$ be a free affine $\mathbb{C}[q^{\pm 1}]$-algebra which is a domain. Let   $N\geq 1$,  consider  the $k_N:= \quotient{\mathbb{C}[q^{\pm 1}]}{(q^N-1)}$ algebra $\mathcal{A}_N:= \mathcal{A}_q \otimes_{\mathbb{C}[q^{\pm 1}]} k_N$ and denote by $\pi: \mathcal{A}_q \to \mathcal{A}_N$ the quotient map. Fix a basis $\mathcal{B}$ of the $\mathbb{C}[q^{\pm 1}]$ module $\mathcal{A}_q$. Then  $\mathcal{A}_N$ has basis $\{q^i \pi(b), b \in \mathcal{B}, i=0, \ldots, N-1\}$ over $\mathbb{C}$. 
Define a $\mathbb{C}$ linear embedding $\hat{\cdot}: 
\mathcal{A}_{N} \to \mathcal{A}_q$ sending a basis element $x= q^i \pi(b) \in \mathcal{A}_{N}$ to the  element $\hat{x}:=q^i b \in \mathcal{A}_q$. Note that $\hat{\cdot}$ is a left inverse for $\pi$. Suppose that the  algebra $\mathcal{A}_{+1}=\mathcal{A}_q\otimes_{q=1}\mathbb{C}$ is commutative and suppose there exists a central embedding  $\phi_N: \mathcal{A}_{+1} \hookrightarrow \mathcal{Z}(\mathcal{A}_{N})$ into the center of $\mathcal{A}_{N}$ such that $\mathcal{A}_N$ is finitely generated as an $\mathcal{A}_{+1}$-module. The pair $(\mathcal{A}_q, \phi_N)$ will be referred to as a \textit{quantum algebra}.
 Write $\mathcal{X}:=\mathrm{Specm}(\mathcal{A}_{+1})$ and define $D: \mathcal{A}_{+1} \to Der(\mathcal{A}_N)$ by the formula
 $$ D_xy:= \pi \left( \frac{ [\widehat{\phi_N(x)}, \hat{y}]}{N(q^N-1)} \right) .$$
Clearly $D_x$ is a derivation, is independent on the choice of the basis $\mathcal{B}$ and preserves $\phi_N(\mathcal{A}_{+1})$, so it defines a Poisson bracket $\{\cdot, \cdot \}_N$ on $\mathcal{A}_{+1}$ by 
\begin{equation}\label{eq_PoissonOrder}
 D_x \phi_N(y)= \phi_N (\{x,y\}_N).
 \end{equation}
So  $\mathcal{X}=\mathrm{Specm}(\mathcal{A}_{+1})$ is a Poisson variety. Let $\zeta$ be a root of unity of order $N$, write $\mathcal{A}_{\zeta}=\mathcal{A}_q \otimes_{q=\zeta} \mathbb{C}$ and consider $D: \mathcal{A}_{+1} \to Der(\mathcal{A}_{\zeta})$ and $\phi_{\zeta}$ obtained by tensoring by $\mathbb{C}$. 
Then $(\mathcal{A}_{\zeta}, \mathcal{X}, \phi_{\zeta}, D)$ is a Poisson order. Note that $\mathcal{X}$ also admits a Poisson structure coming from deformation quantization, namely setting $\mathcal{A}_{\hbar}=\mathcal{A}_q\otimes_{q=\exp(\hbar)}\mathbb{C}[[\hbar]]$ with product $\star_{\hbar}$, we also have a bracket $\{\cdot, \cdot\}$ by the formula
$$ a \star b - b \star a \equiv \{a,b\} \hbar \pmod{\hbar^2}.$$
It is not clear that the two brackets $\{\cdot, \cdot\}$ and $\{\cdot, \cdot\}_N$ coincide in general. When the two brackets coincide, we say that the quantum algebra $(\mathcal{A}_q, \phi_N)$ is \textit{regular}.
Let $\mathcal{S}_q(\Sigma)$ be the skein algebra over the ring $\mathbb{C}[q^{\pm 1}]$ and $Ch_N: \gamma \mapsto T_N(\gamma)$.
Clearly $(\mathcal{S}_q(\Sigma), Ch_N)$ is a quantum algebra so  we have a Poisson order $(\mathcal{S}_A(\Sigma), \mathcal{X}_{\SL_2}(\Sigma), Ch_A, D)$. 

\begin{lemma}(\cite[Lemma $4.6$]{KojuRIMS}) \label{lemma_PB} $(\mathcal{S}_q(\Sigma), Ch_A)$ is regular, i.e. 
the Poisson brackets $\{\cdot, \cdot\}$ and $\{\cdot, \cdot\}_N$ on $\mathcal{S}_{\varepsilon}(\Sigma)$  coming respectively from deformation quantization (Definition \ref{def_DQSkein}) and from the Poisson order 
structure coincide.
\end{lemma}

Since Lemma \ref{lemma_PB} is published in a non peer-reviewed journal, we add its proof for the reader convenience. 
\begin{proof}
We suppose $\varepsilon=+1$, the case $\varepsilon=-1$ is similar.
We first prove the analogue of Lemma \ref{lemma_PB} in the easier case of a quantum torus. Let $E$ be a free $\mathbb{Z}$-module of finite rank (so $E\cong \mathbb{Z}^n$) and $(\cdot, \cdot): E\otimes E \to \mathbb{Z}$ a skew-symmetric pairing. The associated \textit{quantum torus} $\mathbb{T}_q$ is the $\mathbb{C}[q^{\pm 1}]$-algebra equal to  $\mathbb{C}[q^{\pm 1}][Z_e, e\in E]$ as a module with product $Z_a Z_b = q^{(a,b)}Z_{a+b}$. Clearly $\mathbb{T}_q$ is free with basis $\{Z_e, e\in E\}$, is a domain and is finitely generated (by the set $Z_e$ for $e$ in a generating set of $E$). 
When $\zeta$ is a root of unity of order $N$, the Frobenius morphism $Fr_N: \mathbb{T}_{+1} \hookrightarrow \mathcal{Z}(\mathbb{T}_{\zeta})$ is defined by $Fr_N(Z_e)= (Z_e)^N (= Z_{Ne})$. 
By the preceding discussion we obtain a Poisson order $(\mathbb{T}_{\zeta}, \mathcal{X}, Fr_N, D)$ where $\mathcal{X}= \Specm(\mathbb{T}_{+1}) \cong (\mathbb{C}^*)^n$ is equipped with the Poisson bracket $\{\cdot, \cdot\}_N$. On the one hand, the Poisson bracket coming from deformation quantization is computed via: 
$$ Z_a \star Z_b - Z_b \star Z_a = (q^{(a,b)}-q^{-(a,b)})Z_{a+b}= \hbar 2(a,b) Z_{a+b} \pmod{\hbar^2}, $$
  giving the Poisson bracket $\{ Z_a, Z_b\}= 2(a,b) Z_{a+b}$. On the other hand, the Poisson bracket coming from the Poisson order structure is computed via:
  \begin{align*}
& \pi \left( \frac{ [Fr_N(Z_a), Fr_N(Z_b)]}{N(q^N-1)} \right) =  \pi \left( \frac{ [Z_a^N, Z_b^N]}{N(q^N-1)} \right) \\
&= \pi \left( \frac{q^{N^2(a,b)} - q^{-N^2(a,b)}}{N(q^N-1)} \right) Z_{a+b}^N \\
&= \pi \left( q^{-N^2(a,b)}\frac{1}{N} \sum_{i=0}^{2N(a,b)-1} q^{Ni} \right) Fr_N(Z_{a+b})
= Fr_N(2(a,b)Z_{a+b}).
\end{align*}
So $\{Z_a, Z_b\}_N= 2(a,b)Z_{a+b}$ and the Poisson bracket $\{\cdot, \cdot \}_N$ coincides with the bracket coming from deformation quantization and $(\mathbb{T}_q, Fr_N)$ is regular. 
\par A morphism of quantum algebras $f : (\mathcal{A}_q, \phi_N)\to (\mathcal{B}_q, \Psi_N)$ is an algebra morphism $f: \mathcal{A}_q \to \mathcal{B}_q$ such that the induced morphisms $f_N: \mathcal{A}_{N} \to \mathcal{B}_N$ and $f_{+1}: \mathcal{A}_{+1} \to \mathcal{B}_{+1}$ satisfy $\Psi_N\circ f_{+1}=f_N\circ \phi_N$. Clearly $f_{+1}$ is a Poisson morphism for both the deformation quantization and the Poisson order Poisson brackets. In particular if $f$ is injective, then $ (\mathcal{B}_q, \Psi_N)$ is regular implies that $ (\mathcal{A}_q, \phi_N)$ is regular as well. For an open surface $\Sigma_{g,n}$ with $n\geq 1$, Bonahon and Wong defined in \cite{BonahonWongqTrace} an injective morphism of quantum algebras $\Tr_q^{\Delta}: (\mathcal{S}_q(\Sigma_{g,n}), Ch_N) \hookrightarrow (\mathbb{T}_q, Fr_N)$, named the quantum trace, from the skein algebra to a quantum torus named the balanced Chekhov-Fock algebra so $(\mathcal{S}_q(\Sigma_{g,n}), Ch_N)$ is regular. 
\par Similarly, if $p:  (\mathcal{A}_q, \phi_N)\twoheadrightarrow (\mathcal{B}_q, \Psi_N)$ is a surjective morphism, then $(\mathcal{A}_q, \phi_N)$ is regular implies that $(\mathcal{B}_q, \Psi_N)$ is regular as well. The inclusion $\iota: \Sigma_{g,1}\hookrightarrow \Sigma_g$ induces a surjective morphism $\iota_*: (\mathcal{S}_q(\Sigma_{g,1}), Ch_N) \to (\mathcal{S}_q(\Sigma_g), Ch_N)$ (sending $\gamma$ to $\iota(\gamma)$), so $(\mathcal{S}_q(\Sigma_g), Ch_N)$ is regular.
\end{proof}

\begin{remark} The above proof generalizes word-by-word to prove that the stated skein algebras and the reduced stated skein algebras define regular quantum algebras as well. Indeed, stated skein algebras can be embedded into reduced stated skein algebra (see \cite{LeYu_SSkeinQTraces}) and Bonahon-Wong quantum trace embeds reduced stated skein algebras into quantum tori as well (see \cite{LeStatedSkein, CostantinoLe19}). 
\end{remark}

By Theorem \ref{theorem_smooth_symplectic} and Lemma \ref{lemma_X1}, both $\mathcal{X}^{(0)}$ and $\mathcal{X}^{(1)}$ are symplectic leaves of $\mathcal{X}_{\SL_2}(\Sigma)$. Since the smooth locus is dense, Corollary \ref{coro_unicity} and Theorem \ref{theorem_PO} imply the 

\begin{corollary}\label{coro_GJS}(Ganev-Jordan-Safronov \cite[Theorem $1.1$]{GanevJordanSafranov_FrobeniusMorphism}) The smooth locus of $\mathcal{X}_{\SL_2}(\Sigma)$ is included into the Azumaya locus of $\skein(\Sigma)$.
\end{corollary}

\begin{corollary}\label{coro_X1} For $g\geq 2$, if the Azumaya locus of  $\skein(\Sigma)$ intersects $\mathcal{X}^{(1)}$ non-trivially, then it contains $\mathcal{X}^{(1)}$.
\end{corollary}

\subsection{Witten-Reshetikhin-Turaev representations}\label{sec_WRT}

The Witten-Reshetikhin-Turaev TQFTs from \cite{Wi2, RT} provide representations of Kauffman-bracket skein algebras at roots of unity. Let us describe a simple construction following \cite{BHMV2}. 
  Let $H_g$ be a genus $g$ handlebody and consider the skein module $\mathcal{S}_A(H_g)$. By identifying a tubular neighborhood $\mathcal{N}$ of $\partial H_g$ with $\Sigma_g \times [0,1]$, we get an oriented embedding $i : \Sigma_g\times [0,1] \to H_g$. We  consider $\mathcal{S}_A(H_g)$ as a $\mathcal{S}_A(\Sigma_g)$-left module by defining for two framed links $L_1 \subset \Sigma_g\times [0,1]$ and $L_2\subset H_g$ the module product $[L_1]\cdot [L_2] := [i(L_1)\cup L_2]$, where $L_2$ has been first isotoped outside $\mathcal{N}$.
  Fix a genus $g$ Heegaard splitting of the sphere, i.e. two oriented homeomorphisms  $S: \partial H_g \cong \overline{\partial H_g}$ and 
$ H_g\bigcup_{S:\partial H_g\rightarrow \overline{\partial H_g}} H_g \cong S^3$. Denote by $\varphi_1, \varphi_2 : H_g \hookrightarrow S^3$ the embeddings into the first and second factors.
\par For two framed links $L_1,L_2 \subset H_g$ , the above gluing defines a link $\varphi_1(L_1)\bigcup \varphi_2(L_2) \subset S^3$. Let $\left< \cdot \right> : \mathcal{S}_A(S^3) \cong \mathbb{C}$ be the isomorphism sending the empty link to $1$. The Hopf pairing is the Hermitian form:
$$ \left( \cdot , \cdot \right)_A^H :\mathcal{S}_A(H_g)\times\mathcal{S}_A(H_g) \rightarrow \mathbb{C}$$
defined by $$\left( L_1 , L_2 \right)_{A}^H := \left< \varphi_1(L_1)\bigcup \varphi_2(L_2) \right>$$
and the spaces $V_A(\Sigma_g)$ are the quotients:
$$ V_A(\Sigma_g) := \quotient{ \mathcal{S}_A(H_g)}{\Rker \left(\left( \cdot , \cdot \right)_{A}^H \right)}$$
  The structure of left $\mathcal{S}_A(\Sigma_g)$-module of $\mathcal{S}_A(H_g)$ induces, by passing to the quotient, the so-called \textit{Witten-Reshetikhin-Turaev representation}
  $$ r^{WRT} : \mathcal{S}_A(\Sigma_g) \to \End(V_A(\Sigma_g)).$$

  The following theorem was stated for $A$ a root of unity of even order, though as detailed in \cite[Theorem $5.7$]{KojuSurvey}, the proofs extend straightforwardly to the odd case.
  \begin{theorem}\label{theorem_WRT}
  \begin{enumerate}
  \item (Reshetikhin-Turaev \cite{RT, Tu}, see also \cite{BHMV2}) The space $V_A(\Sigma_g)$ are non-zero and finite dimensional with dimension given by the Verlinde formula. In particular $\dim (V_A(\Sigma_g)) < PI-deg (\skein(\Sigma_g))$.
  \item (Gelca-Uribe \cite[Theorem $6.6$]{GelcaUribe_SU2}, see also \cite{BonahonWong4}) The representation $r^{WRT}$ is irreducible.
  \item (Bonahon-Wong \cite{BonahonWong4}) The classical shadow of $r^{WRT}$ is a central representation.
  \end{enumerate}
  \end{theorem}
  
  "Which central representation is the classical shadow of $r^{WRT}$" depends on the choice of spin structure used in the identification of Theorem \ref{theorem_skein+1}. This knowledge is irrelevant thanks to the
  
  \begin{lemma}\label{lemma_X2}
  If $[\rho_1]$, $[\rho_2]$ are both in $\mathcal{X}^{(2)}$ then $\left( \skein(\Sigma) \right)_{[\rho_1]} \cong \left( \skein(\Sigma) \right)_{[\rho_2]}$. 
  \end{lemma}
  
  \begin{proof} The group $ \mathrm{H}^1(\Sigma; \mathbb{Z}/2\mathbb{Z})$ acts by automorphisms on $\skein(\Sigma)$ by the formula defined on a multi-curve $\gamma=\gamma^{(1)}\cup \ldots \cup \gamma^{(n)}$ with connected components $\gamma^{(i)}$ and $\omega \in \mathrm{H}^1(\Sigma; \mathbb{Z}/2\mathbb{Z})$  by the formula
  $$\omega_* [\gamma] := (-1)^{\sum_i \omega([\gamma^{(i)}])} [\gamma].$$
  By Theorem \ref{theorem_chebyshev}, this action restricts to an action on the center of $\skein(\Sigma)$ which, once identified with $\mathcal{O}[\mathcal{X}_{\SL_2}(\Sigma)]$, corresponds to the algebraic action of $ \mathrm{H}^1(\Sigma; \mathbb{Z}/2\mathbb{Z})$  on $\mathcal{X}_{\SL_2}(\Sigma)$ defined by 
  $$ (\omega \cdot \rho) (\gamma) := (-1)^{\omega([\gamma])} \rho (\gamma).$$
  Therefore, the automorphism $\omega_* \in \Aut(\skein(\Sigma))$ induces an isomorphism  $\left( \skein(\Sigma) \right)_{[\rho]} \cong \left( \skein(\Sigma) \right)_{\omega \cdot [\rho]}$.
We conclude using the fact that  the action of $ \mathrm{H}^1(\Sigma; \mathbb{Z}/2\mathbb{Z})$  on $\mathcal{X}^{(2)}$ is transitive.
  \end{proof}
  
  Theorem \ref{theorem_WRT} implies that the classical shadow of $r^{WRT}$ is not in the Azumaya locus. Since it belongs to $\mathcal{X}^{(2)}$, Lemma \ref{lemma_X2} implies the
  
  \begin{corollary}\label{coro_X2} The Azumaya locus of $\skein(\Sigma)$ does not intersect $\mathcal{X}^{(2)}$.
  \end{corollary}
Together with Corollary \ref{coro_GJS}, this implies the

\begin{corollary}\label{coro_genusOne} The Azumaya locus of $\mathcal{S}_A(\Sigma_1)$ is the locus of non-central representations. \end{corollary}

\section{Categories of representations of the unrolled quantum groups}\label{sec_category}

The non-semi extended TQFTs from \cite{BCGPTQFT} are defined using some $\mathbb{C}$-modular relative categories as described by De Renzi in \cite{DeRenzi_NSETQFT} which are defined as representation categories of (unrolled) quantum groups  at root of unity. In this paper, we are interested in the category of representations of $U^H_q\mathfrak{sl}_2$ when $q:=A^2$ has odd order $N$. The  category of representations of $U_q^H\mathfrak{sl}_2$ is studied in great details in \cite{CGP_unrolledQG} when $q$ has even order whereas the category of representations of $U^H_q\mathfrak{g}$ for any simple $\mathfrak{g}$ at odd roots of unity is studied in \cite{DeRenziGeerPatureau_TQFT_QG}. Our case of interest ($U^H_q\mathfrak{sl}_2$ at odd roots of unity) is much simpler than these two cases though many computations (such as the $6j$ symbols) have never been done in this case.
Since we will need to perform precise computations in this category, we now describe it completely.

\subsection{Unrolled quantum groups and their representations}

Throughout all this section, $q:=A^2$ is a root of unity of odd order $N\geq 3$, so $q=\exp( \frac{2i \pi k}{N})$ for some $k\in \{1, \ldots, N-1\}$ prime to $N$. 
 For $z\in \mathbb{C}$, we write $q^z:= \exp(\frac{2i \pi k z}{N})$ and $A=q^{1/2}$. Set also $\{ z \} := q^z - q^{-z}$ and $[z]:= \frac{q^z -q^{-z}}{q- q^{-1}}$.

\begin{definition} 
The \textit{unrolled quantum group} $U_q^H \mathfrak{sl}_2$ is the algebra generated by $E,F,K^{\pm 1}$ and $H$ with relations:
\begin{align*}
&  KE=q^2EK, \quad KF=q^{-2}FK, \quad [E,F]=\frac{K-K^{-1}}{q-q^{-1}} \\
& [H,E]=2E, \quad [H,F]=-2F, \quad [H,K]=0.
\end{align*}
It has a Hopf algebra structure with coproduct:
$$
\Delta(E)=1\otimes E + E\otimes K, \quad \Delta(F)= K^{-1}\otimes F + F\otimes 1, \quad \Delta(K)=K\otimes K, \quad \Delta(H)= H\otimes 1 + 1\otimes H, $$
with counit:
$$ \epsilon(E)=\epsilon(F)=\epsilon(H)=0, \quad \epsilon(K)=1, $$
and antipode:
$$ S(E)= -EK^{-1}, \quad S(F)=-KF, \quad S(K)=K^{-1}, \quad S(H)=-H.$$

We denote by $\mathcal{C}$ the monoidal category of finite dimensional  representations  $\rho: U_q^H\mathfrak{sl}_2 \to \End(V)$ such that: 
\begin{enumerate}
\item $\rho(E^N)=\rho(F^N)=0$, $\rho(K^N)$ is semisimple and 
\item if $\rho(H)\cdot v = \lambda v$ for $v\in V$, then $\rho(K)\cdot v = q^{\lambda}v$.
\end{enumerate}
\end{definition}

 So $V$ is a direct sum of submodules on which $\rho(K^N)$ is scalar; in such a submodule, the minimal polynomial of $\rho(K)$ divides $X^N-z$ for $z\in \mathbb{C}^*$, so its zeroes have multiplicity one. This implies that 
 $\rho(K)$ is diagonalizable, so $V$ is spanned by vectors $v\in V$ such that $Hv=\lambda v$; we call such  $v$ a \textit{weight vector of weight} $\lambda$. For $\overline{\alpha} \in \mathbb{C}/\mathbb{Z}$, we say that $V$ \textit{has weight } $\overline{\alpha}$ if it is spanned by weight vectors with weights in $\overline{\alpha}$.
The category $\mathcal{C}$ is $\mathbb{C}/\mathbb{Z}$-graded $\mathcal{C}=\bigsqcup_{\overline{\alpha}\in \mathbb{C}/\mathbb{Z}} \mathcal{C}_{\overline{\alpha}}$ where $\mathcal{C}_{\overline{\alpha}}$ is the full subcategory of modules with weights $\alpha \in \overline{\alpha}$. Since $H$ is primitive, one has $\mathcal{C}_{\overline{\alpha}} \otimes \mathcal{C}_{\overline{\beta}} \subset \mathcal{C}_{\overline{\alpha + \beta}}$. If $V\in \mathcal{C}_{\overline{\alpha}}$, $W\in \mathcal{C}_{\overline{\beta}}$ with $\overline{\alpha}\neq \overline{\beta}$, then clearly $\Hom_{\mathcal{C}}(V,W)=0$.

\begin{definition}
\begin{enumerate}
\item For $\alpha \in \mathbb{C}$,  the module $V_{\alpha} \in \mathcal{C}_{\overline{\alpha}}$ has the basis $\{v_0, \ldots, v_{N-1}\}$ and module structure
  $$ Hv_n = (\alpha + N-1 -2n)v_n \quad, Fv_n=v_{n+1}, Fv_{N-1}=0 \quad, Ev_n=[n][\alpha - n] v_{n-1}, Ev_0=0.$$
\item For $n\in \{0, \ldots, N-1\}$, the module $S_n\in \mathcal{C}_{\overline{0}}$ is the $n+1$ dimensional module with the basis $\{e_0, \ldots, e_n\}$ and module structure
$$ He_i= (n-2i)e_i, \quad Fe_i = e_{i+1}, Fe_n=0, \quad Ee_i= [i][n-i+1]e_{i-1}, Ee_0=0.$$
\item For $n\in \{0, \ldots, N-2\}$, the module $P_n \in \mathcal{C}_{\overline{0}}$ is the $2N$ dimensional module with the basis $\{x_i, y_i, i=0, \ldots N-1\}$ and module structure 
\begin{align*}
{} & Hx_i = (2N-2-n-2i)x_i, \quad H y_i = (n-2i) y_i, \quad Fx_i=x_{i+1}, \quad F y_i = y_{i+1}, \quad (\mbox{here }x_N=y_N=0 );
\\
{} & E x_i = -[i][i+1+n] x_{i-1}, \quad Ey_0= x_{N-2-n}, \quad E y_{n+1}=x_{N-1}, \quad Ey_i = [i][n+1+i]y_{i-1} \mbox{ for } i\neq 0, n+1.
\end{align*}
\item For $n\in \mathbb{Z}$, the module $\sigma^{n}$ is the one-dimensional module $\mathbb{C}v$ such that $Ev=Fv=0, Kv=v$ and $Hv= n\frac{N}{k}v$.
\end{enumerate}
\end{definition}

The functor sending $V$ to $V\otimes \sigma^{\otimes n}$ induces an equivalence between $\mathcal{C}_{\overline{\alpha}}$ and $\mathcal{C}_{\overline{\alpha + n\frac{N}{k}}}$. Let $\mathcal{P}\subset \mathbb{C}/\mathbb{Z}$ be the class of $\frac{N}{k}\mathbb{Z}$, so $\mathcal{C}_{\overline{\alpha}}\cong \mathcal{C}_{\overline{\alpha + p}}$ for $p\in \mathcal{P}$. We say that $\mathcal{C}$ is $\mathbb{C}/\mathbb{Z}$-graded with periodicity group $\mathcal{P}$.
  Note that  $\quotient{(\mathbb{C}/\mathbb{Z})}{\mathcal{P}}\cong \mathbb{C}/\frac{N}{k}\mathbb{Z} \cong \mathbb{C}^*$ where the second bijection is $\alpha \mapsto q^{\alpha}$. 

\begin{notations}\label{notation_C..} Let  $\ddot{\mathbb{C}}:= (\mathbb{C}\setminus \frac{1}{2k}\mathbb{Z})\cup \frac{N}{2k}\mathbb{Z}$.
\end{notations}

\begin{lemma}\label{lemma_representations}
\begin{enumerate}
\item For  $\alpha \in \mathbb{C}$, $V_{\alpha}$ is simple if and only if  $\alpha \in \ddot{\mathbb{C}}$. Also $V_{\alpha}$ is projective if and only if $\alpha \in \ddot{\mathbb{C}}$.
Every $S_n$ are simple. The $P_n$ are indecomposable, projective but never simple.
\item The category $\mathcal{C}_{\overline{\alpha}}$ is semisimple if and only if $\alpha \notin \frac{1}{2k}\mathbb{Z}$ in which case the simple objects of $\mathcal{C}_{\overline{\alpha}}$ are the modules isomorphic to a $V_{\alpha}$ for $\alpha \in \overline{\alpha}$. 
\item The indecomposable projective modules of $\mathcal{C}_{\overline{0}}$ are the modules $V_0 \otimes \sigma^{mk}$ and $P_n \otimes \sigma^{mk}$. 
\item One has $V_{\alpha + n\frac{N}{k}}\cong V_{\alpha}\otimes \sigma^n$ and $S_{N-1}=V_0$.
\item For $n\in \{0, \ldots, N-2\}$, one has an exact sequence in $\mathcal{C}$:
$$ 0 \to S_n \xrightarrow{\iota} V_{N-1-n} \xrightarrow{p} S_{N-n-2}\otimes \sigma^{k} \to 0, $$
where $\iota(e_i) := v_{i+N-1-n}$ and $p(v_m):=e_m$ for $m\leq N-2-n$ and $p(v_l)=0$ for $l> N-2-n$.
\item For $n\in \{0, \ldots, N-2\}$, one has an exact sequence in $\mathcal{C}$:
$$ 0 \to V_{N-1-n} \xrightarrow{\iota} P_n \xrightarrow{p} V_{-(N-1-n)} \to 0, $$
where $\iota(v_i):= x_i$ and $p(x_i):=0$, $p(y_i):=v_i$. Moreover, this exact sequence does not split in $\mathcal{C}$.
\item Let $H_N:= \{ N-1-2n , n=0, \ldots, N-1\}$. If  $\alpha + \beta \notin  \frac{1}{2k}\mathbb{Z}$,  then 
$$ V_{\alpha}\otimes V_{\beta} \cong \oplus_{n\in H_N} V_{\alpha + \beta +n}.$$
\item If $\alpha \notin  \frac{1}{2k}\mathbb{Z} $, one has 
$$ V_{\alpha}\otimes S_n \cong \oplus_{j\in \{0, \ldots, n\}} V_{\alpha +(n-2j)}.$$
\item For $\alpha \in \ddot{\mathbb{C}}$ then 
$$ V_{\alpha}\otimes V_{-\alpha} \cong V_0 \oplus \oplus_{n=0}^{(N-3)/2} P_{2n}.$$
\end{enumerate}
\end{lemma}

\begin{proof} 
The proof is a straightforward adaptation of the arguments in \cite{CGP_unrolledQG} (where $q$ has order $2N$ instead of $N$) which we include for the reader convenience.
\par $(1)$ First note that $Ev_i=0$ if and only if $[\alpha-i]=0$, i.e. if and only if $\alpha \in i + \frac{N}{2k} \mathbb{Z}$. Such an $i \in \{1, \ldots, N-1\}$ exists if and only if $\alpha \in \{1, \ldots, N-1\}+ \frac{N}{2k}\mathbb{Z}$, i.e. if and only if $\alpha \notin \ddot{\mathbb{C}}$. 
So if  $\alpha  \in \ddot{\mathbb{C}}$ then  $Ev_i \neq 0$ for all $1 \leq i \leq N-1$ and every vector $v_n$, for $0\leq n \leq N-1$,  is cyclic. Moreover each axis $\mathbb{C}v_n$ is an eigenspace of the operator associated to $H$ so if $W \subset V_{\alpha}$ is a non-null submodule, then $W$ contains one vector $v_n$ and thus all of them since they are cyclic, so $W=V_{\alpha}$ and $V_{\alpha}$ is simple.
 The fact that $S_n$ is simple is proved similarly. 
\par $(2)$ Let $\widetilde{U_q\mathfrak{sl}_2}\subset U_q^H \mathfrak{sl}_2$ be the sub-Hopf algebra generated by $E,F, K^{\pm 1/2}$ (where $K^{1/2}=q^{H/2}$)  and let $\widetilde{\mathcal{C}}$ be the full subcategory of $ \widetilde{U_q\mathfrak{sl}_2}-\Mod$ of finite dimensional representations for which $E^N=F^N=0$ and $K^{N/2}$ acts semi-simply.  Consider the graduation $\widetilde{\mathcal{C}}=\sqcup_{\overline{\alpha}\in \mathbb{C} / \mathbb{Z}} \widetilde{\mathcal{C}}_{\overline{\alpha}}$ where $K^{N/2}$ acts as $e^{i \pi \alpha} \id$ on the modules in $\widetilde{\mathcal{C}}_{\overline{\alpha}}$.
There is an obvious forgetful functor $\mathcal{C}\to \widetilde{\mathcal{C}}$ sending $V$ to $\widetilde{V}$ which consists in forgetting how $H$ acts. A lift from $\widetilde{V}$ to $V$ is simply determined by how we lift the action of $K^{1/2}$ to an action of $H$ and we easily see that if $V, V' \in \mathcal{C}$ are indecomposable and sent to the same $\widetilde{V} \in \widetilde{C}$ then $V' \cong V \otimes \sigma^{\otimes n}$ for some $n\in \mathbb{Z}$. So  it suffices to prove that $\widetilde{\mathcal{C}}_{\overline{\alpha}}$ is semisimple when $\alpha \notin \mathbb{Z}$ with simple objects isomorphic to the $\widetilde{V}_{\alpha}$ with $\alpha \in \overline{\alpha}$. It is standard facts, using PBW basis and the filtration associated, that $\widetilde{U_q\mathfrak{sl}_2}$ is a domain and that its center $Z$ is generated by $K^{\pm N/2}$ and the Casimir $C:= EF + \frac{q K^{-1}+q^{-1}K}{(q-q^{-1})^2}$ with relation $T_N(-(q-q^{-1})^2 C)= K^N + K^{-N} $.  Moreover, $\widetilde{U_q\mathfrak{sl}_2}$ is free over $Z^0:= \mathbb{C}[K^{\pm N/2}]$ with basis given by: 
$$ \{ E^a K^{b/2} F^c,  0 \leq a,b,c \leq N-1\}.$$
So $\widetilde{U_q\mathfrak{sl}_2}$  has rank $N^3$ over $Z^0$ and  $Z$ is an algebraic extension of $Z^0$ of degree $N$ (since $Z\cong \quotient{Z^0[X]}{(T_N(X)-c)}$ for $c\in Z^0$). 
Therefore $\widetilde{U_q\mathfrak{sl}_2}$ is almost Azumaya and its PI-degree is $N$($=\sqrt{\frac{N^3}{N}}$). Since $\widetilde{V}_{\alpha}$ has dimension $N$ and is simple for $\alpha \notin \mathbb{Z}$, its classical shadow belongs to the Azumaya locus of $\widetilde{U_q\mathfrak{sl}_2}$; this implies that $\widetilde{\mathcal{C}}_{\overline{\alpha}}$ is semisimple with simple objects the $\widetilde{V}_{\alpha}$, $\alpha \in \overline{\alpha}$ which are thus projective (in a semisimple category, every object is projective).  
\par Item $(3)$ follows from the fact, proved in \cite{KojuKaruo_RepRSSkein}, that in $\widetilde{\mathcal{C}}_{\overline{0}}\cong \mathfrak{u}_q\mathfrak{sl}_2-\Mod$ the indecomposable projective objects are the $\widetilde{V}_0$ and $\widetilde{P}_n$. The proof given in  \cite{KojuKaruo_RepRSSkein} is a straightforward adaptation of the similar statement proven in \cite{Suter_Uqsl2Modules} when $q$ has order $2N$. 
\par Items $(4)$, $(5)$ and $(6)$ are straightforward verifications. They imply that $P_n$ is not simple and indecomposable and that  when $\alpha \in \ddot{\mathbb{C}}$, then $V_{\alpha}$ is  not simple so its classical shadow does not belong to the Azumaya locus and $\mathcal{C}_{\overline{\alpha}}$ is not semisimple. 
\par To prove items $(7)$, $(8)$  and $(9)$, 
consider the symbol morphism $\chi : K^0(\mathcal{C}) \to \mathbb{Z}[ \mathbb{C}]$ sending a weight module $V$ with weight decomposition $V=\oplus_{\alpha \in \mathbb{C}}V(\alpha)$ (here $V(\alpha)$ is a weight submodule of weight $\alpha$) to its symbol $\chi([V]) = \sum_{\alpha \in \mathbb{C}} \dim(V(\alpha))X^{\alpha}$. Since $H$ is primitive, the symbol map is a ring morphism.  For instance $\chi(V_{\alpha})=\sum_{n\in H_N} X^{\alpha +n}$, $\chi(S_n)=X^n + X^{n-2} + \ldots + X^{-n}$,  $\chi(\sigma)=X^{\frac{N}{2k}}$ and $\chi(P_m)= \sum_{n \in H_N} X^{N-1-n+m} + X^{-(N-1-n+m)}$.
By $(2)$ and $(3)$,  the restriction of  $\chi$ to the subcategory of projective objects in $\mathcal{C}$ is injective since the $\chi(V_{\alpha}\otimes \sigma^m)$, $\alpha \in \ddot{\mathbb{C}}$ and $\chi(P_n\otimes \sigma^m)$ are linearly independent. Items $(7)$, $(8)$ and $(9)$ 
 thus follow from the equalities 
$$ \chi(V_{\alpha})\chi(V_{\beta})= \sum_{n\in H_N} \chi(V_{\alpha+\beta+n}) \ \mbox{ and }\  \chi(V_{\alpha})\chi(S_n) = \sum_{j=0}^{n} \chi(V_{\alpha+(n-2j)})$$
$$ \chi(V_{\alpha})\chi(V_{-\alpha})= \chi(V_0)+ \sum_{n=0}^{N-3/2} \chi(P_{2n}).$$
\end{proof}

\begin{remark} The categories $\mathcal{C}_{\overline{\alpha}}$ with $\alpha \in \frac{1}{k}\mathbb{Z}$ are pairwise isomorphic (by tensoring with some $\sigma^{\otimes n}$). By definition, $\widetilde{\mathcal{C}}_{\overline{0}}$ is the category of (weight) representations of the small quantum group $\mathfrak{u}_q \mathfrak{sl}_2$ whose indecomposables were classified in \cite{Suter_Uqsl2Modules} when $q$ has even order and in \cite{KojuKaruo_RepRSSkein} at odd order. It 
 contains an infinite (countable) number of indecomposable representations up to isomorphism but the only indecomposable projective representations are isomorphic to the $\widetilde{P_n}$ and  $\widetilde{V_0}$ (see \cite[Appendix A]{KojuKaruo_RepRSSkein} for details). In this paper we will essentially use the semisimple categories $\mathcal{C}_{\overline{\alpha}}$ for $\alpha \notin \frac{1}{2k}\mathbb{Z}$.
\end{remark}

\subsection{The ribbon structure and twisted skein algebra}
Consider the completions $\widehat{U_q^H\mathfrak{sl}_2}:= \int^{V \in \mathcal{C}} V^* \otimes V$ and $\widehat{(U_q^H\mathfrak{sl}_2)^{\otimes 2}}:= \int^{V,W \in \mathcal{C}} V^* \otimes V \otimes W^* \otimes W$. Let $q^{\frac{H\otimes H}{2}} \in \widehat{(U_q^H\mathfrak{sl}_2)^{\otimes 2}}$ be the operator defined on $V\otimes W$ by 
$$ q^{\frac{H\otimes H}{2}} v\otimes w := A^{wt(v)wt(w)}v\otimes w, $$
where $v\in V$ and $w\in W$ are weight vectors of weights $wt(v)$ and $wt(w)$ respectively (i.e. $Hv=wt(v) v$ and $Hw=wt(w)w$). The $R$-matrix $\mathcal{R} \in \widehat{(U_q^H\mathfrak{sl}_2)^{\otimes 2}}$ is defined by 
$$ \mathcal{R}:= q^{\frac{H\otimes H}{2}} \exp_q^{<N}( (q-q^{-1}) E\otimes F), $$
where 
$$ \exp_q^{<N}(X):= \sum_{n=0}^{N-1} \frac{q^{n(n-1)/2}}{[n] !}X^n .$$

The twist $\Theta_0 \in \widehat{U_q^H\mathfrak{sl}_2}$ is the operator defined  by 
$$ \Theta_0 v := K^{N-1} \sum_{n=0}^{N-1}\frac{q^{n(n-1)/2}}{[n] !} (q-q^{-1})^n S(F^n) q^{-H^2/2}E^n, $$
where $q^{-H^2/2}$ is defined by $q^{-H^2/2}v=A^{-wt(v)^2}v$ for a weight vector $v$. The triple $(U_q^H \mathfrak{sl}_2, \mathcal{R}, \Theta_0)$ forms a topological ribbon Hopf algebra which endows the category $\mathcal{C}$ with a structure of ribbon category. The braiding will be denoted by $c_{V,W}: V\otimes W \to W\otimes V$ and the twist will be denoted by $\theta_V: V \to V$. 
 The duality in $\mathcal{C}$ is described as follows. For $V\in \mathcal{C}$, its dual is $V^*=\Hom(V, \mathds{1})$. The left duals are 
$$ \overrightarrow{ev}_V : V^*\otimes V \to \mathds{1}, \overrightarrow{ev}_V( f\otimes v) := f(v), \quad \overrightarrow{coev}_V: \mathds{1} \to V\otimes V^*, \overrightarrow{coev}_V(1):= \sum_i e_i \otimes e_i^*, $$
where $(e_i)$ is a basis of $V$ and $e_i^*$ its dual basis. The right duals are defined by 
$$ \overleftarrow{ev}_V: V\otimes V^* \to \mathds{1}, \overleftarrow{ev}_V (v\otimes f) :=  f(K^{1-N}v), \quad \overleftarrow{coev}_V: \mathds{1} \to V^* \otimes V, \overleftarrow{coev}_V(1)= \sum_i e_i^* \otimes K^{N-1}e_i.$$
In particular $\qdim (S_n)= [n+1]$,  $\qdim(V_{\alpha})=0$ and $\qdim(\sigma^{n})=1$.

\begin{lemma}\label{lemma_almost_transparent}  One has $c_{V_{\alpha}, \sigma^n} \circ c_{\sigma^n, V_{\alpha}} = e^{2i\pi \alpha n} \id_{\sigma^n \otimes V_{\alpha}}$ and $c_{\sigma^{nk}, \sigma^m}\circ c_{\sigma^{m}, \sigma^{nk}}=\id_{\sigma^{nmk}}$.
\end{lemma}

\begin{proof} The first equality follows from the fact that $\exp_q^{<N}((q-q^{-1}) E\otimes F)$ acts by the identity operator on both $\sigma^n \otimes V_{\alpha}$ and $V_{\alpha}\otimes \sigma^n$ and $q^{H\otimes H/2}$ acts by $e^{i\pi n\alpha}$ times the identity on both spaces. Similarly, in $\sigma^m \otimes \sigma^{nk}$, we have  
$$c_{\sigma^{nk}, \sigma^m}\circ c_{\sigma^{m}, \sigma^{nk}}(v\otimes v)= q^{H\otimes H} v\otimes v= q^{(m \frac{N}{k})(nk(\frac{N}{k}))}v\otimes v=v\otimes v.$$
\end{proof}

\begin{notations}\label{notations_ribC} We denote by $\Rib_{\mathcal{C}}$ the category of $\mathcal{C}$ ribbon graphs and by $F^{RT} : \Rib_{\mathcal{C}} \to \mathcal{C}$ the Reshetikhin-Turaev (ribbon, surjective) functor as defined in \cite{Tu}. So objects are tuples $( (V_1, \epsilon_1), \ldots, (V_n, \epsilon_n))$ with $V_i \in \mathcal{C}$ and $\epsilon_i=\pm$, and morphisms are linear combinations of ribbon tangles read from bottom to top. More precisely, for each $n\geq 0$ we consider the set $P_n\subset [0,1]^2$, $P_n=\{p_1, \ldots, p_n\}$ where $p_i= (\frac{i}{n+1}, \frac{i}{2})$. A ribbon tangle morphism $T: ( (V_1, \epsilon_1), \ldots, (V_n, \epsilon_n)) \to ( (W_1, \nu_1), \ldots, (W_m, \nu_m))$ is a $\mathcal{C}$-colored ribbon tangle $T\subset [0,1]^3$ with $\partial T= (T\cap P_n \times \{0\}) \cup (T \cap P_m \times \{1\})$ with compatible colors and orientations
 (see \cite[Chapter $1$ ]{Tu} for details). 
\end{notations}

For $V \in \mathcal{C}$ a simple object and $f\in \End(V)$, we denote by $\left< f \right> \in \mathbb{C}$ the (unique) scalar such that $f= \left< f \right> \id_V$. For $V,W\in \mathcal{C}$ two simple objects, we denote by $S'(V,W) \in \mathbb{C}$ the scalar 
$$ S'(V,W) = \left< F^{RT} \left(   \adjustbox{valign=c}{\includegraphics[width=0.8cm]{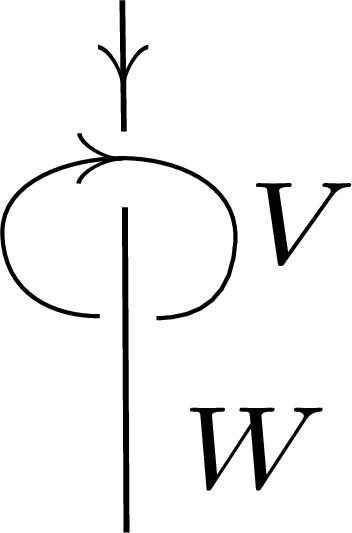}}\right) \right> \in \mathbb{C}.$$

\begin{lemma}\label{lemma_S}
One has 
$$ S'(S_1, V_{\alpha}) = q^{\alpha}+q^{-\alpha}, \quad S'(S_1, S_i) = q^{i+1}+q^{-(i+1)}.$$
\end{lemma}

\begin{proof} The proof is an easy adaptation (left to the reader)  of the similar results in \cite{CGP_unrolledQG} when $q$ has order $2N$.
\end{proof}

\subsection{Renormalized trace and link invariants}\label{sec_links}

Recall that a trace on a category $\mathcal{D}$ is a ring morphism $\tau : \mathrm{HH}_0(\mathcal{D})\to \mathbb{C}$, where $\mathrm{HH}_0(\mathcal{D})=\int^{V\in \mathcal{D}}\End(V)$. The set $\mathrm{HH}_0(\Rib_{\mathcal{D}})$ identifies canonically with the set $\mathcal{RG}_{\mathcal{D}}(\mathbb{S}^3)$ of $\mathcal{D}$-colored ribbon graphs in $\mathbb{S}^3$, so a trace $\tau$ on a ribbon category defines a link invariant 
$$ \left< \cdot \right> : \mathcal{RG}_{\mathcal{D}}(\mathbb{S}^3)\cong \mathrm{HH}_0(\Rib_{\mathcal{D}}) \xrightarrow{F^{RT}} \mathrm{HH}_0(\mathcal{D}) \xrightarrow{\tau} \mathbb{C}.$$
Any ribbon category $\mathcal{D}$ has two canonical traces $\qtr^{L, R}$ given on $f \in \End(V)$ by $\qtr^L(f):= \overrightarrow{ev}_V \circ (\id_{V^*} \otimes f) \circ \overleftarrow{coev}_V$ and $\qtr^R(f):= \overleftarrow{ev}_V \circ (f \otimes \id_{V^*}) \circ \overrightarrow{coev}_V$.  On $\mathcal{C}$ both traces coincide. However, because $\qdim(V_{\alpha}):= \qtr( \id_{V_{\alpha}})=0$ for all $\alpha \in \mathbb{C}$, the associated link invariant vanishes on any ribbon graph with an edge colored by a $V_{\alpha}$. To circumvent this problem, the authors of \cite{GeerPatureauTuraev_Trace} introduced the so-called \textit{modified trace} further studied in \cite{GeerPatureauVirelizier_Trace} that we now briefly describe. 
 Let 
 $$ B := \{ V_{\alpha}, \alpha \in \ddot{\mathbb{C}} \}$$
and $\mathbf{d} : B \to \mathbb{C}^*$ be the map defined by 
$$ \mathbf{d}(V_{\alpha}) =  \frac{ \{ \alpha +1-N \} }{ \{ N(\alpha +1-N) \} }.$$
Then $(B, \mathbf{d})$ is an ambidextrous pair in the sense of \cite{GeerPatureauVirelizier_Trace} and defines a trace $\tau: \mathrm{HH}_0(\mathcal{P}) \to \mathbb{C}$ on the full subcategory $\mathcal{P} \subset \mathcal{C}$ of projective objects of $\mathcal{C}$
such that when $f \in \End(V_{\alpha})$ then $\tau(f)= \mathbf{d}(V_{\alpha}) \left< f \right>$.
 Let $\mathcal{RG}_{\mathcal{P}}(\mathbb{S}^3)$ be the set of (isotopy classes of) $\mathcal{C}$-colored ribbon graphs in $\mathbb{S}^3$ having at least one edge colored by a projective object. The renormalized trace $\tau$ defines an invariant 
$$ \left< \cdot \right> : \mathcal{RG}_{\mathcal{P}}(\mathbb{S}^3) \to \mathbb{C}$$
studied in great details in \cite{GeerPatureau_LinksInv}. When $\Gamma \in \mathcal{RG}_{\mathcal{P}}(\mathbb{S}^3)$ is a ribbon graph with one edge $e$ colored by a module $V_{\alpha}$ for $\alpha \in  \ddot{\mathbb{C}}$, let $\Gamma(e)$ be the $1-1$ tangle in $\Rib_{\mathcal{P}}( (V_{\alpha}, +) , (V_{\alpha}, +) )$ obtained by cutting $\Gamma$ along $e$. The invariant then satisfies 
$$ \left< \Gamma \right> = \mathbf{d}(V_{\alpha}) \left< F^{RT}(\Gamma(e)) \right>.$$
The main result of \cite{GeerPatureau_LinksInv} is the fact that this invariant does not depend on the edge $e$. 

\subsection{Multiplicity modules and trivalent graph calculus}\label{sec_graphcalculus}
For $\alpha, \beta, \gamma \in \ddot{\mathbb{C}}$ in virtue of Items $(7)$, $(9)$ in Lemma \ref{lemma_representations}  the spaces $\Hom_{\mathcal{C}}(V_{\alpha}\otimes V_{\beta} \otimes V_{\gamma}, \mathds{1})$ and $\Hom_{\mathcal{C}}(\mathds{1}, V_{\alpha}\otimes V_{\beta} \otimes V_{\gamma})$ have dimension either $0$ or $1$. The goal of this subsection is to fix, whenever this dimension is $1$, a non-zero element in each of these spaces. This will permit one to consider trivalent ribbon graphs colored by indecomposable modules without specifying the coupons inserted at each vertex and by only considering the cyclic ordering of the adjacent half-edges (see \cite[Section $6.1$]{Tu} for an intensive discussion on these conventions).
\vspace{2mm} \par 
\textbf{Trivalent graph calculs}
A triple $(\alpha, \beta, \gamma)$ is $0$-\textit{admissible} if $\alpha, \beta, \gamma \in \ddot{\mathbb{C}}$  and  $\alpha+\beta-\gamma \in H_N$. In this case, we now define non-zero morphisms
$$ Y_{\gamma}^{\alpha, \beta} : V_{\gamma}\to V_{\alpha}\otimes V_{\beta}, \quad Y_{\alpha, \beta}^{\gamma}: V_{\alpha}\otimes V_{\beta}\to V_{\gamma}, \quad \cap_{\alpha}: V_{\alpha}\otimes V_{-\alpha}\to \mathds{1}, \quad \cup_{\alpha} : \mathds{1}\to V_{\alpha}\otimes V_{-\alpha}, \quad w_{\alpha}: V_{\alpha}\cong (V_{-\alpha})^*$$
such that, writing $ \adjustbox{valign=c}{\includegraphics[width=3cm]{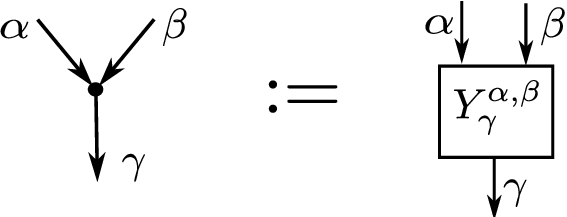}}$, $ \adjustbox{valign=c}{\includegraphics[width=2.5cm]{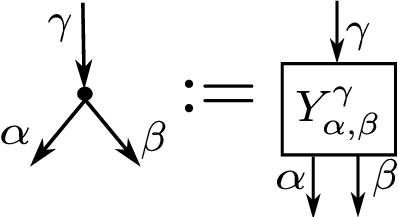}}$, $ \adjustbox{valign=c}{\includegraphics[width=2.5cm]{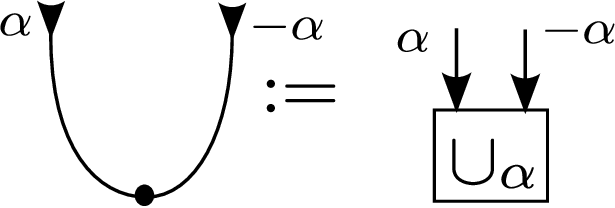}}$ and $ \adjustbox{valign=c}{\includegraphics[width=2.5cm]{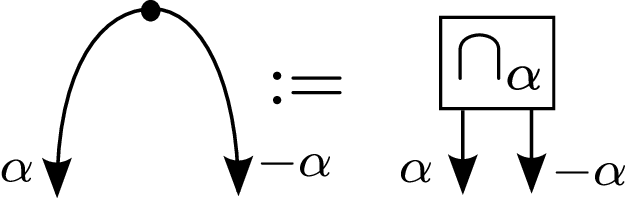}}$, 

 the following equalities hold:
\begin{equation}\label{eq_Cap}
\adjustbox{valign=c}{\includegraphics[width=10cm]{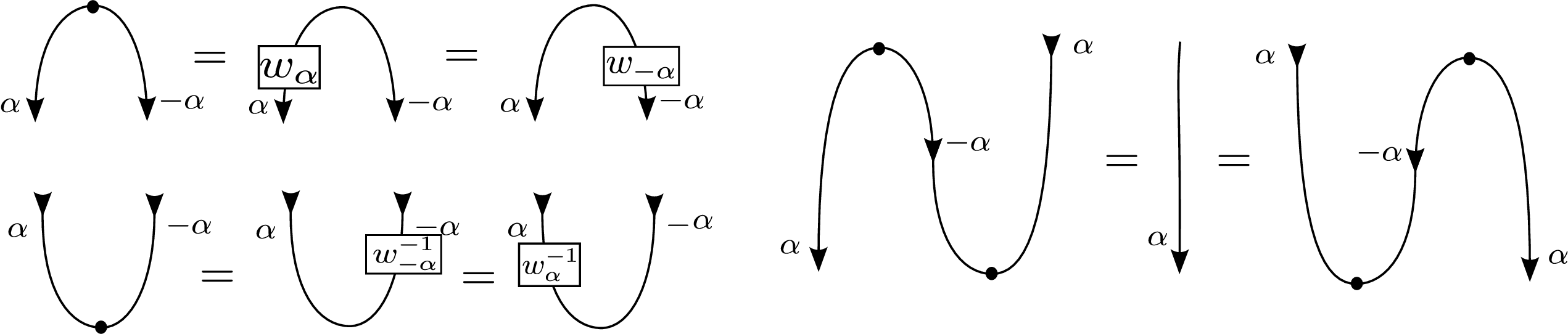}}
\end{equation}
\begin{equation}\label{eq_Y}
\adjustbox{valign=c}{\includegraphics[width=7cm]{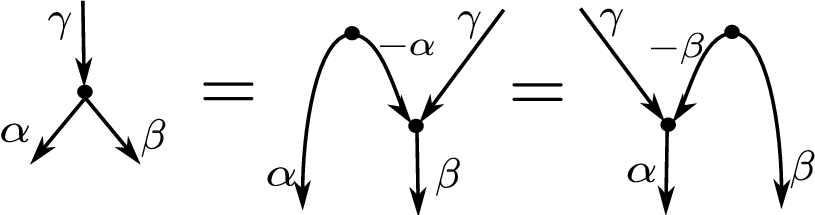}}
\end{equation}
\begin{equation}\label{eq_3j}
\adjustbox{valign=c}{\includegraphics[width=9cm]{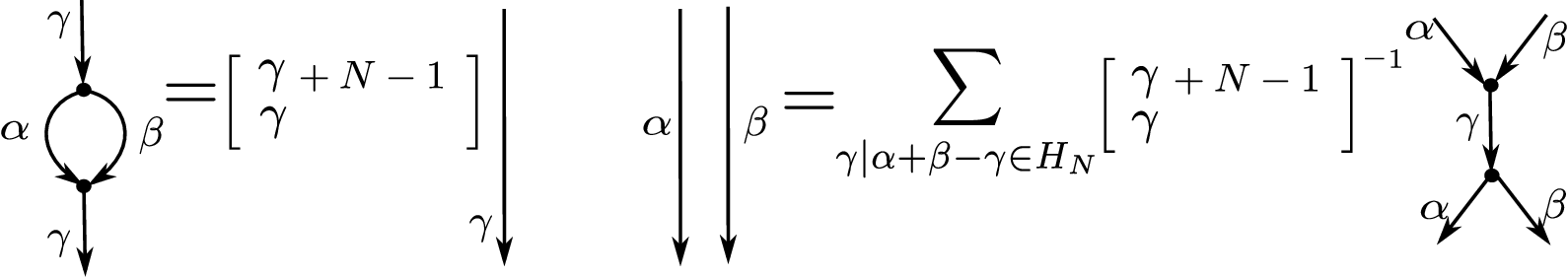}}
\end{equation}
Note that, by  Items $(7)$, $(9)$ of Lemma \ref{lemma_representations}, each of these operators is uniquely defined up to multiplication by a non-zero scalar. 
Here and in all the paper, we use the convention that if a strand of a ribbon graph is colored by $\alpha \in \mathbb{C}$, we actually mean that it is colored by $V_{\alpha}$.
The existence and explicit formulas for these operators will be derived from the work of Costantino-Murakami in \cite{CostantinoMurakami_6j} who compute such morphisms in the case of the (non-unrolled) quantum group $U_q\mathfrak{sl}_2$ with $q$ a $2N$-root of unity. 

\begin{notations}\label{notations_formulas} Let $n \in \mathbb{N}$ and $\alpha, \beta \in \mathbb{C}$ such that $\alpha - \beta \in \{0, \ldots, N-1\}$. We write
$$ \{n\} := q^n - q^{-n}; \quad \{n\} ! := \{n\} \{n-1\} \ldots \{1\}; \{\alpha, \beta\} := \{\alpha\} \{\alpha -1\} \ldots \{ \beta+1\}; \quad \qbinom{\alpha}{\beta}:= \frac{ \{\alpha, \beta\} }{ \{\alpha - \beta\} ! }.$$
When $\alpha=\beta$, we set $\{\alpha, \alpha\}= \qbinom{\alpha}{\alpha}=1$. Let
 $$\kappa_n^{\alpha}:= q^{- \frac{\alpha}{2}+\frac{n(n-1)}{2}} \left\{ \alpha+N-1, \alpha +N-1-n \right\},\quad 
w_{\alpha}^n := \kappa^{\alpha}_n \kappa^{-\alpha}_{N-1-n} q^{-\frac{\alpha}{2}(N-1)-n-\frac{1}{2}}.
$$
Now suppose that $\alpha + \beta - \gamma \in H_N$ and let $i,j,n \in \{0, \ldots, N-1\}$. If $i+j-n \neq \frac{\alpha + \beta - \gamma}{2} +\frac{N-1}{2}$, we set $C_{i,j,n}^{\alpha, \beta, \gamma} := 0$. If $i+j-n = \frac{\alpha + \beta - \gamma}{2} +\frac{N-1}{2}$, we set
\begin{multline*}
C_{i,j,n}^{\alpha, \beta, \gamma} :=q^{\frac{1}{2}\left[ j(\beta-j) - i(\alpha-i)\right]} \qbinom{\gamma+N-1}{\gamma +N-1-n}^{-1} \qbinom{\gamma +N-1}{\alpha+ \beta+\gamma+\frac{N-1}{2}} \\
\sum_{z+w=n} (-1)^z q^{\frac{1}{2}\left[(2z-n)(\gamma -n)\right]}
\qbinom{i+j-n}{i-z} \qbinom{\alpha -i+z+N -1}{\alpha-i+N-1} \qbinom{\beta -j+w +N -1}{\beta-j+N-1}.
\end{multline*}
We eventually set
$$ D_{i,j,n}^{\alpha, \beta, \gamma}:= \frac{\kappa_n^{\gamma}}{\kappa_i^{\alpha} \kappa_j^{\beta}}  C_{i,j,n}^{\alpha, \beta, \gamma} \quad \mbox{ and }
E_{i,j,n}^{\alpha, \beta, \gamma}:= \frac{\kappa_i^{\alpha} \kappa_j^{\beta}}{\kappa_n^{\gamma}} C_{N-1-j, N-1-i, N-1-n}^{-\beta, -\alpha, - \gamma}.$$
\end{notations}

\begin{definition}\label{def_morphisms}
For $(\alpha, \beta, \gamma)$ admissible, the morphisms 
$$ Y_{\gamma}^{\alpha, \beta} : V_{\gamma}\to V_{\alpha}\otimes V_{\beta}, \quad Y_{\alpha, \beta}^{\gamma}: V_{\alpha}\otimes V_{\beta}\to V_{\gamma}, \quad \cap_{\alpha}: V_{\alpha}\otimes V_{-\alpha}\to \mathds{1}, \quad \cup_{\alpha} : \mathds{1}\to V_{\alpha}\otimes V_{-\alpha}, \quad w_{\alpha}: V_{\alpha}\cong (V_{-\alpha})^*$$
are defined by 
\begin{itemize}
\item $w_{\alpha}(v_i):= w_{\alpha}^i (v_{N-1-i})^*; $
\item $\cap_{\alpha}(v_i\otimes v_j) := \delta_{i+j, N-1}w_{\alpha}^i 
; \quad \cup_{\alpha}(1):=  \sum_i  \left( w_{-\alpha}^{N-1-i} \right)^{-1}v_i \otimes v_{N-1-i}; $
\item $Y_{\gamma}^{\alpha, \beta} (v_n) = \sum_{i,j} D_{i,j,n}^{\alpha, \beta, \gamma} v_i\otimes v_j; \quad  Y_{\alpha, \beta}^{\gamma}(v_i\otimes v_j)=\sum_n E_{i,j,n}^{\alpha, \beta, \gamma} v_n.$
\end{itemize}
\end{definition}

The following proposition follows from a straightforward adaptation of the computations in \cite{CostantinoMurakami_6j} detailed in Appendix \ref{sec_appendixA}.

\begin{proposition}\label{prop_appendixA}
The morphisms $w_{\alpha}$, $\cap_{\alpha}$, $\cup_{\alpha}$, $Y_{\alpha, \beta}^{\gamma}$ and $Y_{\gamma}^{\alpha, \beta}$ are equivariant (i.e. are morphisms in $\mathcal{C}$) and satisfy the equalities \eqref{eq_Cap}, \eqref{eq_Y} and \eqref{eq_3j}.
\end{proposition}

For $(\alpha, \beta, - \gamma)$ admissible, we define the multiplicity modules $H^{\alpha, \beta, \gamma}\in \Hom(\mathds{1}, V_{\alpha}\otimes V_{\beta} \otimes V_{\gamma})$ and $H_{\alpha, \beta, \gamma}\in \Hom(V_{\alpha}\otimes V_{\beta} \otimes V_{\gamma}, \mathds{1})$  by 
$$\adjustbox{valign=c}{\includegraphics[width=10cm]{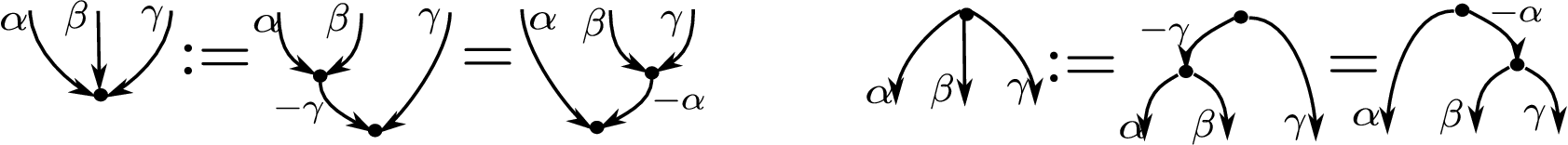}}.$$

We can now describe our conventions for trivalent graph calculus. Let $G \subset \mathbb{S}^3$ be a ribbon trivalent graph. A \textit{coloring} of $G$ is a chain $c \in \mathrm{C}_1(G; \mathbb{C})$, where $G$ is seen as a CW complex. Equivalently, $c$ is described by a choice of orientation $\mathfrak{o}$ to the edges ${E}$ of $G$ together with a map $c: E\to \mathbb{C}$ which we identify with the chain $\sum_{e\in E} c(e)e$. A coloring is \textit{strictly admissible} if $(i)$ $c(e) \in \ddot{\mathbb{C}}$ and $(ii)$ the map $\partial c : V(G)\to \mathbb{C}$ satisfies $\partial c (v) \in H_N=\{-(N-1), \ldots, N-3, N-1\}$ for all vertices $v\in V(G)$. For $(G,c)$ a trivalent graph with strictly admissible coloring, we define the invariant 
$$ \left< G, c \right> \in \mathbb{C}$$
obtained by considering a standard planar projection of $G$ with normal double crossings, coloring an oriented edge $e$ by $V_{c(e)}$, 
 inserting the evaluations and coevaluations $\cup_{\alpha}$ and $\cap_{\alpha}$ and considering the renormalized invariant of the obtained $\mathcal{C}$-colored ribbon graph. The strict admissibility condition ensures that for three adjacent edges oriented towards a common vertex and cyclically ordered by the ribbon graph, the associated highest weight modules $V_{\alpha}$, $V_{\beta}$, $V_{\gamma}$ are such that $(\alpha, \beta, -\gamma)$ is admissible. 

\vspace{3mm}

\par \textbf{The inclusion of the modules $\sigma^n$}
\begin{definition}(Admissible coloring)
A coloring $c\in \mathrm{C}_1(G; \mathbb{C})$  is said \textit{admissible} if  $(i)$ $c(e) \in \ddot{\mathbb{C}}$ for all oriented edge $e$ and $(ii)$ $\partial c (v) \in \mathbb{Z}$ for all vertices $v\in V(G)$.
\end{definition}
 We now extend our trivalent graph calculus in order to include invariants $\left<G, c\right>$ of (non-necessarily strictly) admissible colorings using the modules $\sigma^n$. We will use the convention that a dotted edge colored by $n\in \mathbb{Z}$ represents an edge colored by $\sigma^n$, i.e. $\adjustbox{valign=c}{\includegraphics[width=1cm]{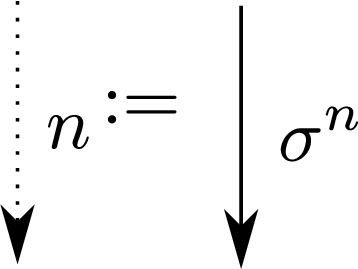}} $. Using the isomorphism $\sigma^{n_1}\otimes \sigma^{n_2}\cong \sigma^{n_1+n_2}$ sending $v\otimes v$ to $v$,  the isomorphisms $V_{\alpha}\otimes \sigma^n \cong V_{\alpha+ n\frac{N}{k}}$ sending $v_i\otimes v$ to $v_i$ and the isomorphisms $\sigma^n \otimes V_{\alpha} \cong V_{\alpha+ n\frac{N}{k}}$ sending $v\otimes v_i$ to $v_i$, we extend our trivalent graph calculus to trivalent graphs having edges possibly colored by some $\sigma^n$.
\par Suppose that $\alpha, \beta, \gamma \in \ddot{\mathbb{C}}$ are such that $\alpha + \beta +\gamma \in \frac{1}{k}\mathbb{Z}$. Then there exists a unique $n\in \mathbb{Z}$ such that $\alpha + \beta + \gamma \in \frac{N}{k}n +H_N$.
In this case, the triple $(\alpha, \beta, -\gamma)$ is said $n$-\textit{admissible}.
Define the multiplicity modules $H^{\alpha, \beta, \gamma}_n\in \Hom(\sigma^n, V_{\alpha}\otimes V_{\beta} \otimes V_{\gamma})$ and $H_{\alpha, \beta, \gamma}^n\in \Hom(V_{\alpha}\otimes V_{\beta} \otimes V_{\gamma}, \sigma^n)$ by 
$$\adjustbox{valign=c}{\includegraphics[width=12cm]{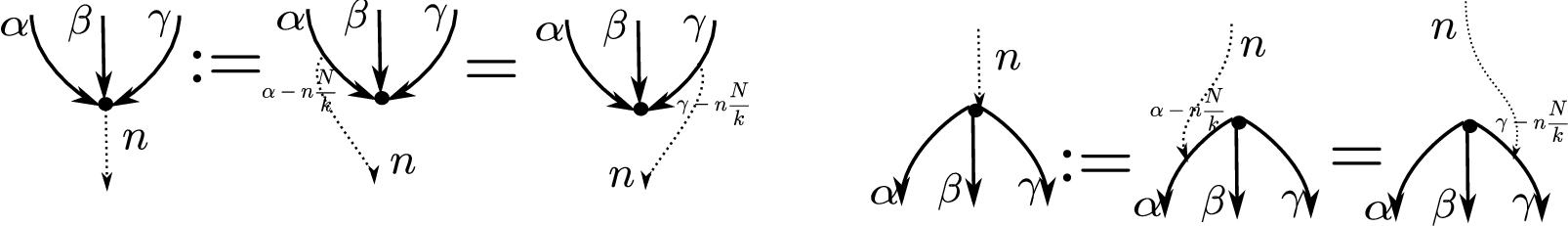}}.$$

The above equalities are ensured by the fact that the matrix coefficients $w_{\alpha}^i$, $D_{i,j,n}^{\alpha, \beta, \gamma}$ and $E_{i,j,n}^{\alpha, \beta, \gamma}$ only depend on the classes of $\alpha, \beta, \gamma$ in ${\mathbb{C}}/{\frac{N}{k}\mathbb{Z}}$ (mainly because $\{x+ \frac{N}{k}n\}= \{x\}$). For $(G,c)$ with $c$ admissible, for each vertex $v\in V(G)$ then there exists a unique $n\in \mathbb{Z}$ such that $\partial c (v) \in nN +H_N$ and we want to introduce a module $\sigma^{nk}$ adjacent to $v$ that is, we want to replace $\adjustbox{valign=c}{\includegraphics[width=1cm]{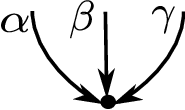}}$ by $\adjustbox{valign=c}{\includegraphics[width=1cm]{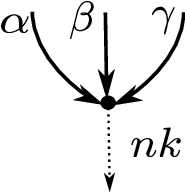}}$ since $(\alpha, \beta, \gamma)$ is $nk$-admissible. This requires to specify $(1)$ where the edge colored by $\sigma^{nk}$ is located in the cyclic order of the half-edges adjacent to $v$ by specifying a ciliated structure on $G$ (in $\adjustbox{valign=c}{\includegraphics[width=1cm]{HnDef2.eps}}$, the edge colored by $\sigma^{nk}$ is between the edge colored by $\alpha$ and the edge colored by $\gamma$) and $(2)$ how the edges colored by $\sigma^n$ are connected. Recall that in a ribbon graph the half-edges adjacent to a vertex $v$ are cyclically ordered and that a \textit{ciliated structure} is the data of a lift of this cyclic ordering to a total ordering (for instance by drawing a cilium between the last and the first half-edge).

\begin{definition}($\sigma$ decoration) A $\sigma$-\textit{decoration} on a trivalent ribbon graph $G \subset S^3$ is the data of $(1)$ a ciliated structure on $G$, $(2)$ a point $v_{\infty} \in S^3\setminus G$ and $(3)$ for each vertex $v$, a smooth path $p_v: [0,1] \to S^3\setminus G \cup \{v\}$ with $p_v(0)=v$ and $p_v(1)=v_{\infty}$.
\end{definition}

When $G$ is equipped with a $\sigma$-decoration and $c$ is an admissible coloring, we define a $\mathcal{C}$-colored graph $G(c)\subset S^3$ by first orienting the edges of $G$ arbitrary and coloring each edge $e$ by $V_{c(e)}$. Then for each vertex $v$, with $\partial c (v)\in H_N +Nn_v$ where $n_v\in \mathbb{Z}$, color the path $p_v$ with $\sigma^{n_vk}$ (with arbitrary framing), replace $v$ by a coupon colored by $H^{\alpha, \beta, \gamma}_{n_v}$ (if the three adjacent edges are colored by $\alpha, \beta, \gamma$ in counter-clockwise order) and eventually replace $v_{\infty}$ by a coupon colored by the morphism $\bigotimes_{v \in V(G)} \sigma^{kn_v} \to \mathds{1}$ sending $\bigotimes v$ to $1$ (well-defined since $\sum_v n_v=0$).
\par We eventually define
$$ \left< G, c \right> \in \mathbb{C} $$
as the invariant of the graph $G(c)$.

For instance for the theta graph with $\sigma$ decoration $\adjustbox{valign=c}{\includegraphics[width=2.8cm]{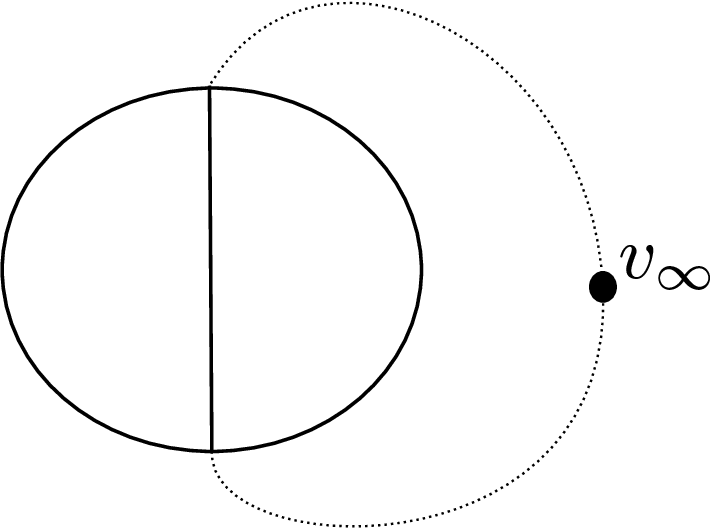}}$
and colored by $\alpha, \beta, \gamma \in \ddot{\mathbb{C}}$ with $\alpha+\beta+\gamma \in nN+H_N$, we set
$$\adjustbox{valign=c}{\includegraphics[width=10cm]{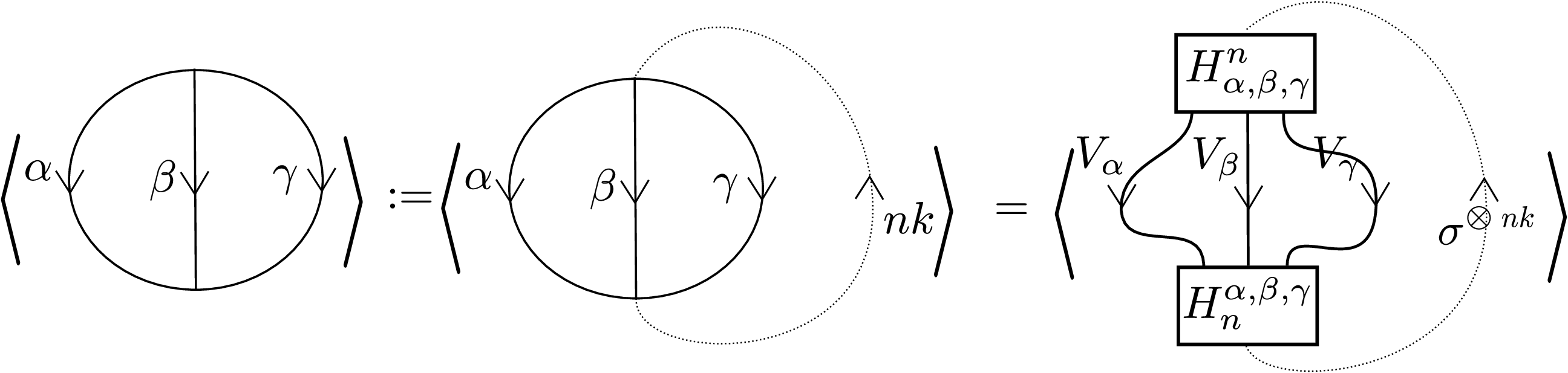}}.$$

\par When $c$ is an admissible coloring of $G$, since $\partial c \in \mathrm{C}_0(G; \mathbb{Z})$, then $c$ induces a class $\overline{c} \in \mathrm{H}_1(G; \mathbb{C}/\mathbb{Z})$ to which we associate a class $\omega_{\overline{c}} \in \mathrm{H}^1(S^3\setminus G; \mathbb{C}/\mathbb{Z})$ as follows. The inclusion $G \subset S^3$ induces a long exact sequence in homology:
$$ 0= \mathrm{H}_2(S^3) \to \mathrm{H}_2(S^3, G) \xrightarrow[\cong]{\partial} \mathrm{H}_1(G) \to \mathrm{H}_1(S^3)=0$$
where homology is taken with coefficients in $\mathbb{C}/\mathbb{Z}$. By excision, we have $\mathrm{H}_2(S^3, G) \cong \mathrm{H}_2(S^3 \setminus G)$ so we obtain an isomorphism $\delta_1:  \mathrm{H}_1(G)\xrightarrow{\partial^{-1}} \mathrm{H}_2(S^3, G) \cong \mathrm{H}_2(S^3\setminus G)$. Let $\delta_2: \mathrm{H}_2(S^3\setminus G) \to \mathrm{H}^1(S^3\setminus G)$ the Poincar\'e duality morphism sending a class $[S]$ to the cocycle $[S]\cap \bullet$. 

\begin{definition}(Shadow of a coloring)\label{def_shadow_coloring} The \textit{shadow} $\omega_{\overline{c}}\in \mathrm{H}^1(S^3\setminus G; \mathbb{C}/\mathbb{Z})$ of an admissible coloring $c$ of $G \subset S^3$ is the image of $\overline{c}$ by the composition $\mathrm{H}_1(G; \mathbb{C}/\mathbb{Z}) \xrightarrow{\delta_1} \mathrm{H}_2(S^3\setminus G; \mathbb{C}/\mathbb{Z}) \xrightarrow{\delta_2} \mathrm{H}^1(S^3\setminus G; \mathbb{C}/\mathbb{Z}).$
\end{definition}

Unfolding the definition, $\omega_{\overline{c}}$ has the following down-to-earth description. For $\gamma \subset S^3\setminus G$ an oriented simple closed curve, let $D_{\gamma} \subset S^3$ be an oriented immersed disc positively bounded by $\gamma$ which intersects $G$ transversally. Orient the edges of $G$ arbitrarily. Then 
$$ \omega_{\overline{c}}([\gamma]) = \sum_{v \in D_{\gamma} \cap G} \epsilon(v) \overline{c}(v), $$
where $\epsilon(v)\in \{-1, +1\}$ is the intersection sign and $\overline{c}(v)\in \mathbb{C}/\mathbb{Z}$ is the image by $\overline{c}$ of the oriented edge of $G$ which intersects $D_{\gamma}$ at $v$. 
\par Suppose that $G$ is $\sigma$-decorated and let $\widehat{G}\subset S^3$ be the ribbon graph obtained from $G$ by adding the vertex $v_{\infty}$ and the edges $p_v$ with arbitrary framing. Define an embedding $\mathrm{C}_1(G; \mathbb{Z})\hookrightarrow \mathrm{Z}_1(\widehat{G}; \mathbb{Z})=\mathrm{H}_1(\widehat{G}; \mathbb{Z})$ by sending $c_0 \in \mathrm{C}_1(G; \mathbb{Z})$ such that $\partial c_0=\sum_{v\in V(G)} n_v v$ to $\widehat{c_0}\in\mathrm{Z}_1(\widehat{G}; \mathbb{Z})$ defined by $\widehat{c}_0:= c_0 + \sum_{v} n_v p_v$ (here $p_v$ is oriented from $v$ to $v_{\infty}$). Using the framing of $G$, we can push an edge of $\widehat{G}$ inside $S^3\setminus G$ thus defining a morphism $\mathrm{H}_1(\widehat{G}; \mathbb{Z}) \to \mathrm{H}^1(S^3\setminus G; \mathbb{Z})$ by which we define $\omega_{\overline{c}}(\widehat{c_0}) \in \mathbb{C}/\mathbb{Z}$.

\begin{lemma}\label{lemma_transparent} 
If $c$ is an admissible coloring and $c_0 \in \mathrm{C}_1(G; \mathbb{Z})$, then 
$$ \left<G, c +\frac{N}{k}c_0 \right>= \left< G, c\right> e^{2i\pi \omega_{\overline{c}}(\widehat{c_0})}.$$
\end{lemma}

\begin{proof} 
By Lemma \ref{lemma_almost_transparent}, we have the skein relation $ \adjustbox{valign=c}{\includegraphics[width=0.7cm]{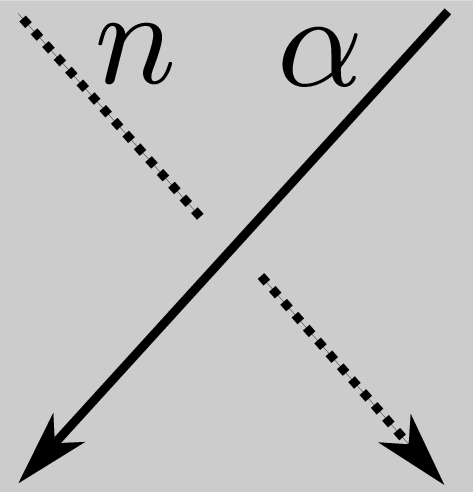}}= e^{2i\pi \alpha n}  \adjustbox{valign=c}{\includegraphics[width=0.7cm]{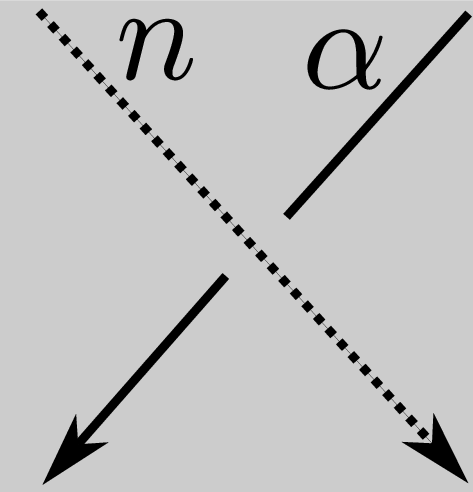}}$ which, together with the fact that $\qdim(\sigma)=1$,  implies the relation:
\begin{equation}\label{eq_skein_sigma}\adjustbox{valign=c}{\includegraphics[width=5cm]{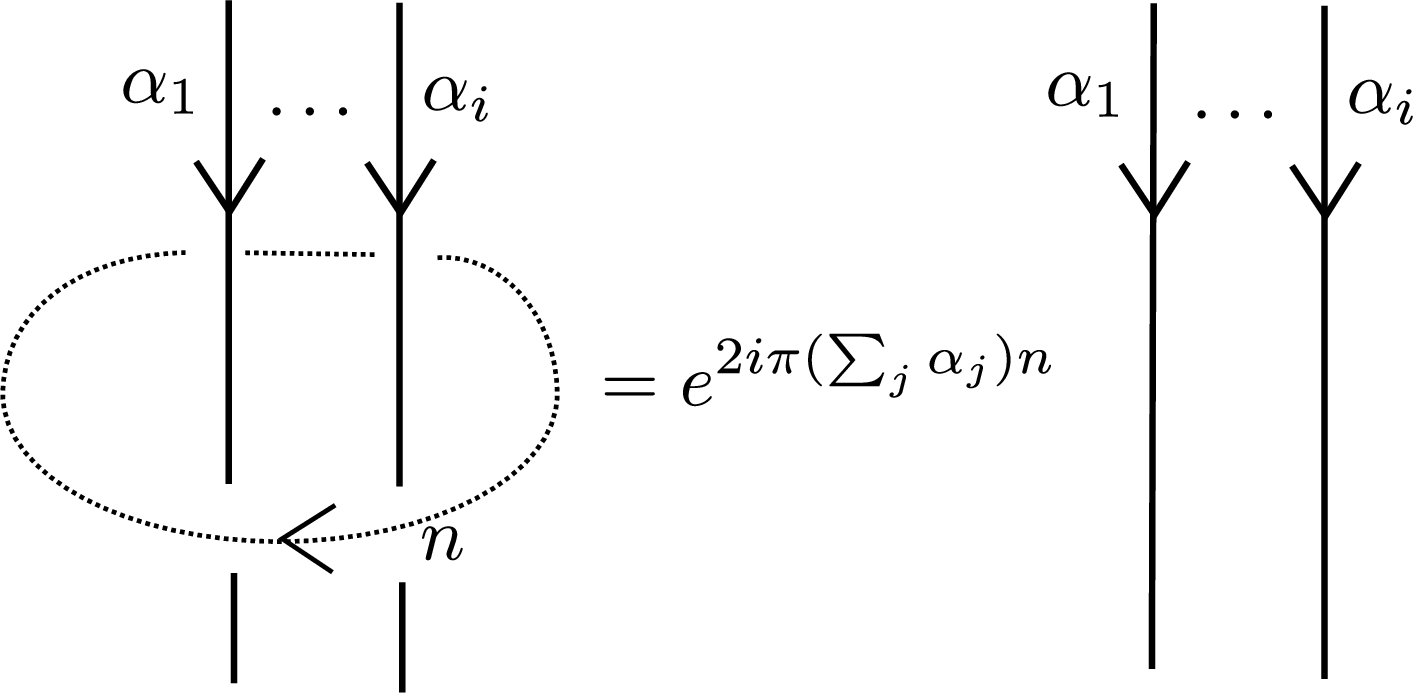}}.\end{equation}

An \textit{inner curve} in the ribbon graph $\widehat{G}$ is a simple closed curve defined by a cycle $v_0 \xrightarrow{e_0} v_1 \xrightarrow{e_1}v_2 \xrightarrow{e_2} \ldots \xrightarrow{e_{n-1}} v_n \xrightarrow{e_n}v_0$ ($v_i$ are vertices, $e_i$ are oriented edges from $v_i$ to $v_{i+1}$) such that the half-edges of $e_i$ and $e_{i-1}$ adjacent to $v_i$ are consecutive in the cyclic ordering at $v_i$ (indexes are modulo $n+1$).
Note that if we thickened the ribbon graph $\widehat{G}$ to a surface $S$ then the inner curves are isotopic to the boundary components of $S$ so the homology classes $[\gamma]$ 
  of inner curves $\gamma$  generate $\mathrm{H}^1(\widehat{G}; \mathbb{Z})$. We can suppose that $\widehat{c_0}=[\gamma]$ for such an inner curve $\gamma$. 
\par As illustrated in Figure \ref{fig_proof_sigma},  $(G, c+ \frac{N}{k}c_0)$ is skein equivalent to the union of $(G, c)$ with the curve $\gamma$ (pushed inside $S^3\setminus G$) colored by $\sigma$. 
Using the skein relation $ \adjustbox{valign=c}{\includegraphics[width=2cm]{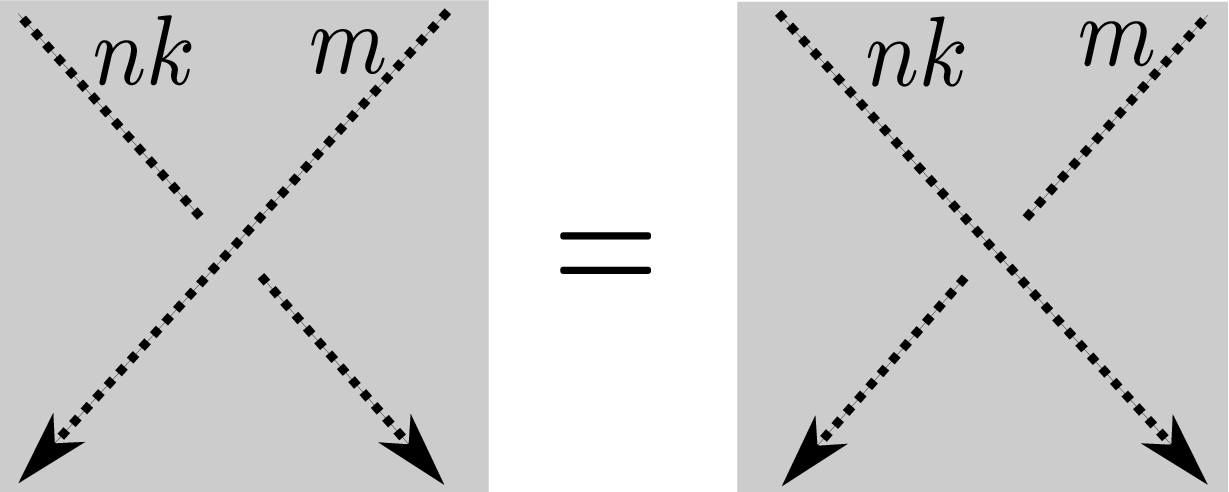}}$ from Lemma \ref{lemma_almost_transparent} we see that $(\gamma, \sigma)$ is transparent with respect to the edges $p_v$ of $\widehat{G}$.
We conclude using the previous skein relation \eqref{eq_skein_sigma}.

\begin{figure}[!h] 
\centerline{\includegraphics[width=12cm]{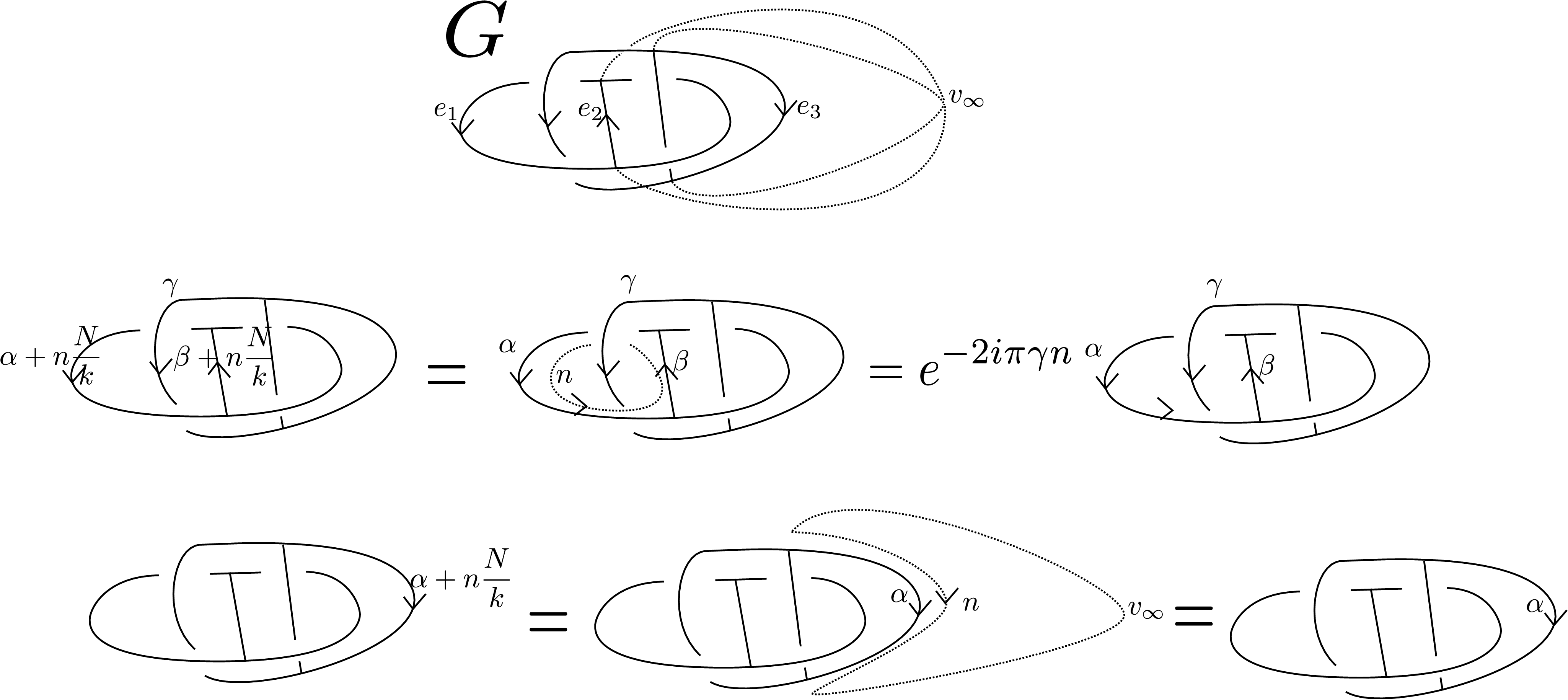} }
\caption{On the top: a ribbon graph $G\subset S^3$ and a $\sigma$ decoration. On the bottom: two illustrations of the equality $ \left<G, c +\frac{N}{k}c_0 \right>= \left< G, c\right> e^{2i\pi \omega_{\overline{c}}(\widehat{c_0})}$ when $c_0= n e_1+ ne_2$ is a cycle so $\widehat{c_0}=c_0$ and when $c_0=n e_3$ (in this case $\widehat{c_0}$ is contractible). 
} 
\label{fig_proof_sigma} 
\end{figure} 
\end{proof}

\par \textbf{The inclusion of the standard module $S_1$}
We now extend our trivalent graph calculus to include trivalent graphs with edges colored by the standard representations $S_1$. We will color these edges in blue.
Recall from Item $(8)$ in Lemma \ref{lemma_representations} that we have morphisms $\iota: S_1\hookrightarrow V_{N-2}$ and $p: V_{N-2}\to S_1$ defined by $\iota(e_n):= v_{N-2+n}$, $p(v_{N-2+n})=e_n$ for $n=0,1$ and $p(v_i)=0$ for $i=0, \ldots, N-3$ (so $p\circ i = \id$). We extend our trivalent graph calculus using the conventions:
$$\adjustbox{valign=c}{\includegraphics[width=10cm]{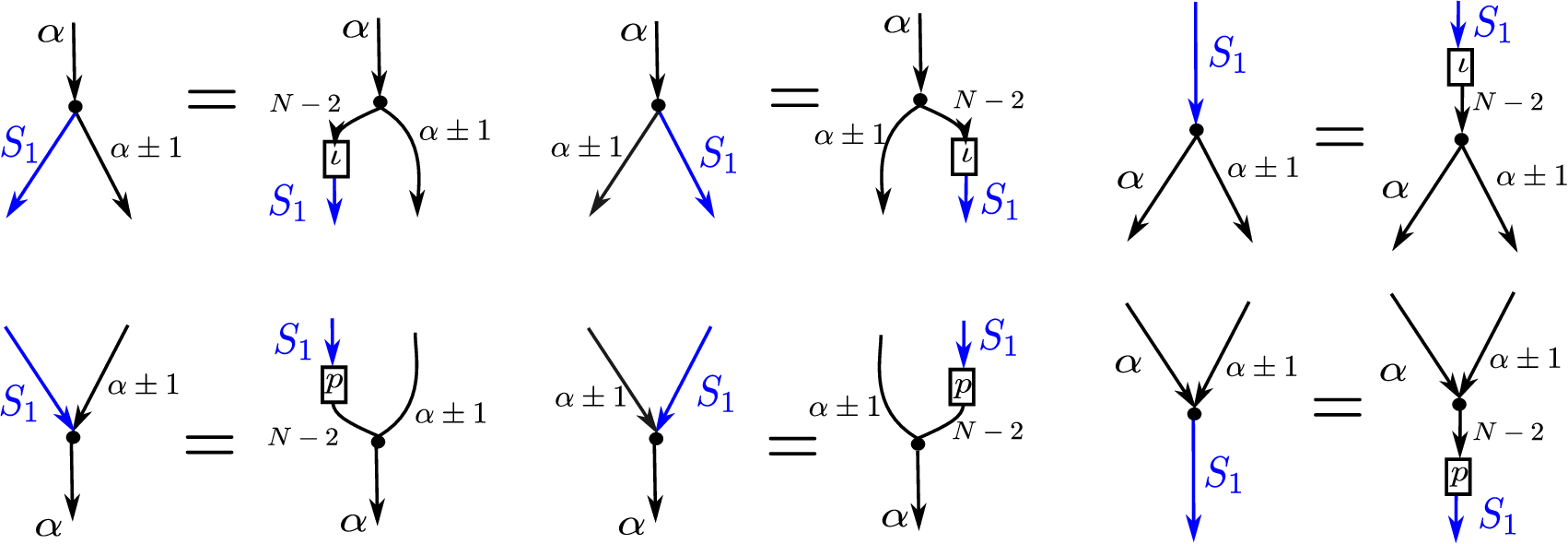}}.$$

\subsection{6j symbols}\label{sec_6j}

\begin{definition} For  $(\alpha, \beta, \gamma)$ $0$-admissible and $\varepsilon_1, \varepsilon_2 \in \{-1, +1\}$, the $6j$-symbol $6S(\alpha, \beta, \gamma; \varepsilon_1, \varepsilon_2)\in \mathbb{C}$ is defined by the skein relation
$$\adjustbox{valign=c}{\includegraphics[width=8cm]{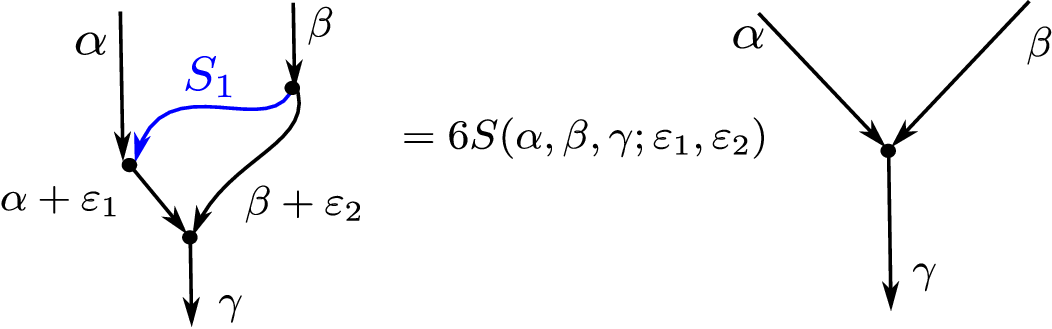}}$$
\end{definition}

The following proposition is the key to prove Theorem \ref{theorem2}. We postpone its (technical) proof to Appendix \ref{sec_appendixB}.

\begin{proposition}\label{prop_6j}
There exists a non-zero rational fraction $R(X_1, X_2, X_3) \in  \mathbb{Q}(A)(X_1, X_2, X_3)$ (which depends on $\varepsilon_1,\varepsilon_2$)  such that 
$$6S(\alpha, \beta, \gamma; \varepsilon_1, \varepsilon_2) = R(q^{\frac{\alpha}{2}}, q^{\frac{\beta}{2}}, q^{\frac{\gamma}{2}}).$$
\end{proposition}

\section{Skein representations coming from non-semisimple TQFTs}\label{sec_representations}

\subsection{The spaces $\mathbb{V}(\Sigma, \omega)$}\label{sec_spaces}

The link invariants described in Section \ref{sec_links} were extended in \cite{CGPInvariants} to invariants of triples $(M, \Gamma, \omega)$ where $M$ is a $3$-manifold equipped with an embedded $\mathcal{C}$-colored ribbon graph $\Gamma \subset M$ with at least one edge colored by a projective module and a cohomology class $\omega \in \mathrm{H}^1(M\setminus \Gamma; \mathbb{C}/\frac{N}{2k}\mathbb{Z})$ compatible with $\Gamma$. These  invariants were proved in \cite{BCGPTQFT} to be part of a more elaborate construction named \textit{extended non-semisimple TQFT} where the authors considered roots of unity of even order. The construction of non-semisimple TQFTs was latter extended to arbitrary modular relative categories in  \cite{DeRenzi_NSETQFT}  and it follows from  \cite[Theorem $1.3$]{DeRenziGeerPatureau_TQFT_QG} that 
 $\mathcal{C}$ is a  relative $\mathbb{C}$-modular category in the sense of \cite{DeRenzi_NSETQFT} so it defines an extended TQFT $\mathbb{V}$. As a byproduct of this elaborated construction, to each surface $\Sigma = \Sigma_g$ with $g\geq 2$ endowed with a cohomology class  $\omega \in \mathrm{H}^1(\Sigma; \mathbb{C}/\mathbb{Z}) \setminus \mathrm{H}^1(\Sigma; \frac{1}{2k}\mathbb{Z}/\mathbb{Z})$, we associate a $N^{3g-3}$ dimensional representation 
$$ r_{\omega} : \mathcal{S}_A(\Sigma) \to \End( \mathbb{V}(\Sigma, \omega) ). $$
 We propose below an alternative construction of this space and of its skein module structure, similar to the construction of the Witten-Reshetikhin-Turaev representations described in Section \ref{sec_WRT}. 
 This will make the paper (almost) self-contained since we will almost never rely on the results in \cite{BCGPTQFT} with the exception of Lemma \ref{lemma_InvPairing}.
 The main difference with the WRT case is the fact that we use a different skein module for the handlebody. The Kauffman-bracket skein modules are particular cases of the following more general class of skein modules.

\vspace{2mm}
\par 

Let $\mathcal{D}$ be a ribbon category and $M$ a compact oriented $3$-manifold. Let $F^{RT}: \Rib_{\mathcal{D}} \to \mathcal{D}$ be the Reshetikhin-Turaev functor where the morphisms in $\Rib_{\mathcal{D}}$ are $\mathcal{D}$-colored standard tangles in the cube $C=[0,1]^3$ with incoming and outgoing endpoints equal to some prescribed points in the sets $\partial_{in}C= [0,1]^2 \times \{0\}$ and $\partial_{out}C= [0,1]^2\times \{1\}$.
Let $\mathbb{C}[\mathcal{RG}_{\mathcal{D}}(M)]$ be the $\mathbb{C}$-vector space freely spanned by $\mathcal{D}$-colored ribbon graphs in $M$. Consider an oriented embedding $\phi : C \hookrightarrow M$. A ribbon graph $\Gamma \subset M$ is $\phi$-adapted if it intersects $\phi(\partial C)$ transversally along $\phi(\partial_{in}C) \cup \phi(\partial_{out}C)$ so that the restriction $\restriction{\Gamma}{\phi(C)}$ defines a morphism $\restriction{\Gamma}{\phi}: V_{in}(\Gamma) \to V_{out}(\Gamma) $ in $\Rib_{\mathcal{D}}$ (so, using Notations \ref{notations_ribC}, $\phi^{-1}(\Gamma) \cap \partial_{in}= P_n\times \{0\}$ and $\phi^{-1}(\Gamma)\cap \partial_{out}=P_m \times \{1\}$)
. A linear combination 
$X= \sum_i x_i \Gamma_i \in \mathbb{C}[\mathcal{RG}_{\mathcal{D}}(M)]$ is \textit{a skein relation with support in} $\phi$ if 
\begin{enumerate}
\item every $\Gamma_i$ are $\phi$ adapted and such that $V_{in}(\Gamma_i) $ and $V_{out}(\Gamma_i)$ do not depend on $i$; 
\item the restriction of the $\Gamma_i$ to $M\setminus \phi(C)$ are pairwise equal; 
\item one has $\sum_i x_i F^{RT}\left( \restriction{\Gamma}{\phi} \right) = 0$.
\end{enumerate}

\begin{definition}(Skein module) \label{def_skeinmodule} The \textit{skein module} $\mathscr{S}^{\mathcal{D}}(M)$ is the quotient of the space $\mathbb{C}[\mathcal{RG}_{\mathcal{D}}(M)]$ by the subspace spanned by skein relations.
\end{definition}

For instance, the Kauffman-bracket skein module is the skein module associated to the Temperley-Lieb category $TL$. As we shall prove in the next subsection, it is closely related to the skein module associated to the full subcategory $\mathcal{C}^{small}\subset \mathcal{C}_{\overline{0}}$ generated by the representations $(S_1^{\pm 1})^{\otimes n}$ where $S_1^{+1}=S_1$ and $S_1^{-1}=(S_1)^*$. By construction, $\mathscr{S}^{\mathcal{D}}(M)$ is functorial in both the variables $\mathcal{D}$ and $M$ in the sense that: 
\begin{enumerate}
\item an oriented embedding $f: M_1\to M_2$ induces a linear map $f_* : \mathscr{S}^{\mathcal{D}}(M_1) \to \mathscr{S}^{\mathcal{D}}(M_2)$ defined by $f_*([\Gamma])=[f(\Gamma)]$, so $f_*$ only depends on the isotopy class of $f$; 
\item a ribbon functor $F: \mathcal{D}_1\to \mathcal{D}_2$ induces a linear map $F_*: \mathscr{S}^{\mathcal{D}_1}(M) \to \mathscr{S}^{\mathcal{D}_2}(M)$ by acting by $F$ on the colors and coupons of ribbon graphs. 
\end{enumerate}
 More precisely, if $\mathcal{M}$ is the category with objects the compact oriented $3$-manifolds and morphisms the oriented embeddings, then $\mathscr{S}^{\mathcal{D}} : \mathcal{M} \to \Vect$ is a functor and $F_* : \mathscr{S}^{\mathcal{D}_1} \to \mathscr{S}^{\mathcal{D}_2}$ is a natural morphism.
 In particular $\mathscr{S}^{\mathcal{D}}(\Sigma):= \mathscr{S}^{\mathcal{D}}(\Sigma \times [0,1])$ has an algebra structure and if $\partial M = \Sigma$ then $\mathscr{S}^{\mathcal{D}}(M)$ is a $\mathscr{S}^{\mathcal{D}}(\Sigma)$-module (since $\Sigma\times [0,1]$ is an algebra object in $\mathcal{M}$ and $M$ a module object over it).
 
Moreover, the inclusion functor $\mathcal{C}^{small} \subset \mathcal{C}_{\overline{0}}$ induces an algebra morphism $\iota: \mathscr{S}^{\mathcal{C}^{small}}(\Sigma) \to \mathscr{S}^{\mathcal{C}_{\overline{0}}}(\Sigma)$ and if $\partial M =\Sigma$ then $\mathscr{S}^{\mathcal{C}}(M)$ is a $\mathscr{S}^{\mathcal{C}^{small}}(\Sigma)$-module. For now on, we will denote by $\mathcal{S}_A^{QG}$ the functor $\mathscr{S}^{\mathcal{C}^{small}}$ for simplicity.

\par Let us now suppose that the ribbon category $\mathcal{D}$ is graded by an abelian group $\mathcal{A}$, i.e. that $\mathcal{D}=\cup_{g\in \mathcal{A}} \mathcal{D}_g$ with $\mathcal{D}_{g_1}\otimes \mathcal{D}_{g_2} \subset \mathcal{D}_{g_1+g_2}$. The category $\Rib_{\mathcal{D}}$ is then $\mathcal{A}$-graded as well where the object $( (V_1, \epsilon), \ldots, (V_n, \epsilon_n)) \in \Rib_{\mathcal{D}}$ belong to the graded part associated to $g= \sum_{i=1}^n \epsilon_i g_i$ where $V_i \in \mathcal{D}_{g_i}$ and $\epsilon_i\in \{\pm 1\}$. By construction, the Reshetikhin-Turaev functor $F^{RT}: \Rib_{\mathcal{D}} \to \mathcal{D}$ preserves the grading (has degree $0$). 

\par Let $G\subset \mathbb{S}^3$ be an oriented ribbon trivalent graph embedded in the sphere, let $V$ and $E$ be its sets of vertices and edges respectively and let $H= N(G)$ be a tubular neighborhood, so $H$ is a handlebody with boundary $\Sigma:= \partial H$. For each edge $e\in E$, one associates a properly embedded oriented disc $D_e \subset H$ with $\gamma_e := \partial D_e \subset \Sigma$ chosen such that $D_e$ intersects $G$ in a single point $p_e$ in the center of $D_e$. The framing of $G$ at $p_e$ defines an oriented axis of $D_e$ passing by $p_e$. An embedded ribbon graph $\Gamma \subset H$ is in good position if it intersects each $D_e$ transversally along the oriented axis with framing given by the oriented axis. So the intersection $\Gamma \cap D_e$ defines an element $X_{\Gamma, e} \in \Rib_{\mathcal{D}}$ and we denote by $g_e(\Gamma) \in A$ its grading. Any ribbon graph can be isotoped such that it is in good position and two isotopic ribbon graphs in good position define the same graded elements $g_e(\Gamma)$ (note that the fact that $A$ is abelian is important here). Note also that for $v\in V$ corresponding to a pair of pants $P_v$ with adjacent edges $e_1, e_2, e_3$ (not necessarily pairwise distinct, so it corresponds rather to half-edges), writing $\epsilon_i =+1$ if $e_i$ is incoming to $v$ and $\epsilon_i=-1$ else, then $\epsilon_1g_{e_1}(\Gamma)+\epsilon_2g_{e_2}(\Gamma) +\epsilon_3 g_{e_3}(\Gamma)=0$. Therefore the $g_e(\Gamma)$ define a cycle $g(\gamma) \in \mathrm{Z}^1(G; \mathcal{A})\cong \mathrm{H}_1(H; \mathcal{A})$.
Therefore, the skein module $\mathscr{S}^{\mathcal{D}}(H)$ is $\mathrm{H}^1(H; \mathcal{A})$-graded. For $\omega_H \in \mathrm{H}^1(H; \mathcal{A})$, we denote by $\mathscr{S}^{\mathcal{D}}(H, \omega_H)$ the $\omega_H$ graded part.

\par Now consider the two  trivalent ribbon graphs $G_g, G'_g \subset \mathbb{S}^3$ of Figure \ref{fig_graphs} and choose the tubular neighborhoods $H_g=N(G_g)$ and $H_g'=N(G_g')$ such that $H_g\cap H_g'= \partial H_g = \partial H_g' =: \Sigma_g$ has genus $g$. 
Note that the graphs have been chosen such that for any edge $e$ of either $G_g$ or $G'_g$, then the curve $\gamma_e \subset \Sigma_g$ is non-separating. 
Let $\omega \in \mathrm{H}^1(\Sigma_g; \mathbb{C}/\mathbb{Z}) \setminus \mathrm{H}^1(\Sigma_g; \frac{1}{2k}\mathbb{Z}/\mathbb{Z})$. By pulling back $\omega$ along the  inclusion maps $\Sigma_g \subset H_g$ and $\Sigma_g \subset H'_g$, we obtain two classes $\omega_{H_g} \in \mathrm{H}^1(H_g; \mathbb{C}/\mathbb{Z})$ and $\omega_{H'_g} \in \mathrm{H}^1(H'_g; \mathbb{C}/\mathbb{Z})$ and the pair $(\omega_{H_g}, \omega_{H_g'})$ determines $\omega$. Consider the Hopf pairing 
$$ \left( \cdot, \cdot \right)^H_{\omega} : \mathscr{S}^{\mathcal{C}}(H'_g, \omega_{H'_g}) \otimes \mathscr{S}^{\mathcal{C}}(H_g, \omega_{H_g}) \to \mathbb{C}, \quad (\Gamma', \Gamma)^H_\omega := \left< \Gamma' \cup \Gamma \right>, $$
where $\left< \Gamma' \cup \Gamma \right>$ is the link invariant of Section \ref{sec_links} defined using the renormalized trace of the ribbon graph $\Gamma \cup \Gamma' \in \mathcal{RG}_{\mathcal{P}}(\mathbb{S}^3)$ (the condition $\omega \notin \mathrm{H}^1(\Sigma_g; \frac{1}{2k}\mathbb{Z}/\mathbb{Z})$ implies that at least one edge of $\Gamma \cup \Gamma'$ is colored by a projective module).

\begin{figure}[!h] 
\centerline{\includegraphics[width=12cm]{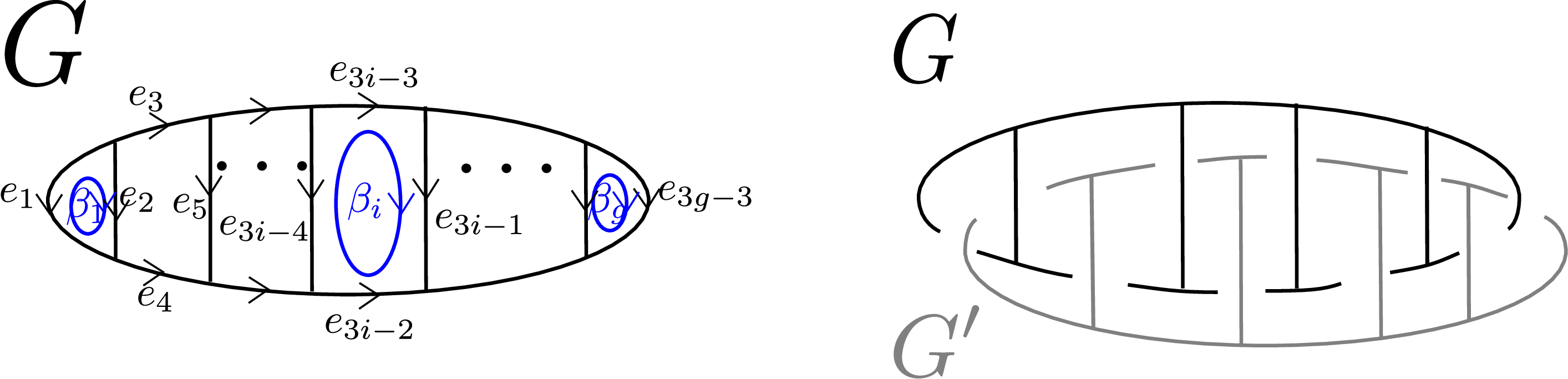} }
\caption{On the left: a trivalent oriented graph $G$, an enumeration $e_1,\ldots, e_{3g-3}$ of its edges and some curves $\beta_1, \ldots, \beta_g$. On the right: two copies $G$ and $G'$ of the graph embedded in $\mathbb{S}^3$ which define a Heegaard splitting of the sphere.} 
\label{fig_graphs} 
\end{figure}

\begin{definition}\label{def_spaceV} The space $\mathbb{V}(\Sigma_g; \omega)$ is defined as the quotient 
$$ \mathbb{V}(\Sigma_g; \omega):= \quotient{ \mathscr{S}^{\mathcal{C}}(H_g, \omega_{H_g})}{\Rker \left( \cdot , \cdot \right)_{\omega}^H}.$$
\end{definition}
We also define the space
$$ \mathbb{V}'(\Sigma_g; \omega):= \quotient{ \mathscr{S}^{\mathcal{C}}(H'_g, \omega_{H'_g})}{\LKer \left( \cdot , \cdot \right)_{\omega}^H}$$
so the Hopf pairing $(\cdot, \cdot)^H :\mathbb{V}'(\Sigma_g, \omega) \otimes \mathbb{V}(\Sigma_g, \omega) \to \mathbb{C}$ is non-degenerate.
\par 
The skein module $\mathscr{S}^{\mathcal{C}}(\Sigma_g)$ naturally acts on the left on $\mathscr{S}^{\mathcal{C}}(H_g)$
and on the right on $\mathscr{S}^{\mathcal{C}}(H'_g)$ 
 and the Hopf pairing is invariant for this action (i.e. $(x, z\cdot y)^H=(x\cdot z, y)^H$). The action of the submodule $\mathscr{S}^{\mathcal{C}_{\overline{0}}}(\Sigma_g) \subset \mathscr{S}^{\mathcal{C}}(\Sigma_g)$ preserves the graded subspaces $\mathscr{S}^{\mathcal{C}}(H_g, \omega_{H_g})$ and $\mathscr{S}^{\mathcal{C}}(H'_g, \omega_{H_g'})$, so we obtain a left action of $\mathscr{S}^{\mathcal{C}_{\overline{0}}}(\Sigma_g) $ on $\mathbb{V}(\Sigma_g, \omega)$ by quotient. 
 (and a right action on $\mathbb{V}'(\Sigma_g, \omega)$).

\begin{definition}\label{def_rep_twisted} The  representation $\overline{r}_{\omega} : \mathcal{S}_A^{QG}(\Sigma_g) \to \End(\mathbb{V}(\Sigma_g, \omega))$ is defined as the composition
$$ \overline{r}_{\omega} : \mathcal{S}_A^{QG}(\Sigma_g) \xrightarrow{\iota} \mathscr{S}^{\mathcal{C}_{\overline{0}}}(\Sigma_g) \to  \End(\mathbb{V}(\Sigma_g, \omega)),$$
 where $\mathcal{S}_A^{QG}$ denotes the functor $\mathscr{S}^{\mathcal{C}^{small}}$.
\end{definition}

\subsection{Relating Kauffman-bracket and quantum groups skein modules}

We now define an algebra morphism $f : \mathcal{S}_A(\Sigma) \to  \mathcal{S}_A^{QG}(\Sigma)$. In this subsection, we do not assume $\Sigma$ to be closed, nor do we assume that $A$ is a root of unity (the results work over the ring $\mathbb{Z}[A^{\pm 1}]$ for instance). This task is not trivial due to the fact that the Temperley-Lieb category $TL$ defining the Kauffman-bracket skein algebra is not ribbon equivalent to $\mathcal{C}^{small}$. As detailed in \cite{Tingley_MinusSign}, there is an equivalence of categories $\mathrm{Cauchy}(TL)\to \mathcal{C}^{small}$ between the Cauchy closure of $TL$ and $\mathcal{C}^{small}$ which is a braided functor preserving the pivotal structure but which does not preserve the twist. Said differently, $\mathrm{Cauchy}(TL)$ is equivalent to $\mathcal{C}^{small}$ with a "modified twist" responsible for the fact that the class of an unknot is sent to $-(q+q^{-1})$ in $\mathcal{S}_A(\Sigma) $ whereas the q-dimension of $S_1$ is $\qdim(S_1)=+(q+q^{-1})$.
\par 
First note that since $S_1$ is isomorphic to its dual $(S_1)^*$, the class $[L(S_1)]\in \mathscr{S}^{\mathcal{C}^{small}}(\Sigma)$ of  an oriented framed link $L\subset \Sigma\times [0,1]$ with all its components colored by $S_1$, does not depend on the orientations of the components of $L$, so $[L(S_1)]$ is well-defined for unoriented framed links $L$. We use a similar notation $\gamma(S_1)$ for $\gamma$ a multicurve.

\begin{theorem}\label{theorem_skeinQG} Let $S$ be a spin structure on $\Sigma$ with Johnson quadratic form $w_S$. 
There is an algebra morphism $f:  \mathcal{S}_A(\Sigma) \to  \mathcal{S}_A^{QG}(\Sigma)$ sending the class $[\gamma]$ of a multicurve to $(-1)^{|\gamma| + w_S([\gamma])} [\gamma(S_1)]$, where $|\gamma|$ is the number of components of $\gamma$. 
\end{theorem}

\begin{proposition}\label{prop_skeinQG} Let $M$ be a $3$-manifold. We have a linear morphism $h: \mathcal{S}_{-A}(M) \to \mathcal{S}_A^{QG}(M)$ defined by $h([L])=(-1)^{| L |} [L(S_1)]$ for any framed link $L \subset M$. Moreover $h$ is an algebra morphism when $M$ is a thickened surface.
\end{proposition}

\begin{proof} By definition,  $\mathcal{S}_{-A}(M)$ is defined as the quotient of the module $\mathbb{C}[ L\subset M]$ freely generated by isotopy classes of framed links by a submodule  of skein relations. 
Let $h: \mathbb{C}[ L\subset M] \to \mathcal{S}^{QG}_A(M)$ be the linear morphism defined by $h([L]):= (-1)^{| L|} [L(S_1)]$ and let us prove that the images $h([L])$ satisfy the Kauffman bracket skein relations of $\mathcal{S}_{-A}(M)$. For $L\subset M$ and $U\subset M$ an unknot bounding a disc not intersecting $L$, since $\qdim(S_1)=q+q^{-1}$, we find 
$$h( [L \cup U])= (-1)^{| L | +1} [L(S_1)] \qdim(S_1) = (-q-q^{-1}) (-1)^{|L|} [L(S_1)]= -(q+q^{-1})h([L]).$$
For the other skein relation, let $f:S_1\xrightarrow{\cong} (S_1)^*$ be the equivariant isomorphism defined by $f(e_0):=e_1^*$ and $f(e_1):= -q e_0^*$. We claim that in $ \mathcal{S}^{QG}(\Sigma)$ the following skein relation is satisfied: 

\begin{equation}\label{eq_twist_skein}
\adjustbox{valign=c}{\includegraphics[width=1cm]{CrossingDown.eps}} = A \adjustbox{valign=c}{\includegraphics[width=1cm]{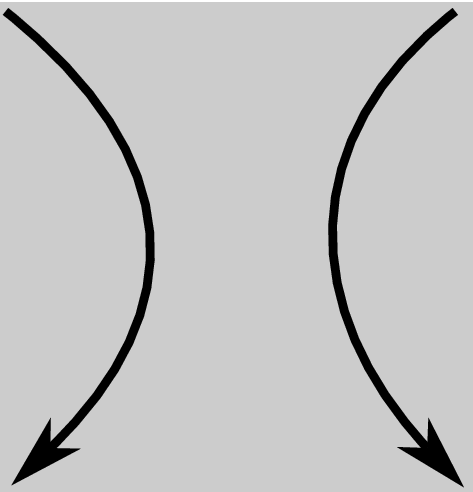}} + A^{-1} \adjustbox{valign=c}{\includegraphics[width=1cm]{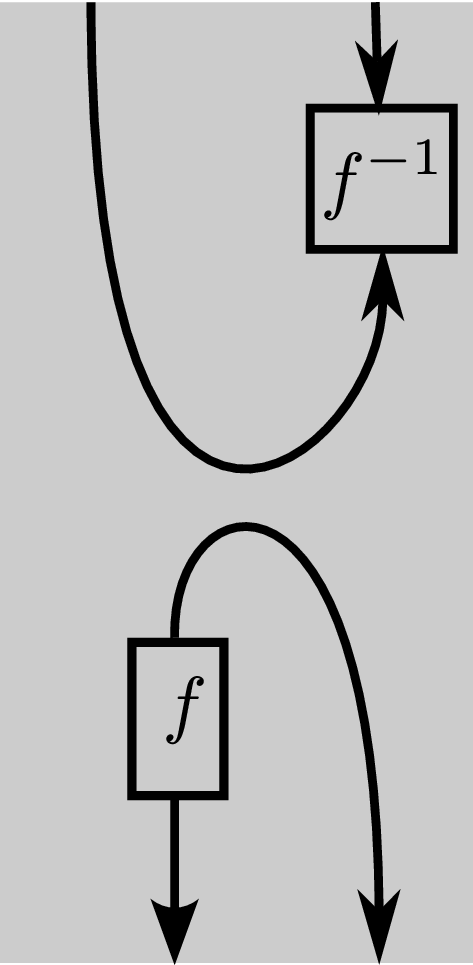}}
\end{equation}
The claim is well-known (see e.g. \cite[Appendix $H.1$]{OhtsukiBook}) and follows from the following computation. First we compute: 
$$ \overrightarrow{ev}_{S_1} (f \otimes \id) : e_i \otimes e_j \mapsto \left\{ 
\begin{array}{ll} 
0 & \mbox{, if }i=j, \\
1 & \mbox{, if }(i,j)=(0,1), \\
-q & \mbox{, if }(i,j)=(1,0).
\end{array} \right. 
\quad 
(\id \otimes f^{-1})\overrightarrow{coev}_{S_1} : 1 \mapsto -q^{-1} e_0 \otimes e_1 + e_1\otimes e_0.
$$
From which we deduce that in the ordered basis $B= (e_0 \otimes e_0, e_0 \otimes e_1, e_1\otimes e_0, e_1\otimes e_1 )$, we have the matrices 
$$ \mathcal{U}:= F^{RT} \left( \adjustbox{valign=c}{\includegraphics[width=1cm]{Cupcap.eps}} \right) = 
\begin{pmatrix} 
0 & 0 & 0 & 0 \\
0 & -q^{-1} & 1 & 0 \\
0 & 1 & -q & 0 \\
0 & 0 & 0 & 0
\end{pmatrix}
, \quad 
\mathscr{R} := F^{RT}\left(  \adjustbox{valign=c}{\includegraphics[width=1cm]{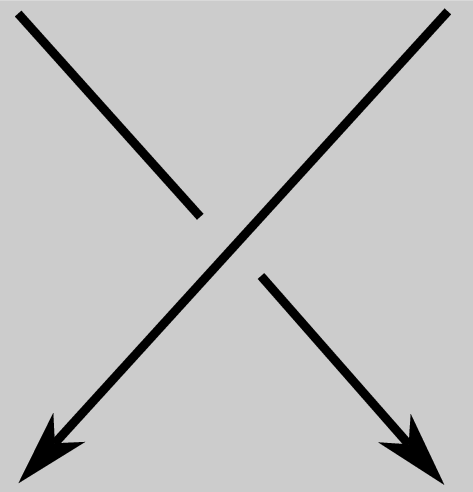}} \right) =
\begin{pmatrix}
 q^{1/2} & 0 & 0 & 0 \\
  0 & 0 & q^{-1/2} & 0 \\
  0 & q^{-1/2} & q^{1/2}-q^{-1/2} & 0 \\
  0 & 0 & 0 & q^{1/2}
  \end{pmatrix}
  $$
The skein relation \eqref{eq_twist_skein} now follows from the equality $\mathscr{R} = q^{1/2} \mathds{1}_4 + q^{-1/2} \mathcal{U}.$
\par Let $L_0, L_1, L_2 \subset M$ be three framed links which are equal outside a ball $B\subset M$ inside which they look like  $\Crosspos$,
   $\begin{tikzpicture}[baseline=-0.4ex,scale=0.5,>=stealth] 
\draw [fill=gray!45,gray!45] (-.6,-.6)  rectangle (.6,.6)   ;
\draw[line width=1.2] (-0.4,-0.52) ..controls +(.3,.5).. (-.4,.53);
\draw[line width=1.2] (0.4,-0.52) ..controls +(-.3,.5).. (.4,.53);
\end{tikzpicture}$ and 
 $ \begin{tikzpicture}[baseline=-0.4ex,scale=0.5,rotate=90]	
\draw [fill=gray!45,gray!45] (-.6,-.6)  rectangle (.6,.6)   ;
\draw[line width=1.2] (-0.4,-0.52) ..controls +(.3,.5).. (-.4,.53);
\draw[line width=1.2] (0.4,-0.52) ..controls +(-.3,.5).. (.4,.53);
\end{tikzpicture}$
 respectively. Let us prove that $h(L_0 +AL_1+A^{-1}L_2)=0$. 
\par \underline{Case 1:} Suppose that the two brands in the crossing $\Crosspos$ belong to two distinct components, i.e. that  $L_0= \adjustbox{valign=c}{\includegraphics[width=1.5cm]{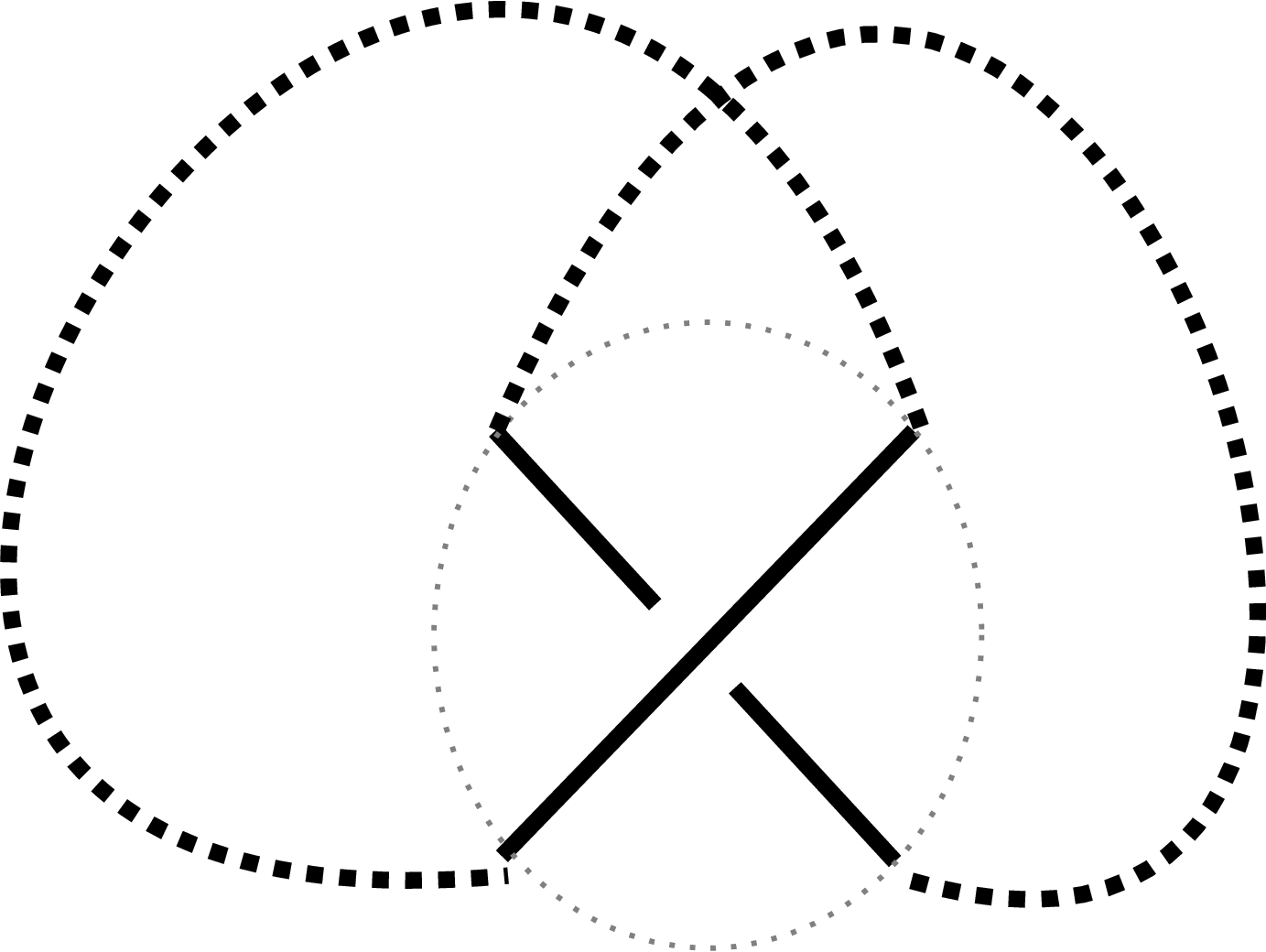}}$. In this case $|L_1|=|L_2|=|L_0|+1$ so 
\begin{multline*} h([L_0]) = (-1)^{|L_0|} \adjustbox{valign=c}{\includegraphics[width=1.5cm]{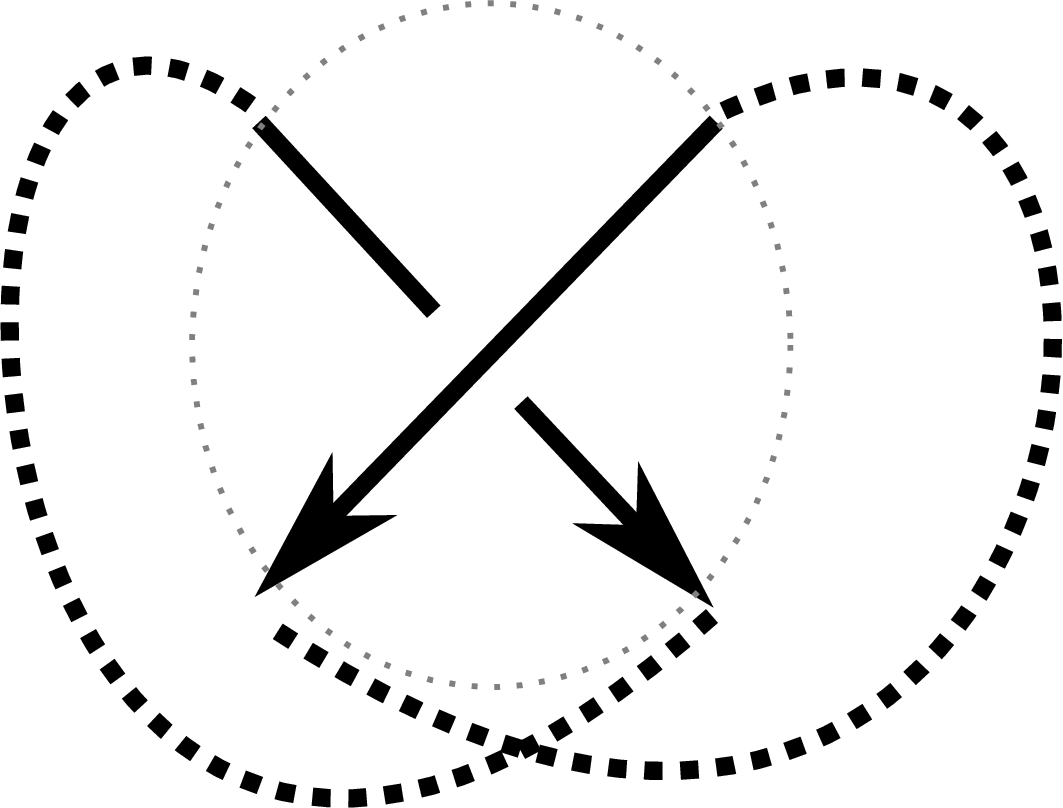}} = (-A) (-1)^{|L_1|} \adjustbox{valign=c}{\includegraphics[width=1.5cm]{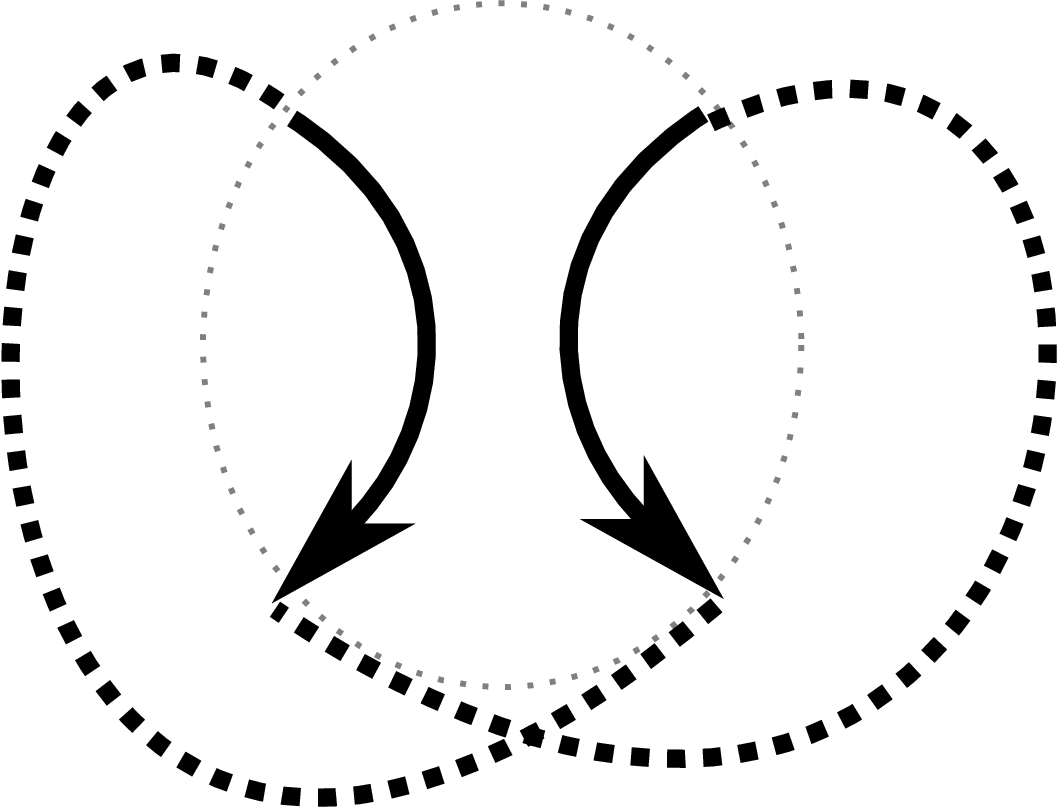}}- A^{-1}(-1)^{|L_2|} \adjustbox{valign=c}{\includegraphics[width=1.5cm]{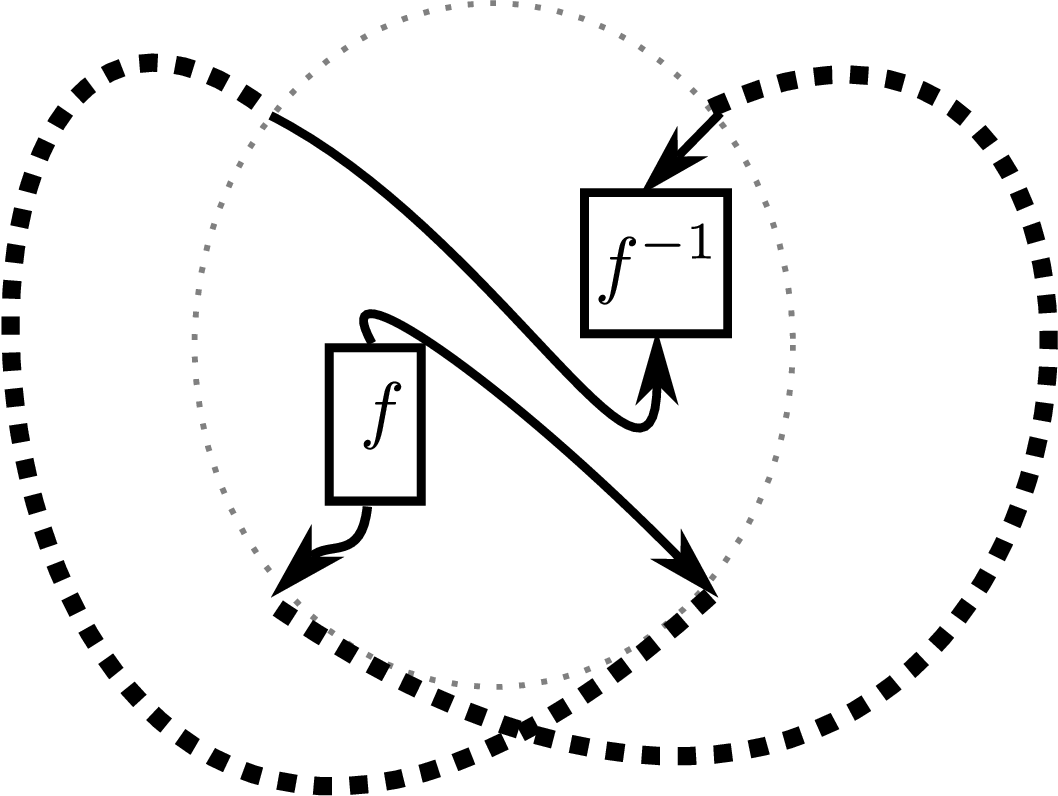}}
\\  = (-A) (-1)^{|L_1|} \adjustbox{valign=c}{\includegraphics[width=1.5cm]{72.eps}} - A^{-1}(-1)^{|L_2|} \adjustbox{valign=c}{\includegraphics[width=1.5cm]{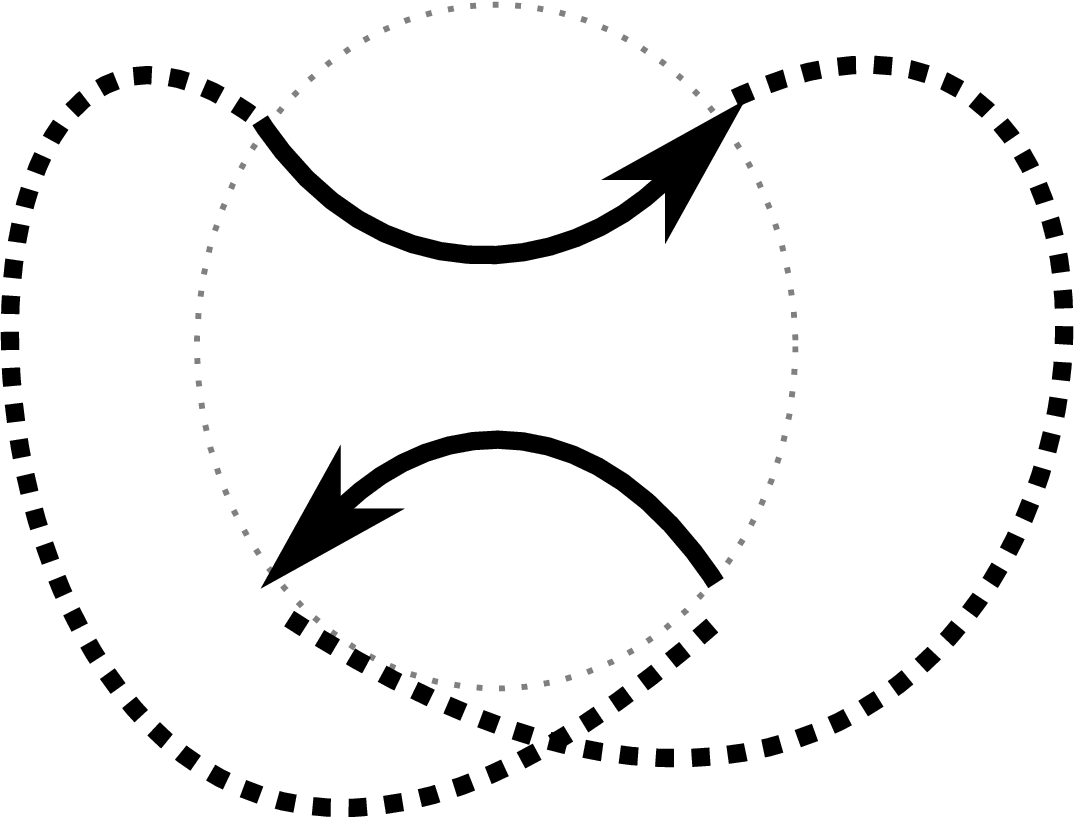}}=-A h([L_1]) - A^{-1} h([L_2]),
 \end{multline*}
where we pushed the box colored by $f^{-1}$ along the strand and canceled it with the box colored by $f$ using a skein relation while passing from the first to the second line.
\par \underline{Case 2:} If the two brands in the crossing $\Crosspos$ belong to the same component, i.e. if $L_0= \adjustbox{valign=c}{\includegraphics[width=1.5cm]{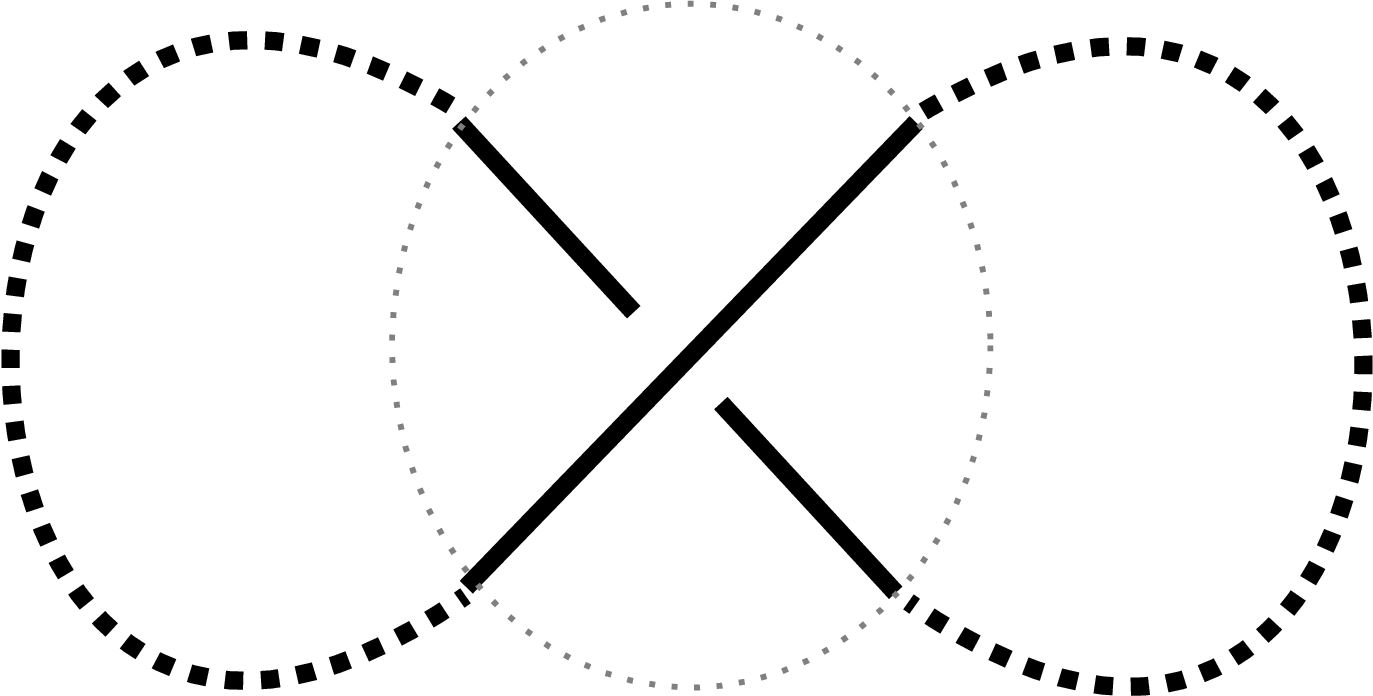}}$ or  $L_0= \adjustbox{valign=c}{\includegraphics[width=0.5cm]{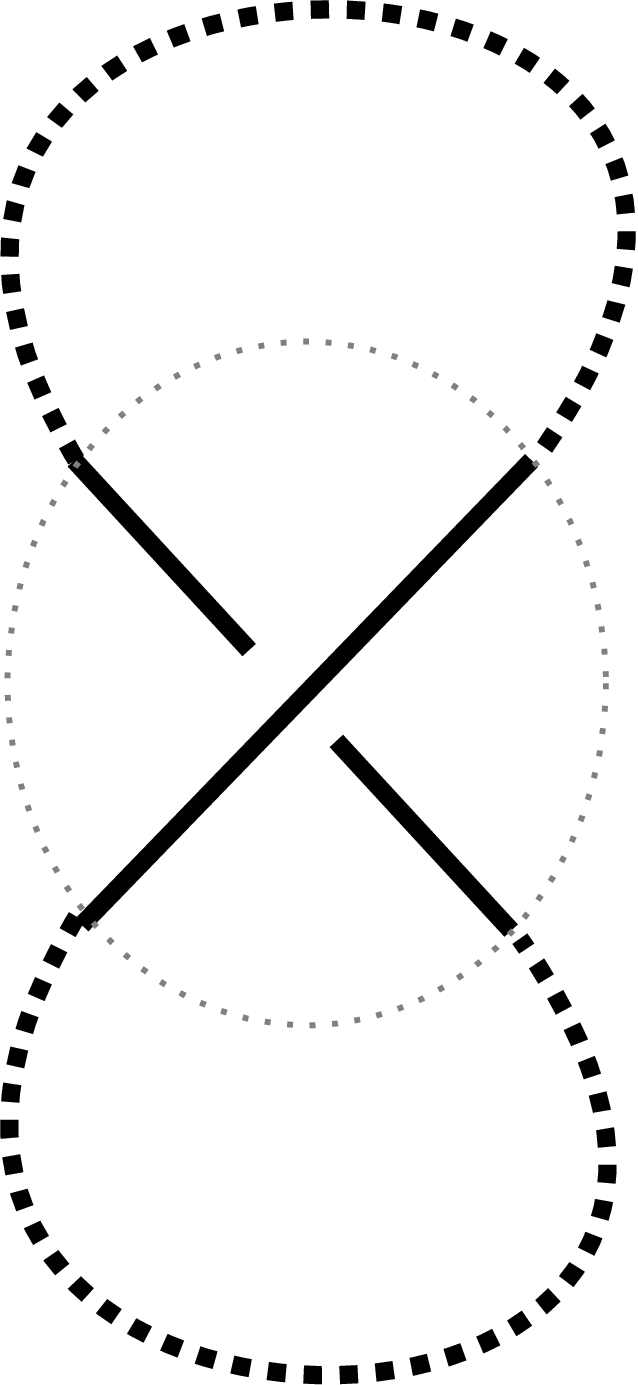}}$ we get 
$$ - A h\left(
 \begin{tikzpicture}[baseline=-0.4ex,scale=0.5,>=stealth] 
\draw [fill=gray!45,gray!45] (-.6,-.6)  rectangle (.6,.6)   ;
\draw[line width=1.2] (-0.4,-0.52) ..controls +(.3,.5).. (-.4,.53);
\draw[line width=1.2] (0.4,-0.52) ..controls +(-.3,.5).. (.4,.53);
\end{tikzpicture}
\right) = -A h \left(  \adjustbox{valign=c}{\includegraphics[width=0.7cm]{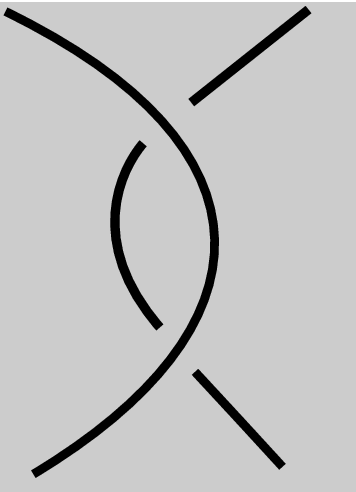}}\right) \stackrel{\mbox{by Case 1}}{=} A^2 h\left( \adjustbox{valign=c}{\includegraphics[width=0.7cm]{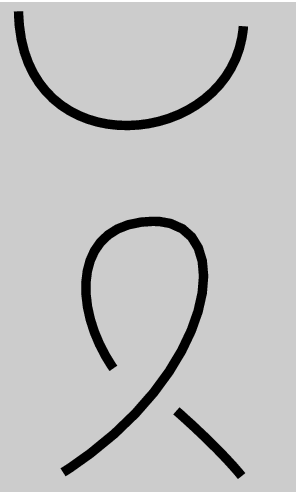}} \right) +h\left(\Crosspos\right) = 
A^{-1} h\left( 
\begin{tikzpicture}[baseline=-0.4ex,scale=0.5,rotate=90]	
\draw [fill=gray!45,gray!45] (-.6,-.6)  rectangle (.6,.6)   ;
\draw[line width=1.2] (-0.4,-0.52) ..controls +(.3,.5).. (-.4,.53);
\draw[line width=1.2] (0.4,-0.52) ..controls +(-.3,.5).. (.4,.53);
\end{tikzpicture}
\right) + h\left(\Crosspos\right).
$$
Where we used the fact that $\theta_{S_1}=+A^3 \id_{S_1}$ in $\mathcal{C}$ in the last equality. So we get again the equality $h([L_0] +A [L_1]+A^{-1}[L_2])=0$. This concludes the proof.

\end{proof}

\begin{proof}[Proof of Theorem \ref{theorem_skeinQG}]
The theorem follows from Propositions \ref{prop_Barett} and \ref{prop_skeinQG}. 
\end{proof}

\begin{definition}\label{def_BCGP_rep} The representation $r_{\omega}: \mathcal{S}_A(\Sigma_g) \to \End(\mathbb{V}(\Sigma_g, \omega))$ is the composition $r_{\omega}= \overline{r}_{\omega} \circ f$ where $\overline{r}_{\omega}$ is defined in Definition \ref{def_rep_twisted}.
\end{definition}

\subsection{Bases and classical shadows}

Given $c \in \mathrm{C}_1(G; \mathbb{C})$ and $c' \in \mathrm{C}_1(G'; \mathbb{C})$ admissible, we want to consider $\mathcal{C}$-colored graphs $(G,c) \subset H_g$ and $(G',c') \subset H'_g$ so that the vector $[G,c] \in \mathbb{V}(\Sigma_g, \omega)$ is well-defined and so is the pairing $\left< [G,c], [G',c'] \right>=\left< (G\cup G', c\cup c')\right>$ (recall that $\left< (G\cup G', c\cup c')\right>$ is the CGP invariants of the colored graph $ (G\cup G', c\cup c')$) . Recall from Section \ref{sec_links} that to define such $\mathcal{C}$-colored graphs, we need to specify a $\sigma$-decoration. We thus fix a $\sigma$-decoration of $G\cup G'$ such that $(1)$ $v_{\infty} \in \Sigma_g \subset S^3\setminus (G\cup G')$ and $(2)$ for $v\in V(G)$ (resp. $v' \in V(G')$) the path $p_v: v\to v_{\infty}$ (resp. $p_{v'}$) lies inside $H_g$ (resp. inside $H_g'$). This condition permits one to define the $\mathcal{C}$-colored graphs $(G,c)\subset H_g$, $(G',c')\subset H_g'$ and $(G\cup G', c\cup c') \subset S^3$. 
\vspace{2mm}
\par 
We now suppose that $\omega$ belongs to the open dense subset $$\mathcal{U}= \{ \omega \in \mathrm{H}^1(\Sigma_g; \mathbb{C}/\mathbb{Z}) \mbox{ such that } \omega(\gamma_e)\notin \left(\frac{1}{2k}\mathbb{Z}\right)/\mathbb{Z}, \forall e \in E(G)\cup E(G')\}.$$
 Recall that the embedded graphs $G\cup G' \subset S^3$ have been chosen in Figure \ref{fig_graphs} such that for any edge $e \in E(G)\cup E(G')$ then $\gamma_e \subset \Sigma_g$ is non-separating. This ensures that indeed $\mathcal{U}$ is non empty.

For $c \in \mathrm{C}_1(G; \mathbb{C})$ such that $\overline{c(e)}=\omega(\gamma_e)$ for all edge $e\in E(G)$, we say that $c$ is $\omega$\textit{-compatible} and denote by  $[G,c]\in \mathbb{V}(\Sigma_g, \omega)$ the corresponding vector. Similarly, for $c'\in \mathrm{C}_1(G; \mathbb{C})$ such that $\overline{c'(e)}=\omega(\gamma_e')$ for all $e\in E(G')$, $c'$ is said $\omega$\textit{-compatible} and we define $[G',c'] \in \mathbb{V}'(\Sigma_g, \omega)$. 

\begin{lemma}\label{lemma_span_space} The space $\mathbb{V}(\Sigma_g, \omega)$ is spanned by the vectors $[G,c]$ for $c$ $\omega$-compatible. 
Similarly, $\mathbb{V}'(\Sigma_g, \omega)$ is spanned by vectors $[G', c']$ with $c'$ $\omega$-compatible.
\end{lemma}

\begin{proof}
We prove the assertion concerning $\mathbb{V}(\Sigma_g, \omega)$, the second assertion is proved similarly.
 Let $[\Gamma] \in \mathbb{V}(\Sigma_g, \omega)$ be the class of a $\mathcal{C}$-colored ribbon graph $\Gamma \subset H_g$ which induces the class $\omega$. We suppose that $\Gamma$ is isotoped such that it intersects each disc $D_e$ transversally along the axis given by the framing. As before, the intersection $\Gamma\cap D_e$ defines an element $X_{\Gamma,e}\in \Rib_{\mathcal{C}}$ sent to an element of $\mathcal{C}_{\omega(e)}$ by the Reshetikhin-Turaev functor. By assumption, $\omega(e)\notin \frac{1}{2k}\mathbb{Z} /\mathbb{Z}$ so $\mathcal{C}_{\omega(e)}$ is semisimple with simple objects the $V_{\alpha}$ with $\overline{\alpha} = \omega(e)$. So 
$$F^{RT}(X_{\Gamma, e})\cong \oplus_{\alpha \in \omega(e)} n_{\alpha}V_{\alpha}$$
for some integers $n_{\alpha}$. 
 Using a skein relation in the neighborhood of each $\mathbb{D}_e$ which replaces the identity strand $\id_{X_{\Gamma, e}}$ by $\sum_{\alpha}n_{\alpha} \id_{(V_{\alpha}, +)}$, we see that $[\Gamma]$ is a linear combination of classes $[\Gamma_i]$ such that each $\Gamma_i$ intersects $D_e$ once exactly along an edge colored by a module $V_{\alpha_e}$ with $\overline{\alpha_e} = {\omega(e)}$. 
 \par While cutting $H_g$ along the discs $\mathbb{D}_e$, we get a collection of balls $B_v$ indexed by $v\in V(G)$. If $e,f,g$ are the three (non-necessarily pairwise distinct)  adjacent edges to $v$, the boundary of $B_v$ is the union of the pair of pants $P_v$ and the discs $\mathbb{D}_e, \mathbb{D}_f, \mathbb{D}_g$. Let $n_v\in \mathbb{Z}$ be the unique integer such that $\alpha_e+\alpha_f +\alpha_g \in \frac{N}{k}n_v + H_N$.
Since the multiplicity module $\Hom(V_{\alpha_e}\otimes V_{\alpha_f} \otimes V_{\alpha_g}, \sigma^n)$ is one-dimensional generated by $H_{n_v}^{\alpha_e, \alpha_f, \alpha_g}$, using a skein relation inside each ball $B_v$, we see that $[\Gamma_i]$ is proportional to a vector $[G,c]$ for  $c$ $\omega$-compatible.

\end{proof}

\begin{lemma}\label{lemma_modulo}  For $c_0 \in C_1(G; \mathbb{Z})$, we have $[G,c+ Nc_0]= q^{ N \omega(\widehat{c_0})} [G,c]$ in $\mathbb{V}(\Sigma_g, \omega)$.
\end{lemma}

\begin{proof}
Let $c' \in \mathrm{C}_1(G', \mathbb{C})$ be such that $\overline{c'(e)}=\omega(\gamma_e)$ for all edges $e\in E(G')$. 
First note that the inclusion $\Sigma_g \subset S^3\setminus (G\cup G')$ is a retract by deformation and that the induced isomorphism $\mathrm{H}^1(\Sigma_g; \mathbb{C}/\mathbb{Z})\cong \mathrm{H}^1(S^3 \setminus (G\cup G') ; \mathbb{C}/\mathbb{Z})$ sends $\omega$ to the shadow $\omega_{\overline{c+c'}}$ of Definition \ref{def_shadow_coloring}. 
By Lemma \ref{lemma_transparent},  the CGP invariant satisfies 
$$\left<(G,c+N c_0)\cup (G', c')\right> = q^{ N \omega(\widehat{c_0})} \left<(G,c)\cup (G',c')\right>.$$
Since this is true for every $c'$ and since, by Lemma \ref{lemma_span_space}, the vectors $[G',c']$ span $\mathbb{V}'(\Sigma_g, \omega)$, then  $(G,c+Nc_0)-e^{4i\pi N \omega(\widehat{c_0})}(G,c)$ belongs to the right kernel of the Hopf pairing and $[G,c+Nc_0]=q^{ N \omega(\widehat{c_0})}[G,c]$.
 \end{proof}
 
Set $$ \col(G) := \{ c : E(G)\to \mathbb{C}, \mbox{ such that } \overline{c(e)}=\omega(\gamma_e) \mbox{ and } 0\leq  \Re( c(e)) < N \mbox{ for all }e\in E(G) \}$$
and 
$$ \mathcal{B}:= \{ [G,c], c\in \col(G) \} \subset \mathbb{V}(\Sigma_g, \omega).$$

While gluing the handlebody $H_g$ with itself along the boundary using the identity map $\id : \Sigma_g \to \Sigma_g$ we obtain the manifold $X_g:=H_g\cup_{\id } H_g \cong (\mathbb{S}^1\times \mathbb{S}^2)^{\# g}$ made of $g$ connected sums of $\mathbb{S}^1\times \mathbb{S}^2$. 
Let $\iota_1, \iota_2 : H_g \hookrightarrow X_g$ the two inclusion maps.
The class $\omega$ induces a class $\overline{\omega}\in \mathrm{H}^1( X_g; \mathbb{C}/\mathbb{Z})$ given by the class of $\omega_{H_g}$ in both copies of $H_g$ inside $X_g$. Given two ribbon graphs $\Gamma, \Gamma' \in \mathscr{S}^{\mathcal{C}}(H_g , \omega_{H_g})$, we get a ribbon graph $\iota_1(\Gamma)\cup \iota_2(\Gamma')$ in $X_g$ compatible with $\overline{\omega}$ so we can consider the   CGP invariant $\left< X_g, \iota_1(\Gamma)\cup \iota_2(\Gamma'), \overline{\omega} \right>\in \mathbb{C}$ defined in  \cite{CGPInvariants}. The pairing 
$$ \left< \cdot, \cdot \right>_{\omega} : \mathscr{S}^{\mathcal{C}}(H_g) \otimes \mathscr{S}^{\mathcal{C}}(H_g) \to \mathbb{C}, \quad \left< \Gamma, \Gamma'\right>_{\omega}:= \left< X_g, \iota_1(\Gamma)\cup \iota_2(\Gamma'), \overline{\omega} \right>, $$
has the same right kernel than the Hopf pairing and 
induces a non-degenerate pairing $ \left< \cdot, \cdot \right>_{\omega}: \mathbb{V}(\Sigma_g, \omega)^{\otimes 2} \to \mathbb{C}$, named the \textit{invariant pairing} which is indeed invariant in the sense that $\left< r_{\omega}(z)\cdot x, y\right> = \left< x, r_{\omega}(z)\cdot y\right>$. The following lemma is the only place where we actually need the whole TQFTs properties of $\mathbb{V}(\Sigma, \omega)$ proved in \cite{BCGPTQFT,  DeRenzi_NSETQFT, DeRenziGeerPatureau_TQFT_QG}.

\begin{lemma}\label{lemma_InvPairing} For $c,c' \in \col(G)$, then  $\left< [G,c], [G, c']\right>=0$ if $c\neq c'$ and $\left< [G,c], [G, c]\right>\neq 0$. \end{lemma}

\begin{proof}
If $c(e)\neq c'(e)$, then one has an open ball $\mathbb{B}\subset X_g$ transversed to the graph $\iota_1(\Gamma)\cup \iota_2(\Gamma')$ such that $\partial \mathbb{B}$ intersects $\iota_1(\Gamma)\cup \iota_2(\Gamma')$ along two points: one colored by $V_{c(e)}$ and the other colored by $V_{c'(e)}$. Since $c(e)\neq c'(e)$ have real part in $[0, N)$ they are not congruent modulo $\mathbb{Z}$ so one has $\Hom(V_{c(e)}, V_{c'(e)})=0$.  
It follows from the TQFTs properties of the CGP invariant $\left<M,\Gamma, \omega\right>$ that whenever we have an embedded sphere $S^2\subset M$ which intersects $\Gamma$ along two points colored by $V,W$ with $\Hom_{\mathcal{C}}(V,W)=0$ then $\left<M,\Gamma, \omega\right>=0$ so this implies that $\left< [G,c], [G, c']\right>=0$. 
\par When $c'=c$, we need again to use the TQFTs property of the CGP invariant to prove that $\left< [G,c], [G, c]\right>$ can be written as a product
of  non-vanishing $3j$-symbols with the invariant of the sphere $S^3$ (see the proof of \cite[Proposition $6.7$]{BCGPTQFT} for details). This implies that  $\left< [G,c], [G, c]\right>\neq 0$ and concludes the proof.
\end{proof}

An analogue of the following was proved in \cite[Proposition $6.7$]{BCGPTQFT} and \cite{BCGP_Bases} in the case of even roots.

\begin{theorem}\label{theorem_basis} $\mathcal{B}$ is a basis of $\mathbb{V}(\Sigma_g, \omega)$.
\end{theorem}

In particular $\dim (\mathbb{V}(\Sigma_g, \omega))= N^{|E(G)|} = N^{3g-3} = PI-deg(\mathcal{S}_A(\Sigma_g))$. 

\begin{proof}
Let us first prove that $\mathcal{B}$ spans $\mathbb{V}(\Sigma_g, \omega)$. By Lemma \ref{lemma_span_space}, $\mathbb{V}(\Sigma_g, \omega)$ is spanned by the vectors $[G,c]$ with $c$ $\omega$-compatible. For such a coloring $c$, there exists $c_0 \in \mathrm{C}_1(G; \mathbb{Z})$ and $c_1\in \col(G)$ such that $c=c_1+Nc_0$. By Lemma \ref{lemma_modulo}, we have $[G,c]= q^{ N \omega(c_0)}[G, c_1]$ so $\mathcal{B}$ spans $\mathbb{V}(\Sigma_g, \omega)$.
\vspace{2mm}
\par Let us prove that $\mathcal{B}$ is free. First, by Lemma \ref{lemma_InvPairing}, the vectors $[G,c]$ are non-null. 
Consider the commutative subalgebra $T\subset \mathcal{S}_A(\Sigma_g)$ generated by the classes  $[\gamma_e]$ of the curves in the pants decomposition. By Lemma \ref{lemma_S}, writing $\lambda_e:=  q^{ \omega([\gamma_e])} + q^{-\omega([\gamma_e])} $, one has 
$$r_{\omega}([\gamma_e])\cdot [G,c]= \lambda_{c(e)} [G,c], \quad \mbox{ for all }e\in {E}(G).$$
Note that $\lambda_{c(e)}= \lambda_{c'(e)}$ implies $c(e)=c'(e)$ for every $c,c'\in \col(G)$, so each axis $[G,c]\mathbb{C}$ is an intersection of eigenspaces of the operators in $r_{\omega}(T)$. This implies that $\mathcal{B}$ is free and concludes the proof.
\end{proof}

\begin{proposition}\label{prop_shadow}
The representation $r_{\omega} : \mathcal{S}_A(\Sigma_g) \to \End(\mathbb{V}(\Sigma_g, \omega))$ is central and its classical shadow is the class of the diagonal representation
$$\rho_{\omega} : \gamma \mapsto  \varepsilon_{\gamma}\begin{pmatrix} q^{N  \omega([\gamma])} & 0 \\ 0 & q^{-N  \omega([\gamma])} \end{pmatrix}, $$
where $\varepsilon_{\gamma}=+1$ if $A^N=1$ and $\varepsilon_{\gamma}=(-1)^{w_S([\gamma])}$ if $A^N=-1$.
\end{proposition}

\begin{remark} Here we implicitly used the same spin structure $S$ in the  identifications of Theorems \ref{theorem_skein+1} and \ref{theorem_skeinQG}.
\end{remark}

\begin{proof}
We need to prove that for all $\gamma$, then $T_N( r_{\omega}([\gamma]))=  \varepsilon_{\gamma}(q^{N  \omega([\gamma])} + q^{-N  \omega([\gamma])})\id$. Since the skein algebra at $A=\pm 1$ is generated by the classes of non-separating curves, we can suppose that $\gamma$ is non-separating.
When $e\in E(G)$, then by Lemma \ref{lemma_S} one has 
$$ r_{\omega}([\gamma_e]) \cdot [G,c]= \varepsilon_{\gamma}(q^{ c(e)} + q^{- c(e)} ) [G,c].$$
 Since $T_N(x+x^{-1})= x^N +x^{-N}$, one finds that $T_N(r_{\omega}([\gamma_e])) \cdot [G,c] = \varepsilon_{\gamma}(q^{N \omega([\gamma_e])} + q^{-N \omega([\gamma_e])}) [G,c]$ as expected. Note that the basis $\mathcal{B}$ depends on the choice of a pants decomposition of $\Sigma$ and we have proved that for any pants decomposition $\mathcal{P}=\{\gamma_e\}_e$ such that all $\gamma_e$ are non-separating, all curves $\gamma_e$ satisfy $T_N(r_{\omega}([\gamma_e])) \cdot [G,c] = \varepsilon_{\gamma}(q^{N \omega([\gamma_e])} + q^{-N \omega([\gamma_e])}) [G,c]$.  We conclude using the fact that every non-separating curve $\gamma\subset \Sigma$ belongs to such a pants decomposition of $\Sigma$. 
\end{proof}

\subsection{Irreducibility}

The following classical lemma is the key to prove the irreducibility of $r_{\omega}$.

  \begin{lemma}\label{lemma_irrep}
  Let $\mathcal{A}$ be a $\mathbb{C}$-algebra and $r: \mathcal{A} \to \End(V)$ a finite dimensional representation such that $V$ is equipped with an invariant form $\left<\cdot, \cdot\right>: V \otimes V \to \mathbb{C}$. Let $T\subset \mathcal{A}$ a commutative subalgebra and consider the weight decomposition
  $$ V= \oplus_{\chi: T\to \mathbb{C}} V^{(\chi)}, \quad V^{(\chi)}:=\{v\in V| r(t)v=\chi(t)v, \forall t\in T \}.$$
  Suppose that: 
  \begin{enumerate}
  \item the spaces $V^{(\chi)}$ have dimension  $1$, say $V^{(\chi)}=\mathbb{C} v_{\chi}$; 
  \item for each $\chi, \chi'$, there exists $a \in \mathcal{A}$ such that $\left< v_{\chi'}, r(a)\cdot v_{\chi}\right>\neq 0.$
\end{enumerate}
Then $r$ is irreducible.
\end{lemma}

\begin{proof}
Let $0\neq W\subset V$ be an $\mathcal{A}$-invariant subspace and let $\pi_W : V\to W$ be the orthogonal projector with image $W$. Since $W$ is $\mathcal{A}$-invariant, then $\pi_W$ commutes with $r(\mathcal{A})$. In particular it commutes with $r(T)$ so for each character $\chi$, either $V^{(\chi)}\subset W$ or $V^{(\chi)}\cap W =\{0\}$. Since $W\neq 0$, there exists $\chi_0$ such that $v_{\chi_0} \in W$ so the cyclic space $r(\mathcal{A})\cdot v_{\chi_0}$ is included in $W$. For every $\chi$, one has $\left<v_{\chi}, r(\mathcal{A})v_{\chi_0}\right>\neq 0$ so $V^{(\chi)}\subset W$ therefore $W=V$ so $V$ has no non-zero proper submodule.
\end{proof}

We want to apply Lemma \ref{lemma_irrep} to the representation $r_{\omega}$ of  $\mathcal{S}_A(\Sigma_g)$ with torus $T= \{ [\gamma_e], e \in E(G)\}$. In this case, the decomposition writes 
$$ \mathbb{V}(\Sigma_g, \omega) = \oplus_{c \in \col(G)} \mathbb{C} [G,c] $$
where $\mathbb{C} [G,c]$ is the one-dimensional weight space corresponding to the character $\chi_c$ sending $[\gamma_e]$ to $(q^{ \omega(c(e))} + q^{- \omega(c(e))})$. So Lemma \ref{lemma_irrep} shows that $r_{\omega}$ is irreducible if and only if  for each $c, c'\in \col(G)$ there exists an element $\beta \in \mathcal{S}_A(\Sigma_g)$ such that $\left< [G,c'], r_{\omega}(\beta)\cdot [G,c]\right>\neq 0$. 
\par  
Let $\beta_1, \ldots, \beta_g \subset H_g$ be the framed curves drawn in Figure \ref{fig_graphs}. Each $\beta_i$ can be represented uniquely by a cycle $c_{\beta_i} \in \mathrm{Z}_1(G; \mathbb{Z})$ such that $c_{\beta_i}(e)\in \{-1, 0, +1\}$. Let $cl(\beta_i)$ be the set of cycles $\eta: E(G) \to \{-1, 0, +1\}$ such that $\eta(e)=0$ if and only if $c_{\beta_i}(e)=0$.


\begin{notations}\label{notations_Y}
 Let $Y(\omega)\subset \mathbb{C}$ be the finite set of $6j$-symbols of the form $$6S(\eta_1 c(e_a)  +\varepsilon_1 , \eta_2 c(e_b)  +\varepsilon_2 , \eta_3 c(e_c)  +\varepsilon_3; \eta_4, \eta_5)$$ where  $(e_a, e_b, e_c)$ are three adjacent edges of $G$ oriented towards a common vertex, $c\in \col(G)$,  $\eta_u\in \{-1, +1\}$ and $\varepsilon_v \in \{-1, 0, +1\}$ are such that $\varepsilon_1+\varepsilon_2+\varepsilon_3$ is even.
 \end{notations}

\begin{lemma}\label{lemma_step1}
One has 
$$ r_{\omega}([\beta_i]) \cdot [G,c] = \sum_{\eta \in cl(\beta_i)} x_{\eta} [G, c+\eta] $$
where $x_{\eta}$ is a product of elements of $Y(\omega)$ and of non-vanishing quantum binomials of the form $\qbinom{a+N-1}{a}^{-1}$ with $a\in \mathbb{C}\setminus \frac{1}{2k}\mathbb{Z}$.
\end{lemma}

\begin{proof}
The proof is a simple computation using the following two skein relations: 
$$\includegraphics[width=12cm]{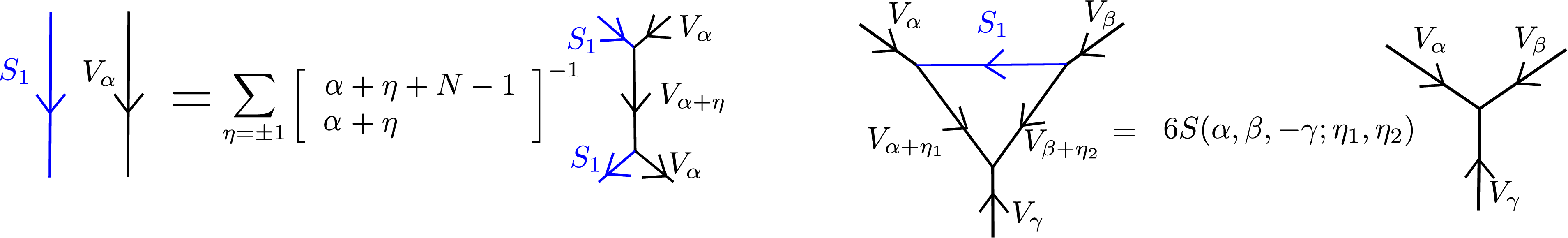}.$$
For instance, for $\beta_1$, writing $x_{\eta_1, \eta_2}:= \qbinom{c_1+\eta_1 +N-1}{c_1 +\eta_1}^{-1} \qbinom{c_2+\eta_2 +N-1}{c_2 +\eta_2}^{-1}$, 
one finds (here we put $c_4:=c_3$ if $g=2$):
\begin{multline*}
 r_{\omega}(\beta_1)\cdot [G,c]= \adjustbox{valign=c}{\includegraphics[width=3cm]{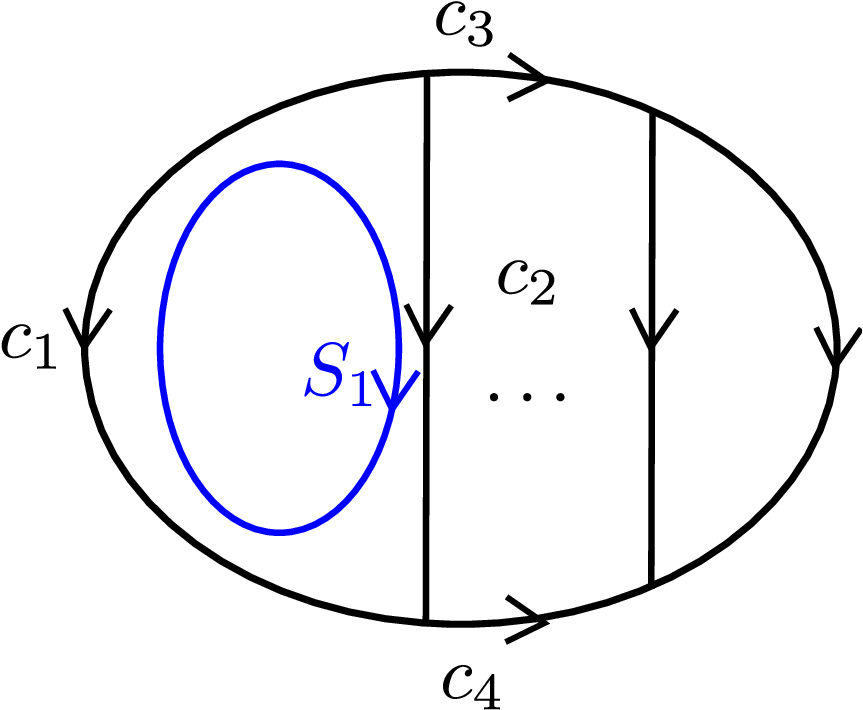}} =\sum_{\eta_1, \eta_2= \pm 1} x_{\eta_1, \eta_2} \adjustbox{valign=c}{\includegraphics[width=3.7cm]{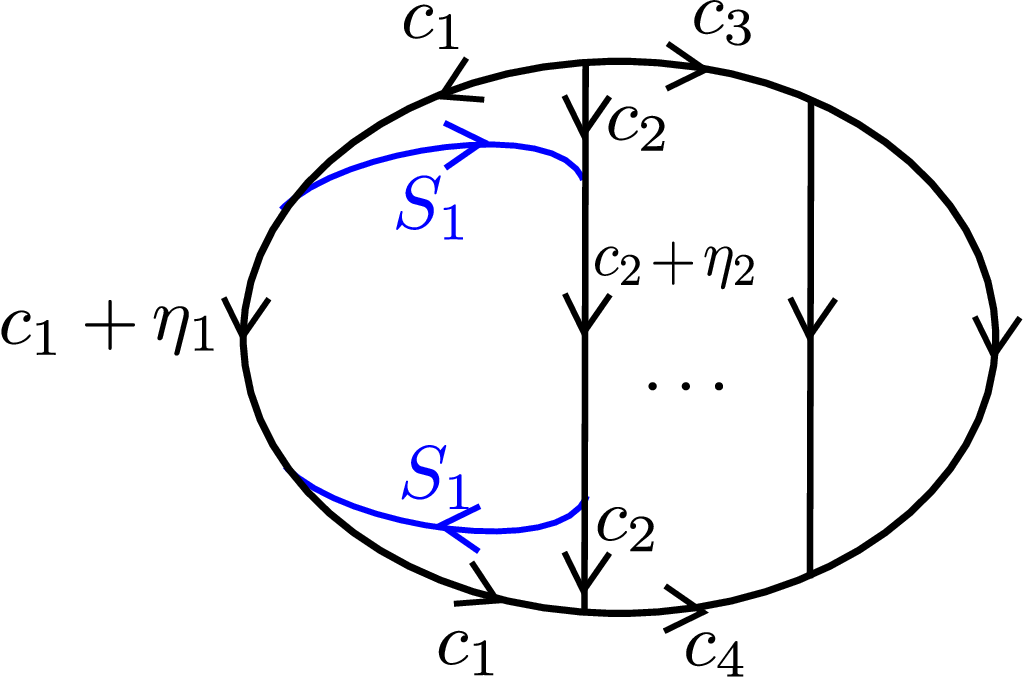}} \\
 = \sum_{\eta_1, \eta_2=\pm 1}x_{\eta_1, \eta_2} 6S(-c_2-\eta_2, -{c_1}-\eta_1 , c_3 ; \eta_2, \eta_1) 6S(c_1 +\eta_1 , c_2 +\eta_2 , c_4 ; -\eta_1, -\eta_2) \adjustbox{valign=c}{\includegraphics[width=3cm]{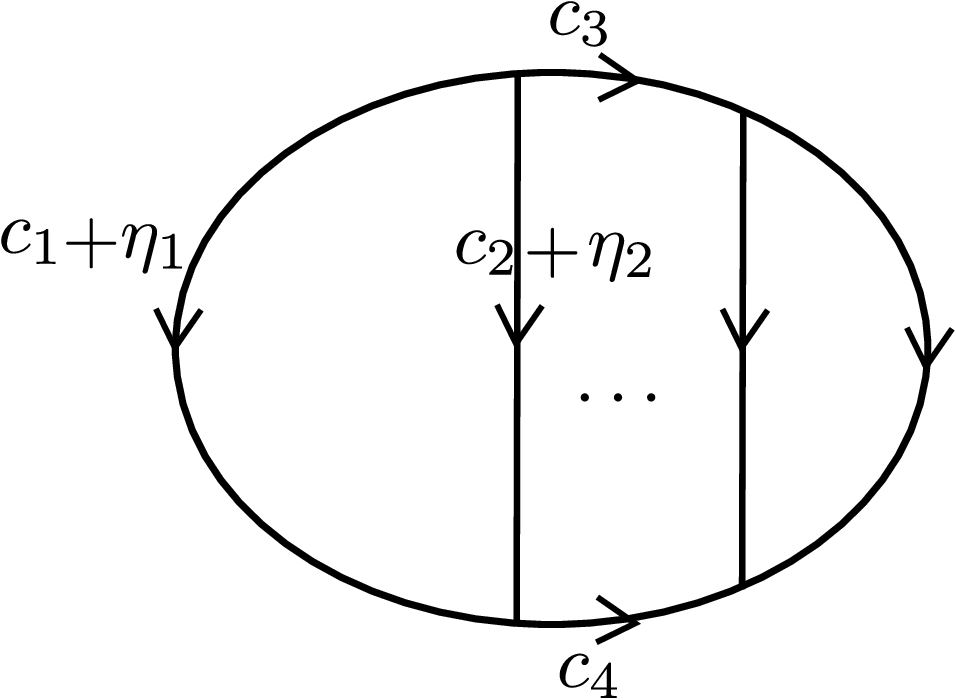}}.
\end{multline*}
The cases of $\beta_i$ for $i>1$ are similar and left to the reader.
\end{proof}

\begin{lemma}\label{lemma_step2}
If $\omega\in \mathcal{U}$ is such that $0 \notin Y(\omega)$, then $r_{\omega}$ is irreducible.
\end{lemma}

\begin{proof}
Let $\mathscr{A}:= r_{\omega}(\mathcal{S}_A(\Sigma_g))\subset \End(\mathbb{V}(\Sigma_g, \omega))$. 
Note that $\mathbb{C}v_c$ is an intersection of eigenspaces of the operators $r_{\omega}(\gamma_e)$. If $\pi \in \End(\mathbb{V}(\Sigma_g, \omega))$ is a projector which commutes with $\mathscr{A}$, it must preserve these eigenspaces so $\pi(v_c) \subset \mathbb{C}v_c$. This means that if $W \subset \mathbb{V}(\Sigma_g, \omega)$ is preserved by $\mathscr{A}$ then there exists $J \subset \col(G)$ such that $W=\Span ( [G,c], c\in J)$. In particular, for $c,c' \in \col(G)$, either 
$\mathscr{A}v_c = \mathscr{A}v_{c'}$, in which case we say that $c,c'$ are \textit{related}, or $\mathscr{A}v_c \cap \mathscr{A}v_{c'}=0$. Since the pairing $\left<\cdot, \cdot \right>$ is $\mathscr{A}$ invariant, then $c,c'$ are related if and only if 
 $\left< [G,c], r_{\omega} (x)  [G,c']\right> \neq 0$ for some $x\in\mathcal{S}_A(\Sigma_g)$. Being related is clearly an equivalence relation.
By Lemma \ref{lemma_irrep}, to prove that $r_{\omega}$ is irreducible,  it suffices to prove that every $c,c' \in \col(G)$ are related. Lemma \ref{lemma_step1} shows that, when $0 \notin Y(\omega)$,  for every $1\leq i \leq g$ and for every $\eta\in cl(\beta_i)$ then $c$ and $c+\eta$ are related. $\col(G)$ has a natural structure of torsor over the abelian group $\mathrm{C}_1(G; \mathbb{Z}/N\mathbb{Z})\cong (\mathbb{Z}/N\mathbb{Z})^{E(G)}$ where the free transitive action is given by sending $c_0\in \mathrm{C}_1(G; \mathbb{Z}/N\mathbb{Z})$ and $c_1\in \col(G)$ to the unique element of $\col(G)$ which induces the class of $c_0+c_1$ in $\mathrm{C}_1(G; \mathbb{C}/N\mathbb{Z})$. The set $\{ \eta; \eta \in k(\beta_i), 1\leq i \leq g\}\subset \col(G)$ generates the group $\mathrm{C}_1(G; \mathbb{Z}/N\mathbb{Z})$ so  every $c,c' \in \col(G)$ are related and $r_{\omega}$ is irreducible.
\end{proof}

\begin{lemma}\label{lemma_step3} There exists $\omega\in \mathcal{U}$ such that $0 \notin Y(\omega)$.
\end{lemma}

\begin{proof}
Let $Z(\omega) \in \mathbb{C}$ be the product of every $6j$-symbols in $Y(\omega)$. By Proposition \ref{prop_6j}, there exists $R \in \mathbb{Q}(A)(X_1, \ldots, X_{3g-3})$ such that $R\neq 0$ and
$$Z(\omega) = R(q^{\frac{\omega(\gamma(e_1))}{2}}, \ldots, q^{\frac{\omega(\gamma(e_{3g-3}))}{2}}).$$
Let  $\varphi : \mathcal{U} \to (\mathbb{C}^*)^{3g-3}$ be the analytic map sending $\omega$ to $(q^{\frac{\omega(\gamma(e_1))}{2}}, \ldots, q^{\frac{\omega(\gamma(e_{3g-3}))}{2}})$;  so $Z$ decomposes as:
$$ Z: \mathcal{U} \xrightarrow{ \varphi} (\mathbb{C}^*)^{3g-3} \xrightarrow{R} \mathbb{C}.$$ 
Since the image $\varphi(\mathcal{U})$ is a non-empty open (analytic) subset, if $Z(\mathcal{U})=\{0\}$, then we would have $R=0$ which would contradict Proposition \ref{prop_6j}. Therefore there exists $\omega \in \mathcal{U}$ such that $Z(\omega)\neq 0$, i.e. such that $0 \notin Y(\omega)$.
\end{proof}

\begin{proof}[Proof of Theorems \ref{main_theorem} and \ref{theorem2}]
By Lemma \ref{lemma_step3}, there exists $\omega \in \mathcal{U}$ such that $0\notin Y(\omega)$ so Lemma \ref{lemma_step2} implies that $r_{\omega}$ is irreducible. By Proposition \ref{prop_shadow}, $r_{\omega}$ is a central representation with classical shadow  $[\rho_{\omega}]\in \mathcal{X}^{(1)}$ and its dimension is equal to the PI-degree of $\mathcal{S}_A(\Sigma_g)$ (i.e. $N^{3g-3}$) by Theorem \ref{theorem_basis}. So Corollary \ref{coro_roro} implies that $[\rho_{\omega}]$ belongs to the Azumaya locus of $\mathcal{S}_A(\Sigma_g)$ and Corollary \ref{coro_X1} implies that the whole strata $\mathcal{X}^{(1)}$ is included in the Azumaya locus.

\end{proof}

\section{Open surfaces}\label{sec_opensurfaces}

We now extend the previous results to the open surfaces $\Sigma_{g,n}$ with $n\geq 1$. Since the skein algebras are commutative (and thus Azumaya) when $(g,n)$ is $(0,0)$, $(0,1)$, $(0,2)$ or $(0,3)$, we now assume that we are not in one of these cases.

\subsection{Center and sliced character varieties}

Let $\gamma_1, \ldots, \gamma_n$ be the peripheral curves encircling once the boundary components of $\Sigma_{g,n}$. 
\begin{theorem}\label{theorem_FKL_open}(Frohman-Kania Bartoszynska, L\^e \cite{FrohmanKaniaLe_UnicityRep, FrohmanKaniaLe_DimSkein}) The center $Z$ of $\mathcal{S}_A(\Sigma_{g,n})$ is the algebra $Ch_A(\mathcal{S}_{+1}(\Sigma_{g,n}))[\gamma_1, \ldots, \gamma_n]$ with relation $T_N(\gamma_i)= Ch_A(\gamma_i)$. Moreover $\mathcal{S}_A(\Sigma_{g,n})$ has PI-degree $N^{3g-3+n}$.
\end{theorem}

Therefore the variety $\Specm(Z)$ is isomorphic to the variety 
$$ \widehat{\mathcal{X}}_{\SL_2}(\Sigma_{g,n}) = \{ ([\rho], z_1, \ldots , z_n)\mid [\rho]\in \mathcal{X}_{\SL_2}(\Sigma_{g,n}), z_i \in \mathbb{C} \mbox{ s.t. } T_N(z_i)=-\tr(\rho(\gamma_i)) \}$$
and the Chebyshev-Frobenius morphism $ 
Ch_A : \mathcal{S}_{\varepsilon}(\Sigma_{g,n}) \hookrightarrow Z$ 
defines a branched covering $\pi : \widehat{\mathcal{X}}_{\SL_2}(\Sigma_{g,n}) \to \mathcal{X}_{\SL_2}(\Sigma_{g,n})$ sending $ ([\rho], z_1, \ldots , z_n)$ to $[\rho]$. 
When $\varepsilon=+1$,  we have chosen the Spin structure $S$ such that its quadratic form $w_S$ satisfies $w_S([\gamma_i])=0$. 

\begin{definition}(Fully Azumaya locus) The \textit{fully Azumaya locus} is the subset 
$$ \mathcal{FAL}:= \{ ([\rho] \in \mathcal{X}_{\SL_2}(\Sigma_{g,n}) \mbox{ such that } \pi^{-1}([\rho]) \subset \mathcal{AL} \}$$
i.e. is the set of classes of representations $[\rho]$ such that $([\rho], z_1, \ldots, z_n)$ is in the Azumaya locus of $\mathcal{S}_A(\Sigma_{g,n})$ for every compatible $(z_1, \ldots, z_n)$. 
\end{definition}

Extend the Poisson structure of  $\mathcal{X}_{\SL_2}(\Sigma_{g,n})$ to $\widehat{\mathcal{X}}_{\SL_2}(\Sigma_{g,n}) $ by setting that the elements $\gamma_i$ are Casimir (i.e. that $\{\gamma_i, f\} = 0$ for all $f$). Recall from Section \ref{sec_PO} that we have a Poisson order  $(\mathcal{S}_A(\Sigma_{g,n}), \widehat{\mathcal{X}}_{\SL_2}(\Sigma_{g,n}), Ch_A, D)$ where $D$ is defined as in the closed case with the additional condition $D_{\gamma_i}=0$. A classification of the symplectic leaves is not known in general (see however \cite[Proposition $2$]{FockRosly} and  \cite[Theorem $9.1$]{GHJW_ModSpacesParBd} for related results). 
Consider the morphisms 
$${\nu}: {\mathcal{X}}_{\SL_2}(\Sigma_{g,n}) \to \mathbb{C}^n, \quad {\nu}([\rho]) := (\tr(\rho(\gamma_1)),\ldots, \tr(\rho(\gamma_n)))$$ and 
$$\widehat{\nu}: \widehat{\mathcal{X}}_{\SL_2}(\Sigma_{g,n}) \to \mathbb{C}^n, \quad \widehat{\nu}([\rho], z_1,\ldots, z_n) := (z_1,\ldots, z_n).$$
For $c\in \mathbb{C}^n$ and $z \in \mathbb{C}^n$ such that $T_N(z_i)=c_i$, we set $\mathcal{X}_{\SL_2}(\Sigma_{g,n}, c):= \nu^{-1}(c)$ and $\widehat{\mathcal{X}}_{\SL_2}(\Sigma_{g,n}, z):= \widehat{\nu}^{-1}(z)$. 
$\mathcal{X}_{\SL_2}(\Sigma_{g,n}, c)$ is called the \textit{relative} or \textit{sliced} character variety. Since the $\gamma_i$ are Casimir elements, then the symplectic leaves are contained in the lifts of the sliced character varieties.
 It is proved independently in \cite{Whang_RelCharVar, FKL_GeometricSkein} that $\mathcal{X}_{\SL_2}(\Sigma_{g,n}, c)$ is irreducible, in \cite{Whang_RelCharVar} that it is normal and in \cite[Proposition $2$]{FockRosly} and  \cite[Theorem $9.1$]{GHJW_ModSpacesParBd} that the smooth loci of sliced character varieties are symplectic. However none of these loci is dense in ${\mathcal{X}}_{\SL_2}(\Sigma_{g,n}) $ so we cannot conclude easily like in the case of closed surfaces. The trick is to consider the \textit{sliced skein algebras}: 
 $$ \mathcal{S}_A(\Sigma_{g,n}, z):= \quotient{\mathcal{S}_A(\Sigma_{g,n})}{\left( \gamma_i - z_i , i=1, \ldots, n \right)}.$$
 Let $c=(c_1, \ldots, c_n)$ and $z=(z_1, \ldots, z_n)$ with $c_i=-T_N(z_i)$. 
 It is proved in \cite{FKL_GeometricSkein} that $ \mathcal{S}_A(\Sigma_{g,n}, z)$ is an almost Azumaya algebra with the same PI-degree as  $\mathcal{S}_A(\Sigma_{g,n})$ (i.e. $N^{3g-3+n}$), therefore the inclusion $\mathcal{X}_{\SL_2}(\Sigma_{g,n}, c) \subset \mathcal{X}_{\SL_2}(\Sigma_{g,n})$ sends the Azumaya locus of $\mathcal{S}_A(\Sigma_{g,n}, z)$ into the Azumaya locus of $\mathcal{S}_A(\Sigma_{g,n})$.
\par 
 The Chebyshev morphism induces an isomorphism  $Ch_A: \mathbb{C}[\mathcal{X}_{\SL_2}(\Sigma_{g,n}, c)] \cong \mathcal{Z}\left(\mathcal{S}_A(\Sigma_{g,n}, z)\right)$ 
  and we obtain a Poisson order $( \mathcal{S}_A(\Sigma_{g,n}, z),  \mathcal{X}_{\SL_2}(\Sigma_{g,n}, z), Ch_A, D)$. Since the smooth locus of $\mathcal{X}_{\SL_2}(\Sigma_{g,n},z)$ is symplectic,   we deduce the
 
 \begin{theorem}(Frohman-Kania Bartoszynska-L\^e \cite[Theorem $11.1$]{FKL_GeometricSkein}) \label{theorem_FKL_open}
 If $[\rho]\in \mathcal{X}_{\SL_2}(\Sigma_{g,n}, c)$ belongs to the  smooth loci of $\mathcal{X}_{\SL_2}(\Sigma_{g,n}, c)$, then $[\rho]$ is in the fully Azumaya locus of $\mathcal{S}_A(\Sigma_{g,n})$.
 \end{theorem}

An explicit description of these smooth loci is not known in general but a lot can be deduced: 

\begin{corollary}\label{coro_FKL_open}
\begin{enumerate}
\item (\cite[Proposition $4.3$]{FKL_GeometricSkein}) If $\rho: \pi_1(\Sigma_{0,n}) \to \SL_2$ is irreducible and $c_i=\tr(\rho(\gamma_i)) \neq \pm 2$ for all $1\leq i \leq n$, then $[\rho]\in \mathcal{FAL}(\mathcal{S}_A(\Sigma_{0,n}))$.
\item (\cite[Proposition $4.4$]{FKL_GeometricSkein}) Let $\rho: \pi_1(\Sigma_{g,n}) \to \SL_2$ such that: $(1)$  $c_i=\tr(\rho(\gamma_i)) \neq \pm 2$ for all $1\leq i \leq n$ and $(2)$
 if $c_i=t_i +t_i ^{-1}$ then $t_1t_2\ldots t_n \neq 1$. Then $[\rho] \in \mathcal{FAL}(\mathcal{S}_A(\Sigma_{g,n}))$.
\end{enumerate}
\end{corollary}

Let us first make some simple remarks. First Lemma \ref{lemma_X2} has an obvious analogue for open surfaces.
The group $\mathrm{H}^1(\Sigma_{g,n}; \mathbb{Z}/2\mathbb{Z})$ acts by automorphism on $\mathcal{S}_A(\Sigma_{g,n})$ by the formula $\chi \cdot [\gamma]:= (-1)^{\chi(\gamma)}[\gamma]$. This action induces an action on the center $Z$ which defines an action on $\widehat{\mathcal{X}}_{\SL_2}(\Sigma_{g,n}) \cong \Specm(Z)$ described by
$$ \chi \cdot ([\rho], z_1, \ldots, z_n) := ([\chi\cdot \rho], (-1)^{\chi(\gamma_1)}z_1, \ldots, (-1)^{\chi(\gamma_n)}z_n), \mbox{ where } \chi\cdot \rho (\alpha):= (-1)^{\chi(\alpha)} \rho(\alpha).$$

\begin{lemma}\label{lemma_X2_open}
 The Azumaya locus is preserved by this $\mathrm{H}^1(\Sigma_{g,n}; \mathbb{Z}/2\mathbb{Z})$  action.
 \end{lemma}

\begin{proof}
This follows from the fact that the action of $\chi \in \mathrm{H}^1(\Sigma_{g,n}; \mathbb{Z}/2\mathbb{Z})$ on $\mathcal{S}_A(\Sigma_{g,n})$ induces an isomorphism $\mathcal{S}_A(\Sigma_{g,n})_x \cong \mathcal{S}_A(\Sigma_{g,n})_{\chi \cdot x}$.
\end{proof}

\begin{lemma}\label{lemma_z2} Suppose that $(g,n)\neq (1,1)$. If $x=([\rho], z_1, \ldots, z_n)$ is such that $z_i=-(q+q^{-1})$ for some $i$, then $x$ does not belong to the Azumaya locus of $\mathcal{S}_A(\Sigma_{g,n})$. In particular, writing $c_i:=\tr(\rho(\gamma_i))=-T_N(z_i)$, if there exists $i$ such that $c_i=2$ then $[\rho]\notin \mathcal{FAL}$. So $n\geq 2$ and $c_i=-2$ for some $i$ also implies that $[\rho]\notin \mathcal{FAL}$.
\end{lemma}

\begin{proof}
The condition $(g,n)\neq (1,1)$ ensures that $PI-deg(\mathcal{S}_A(\Sigma_{g,n-1}))<PI-deg(\mathcal{S}_A(\Sigma_{g,n}))$. The inclusion $\Sigma_{g,n} \hookrightarrow \Sigma_{g,n-1}$ ($\Sigma_{g,n}$ is obtained from $\Sigma_{g,n-1}$ by removing an open disc)  induces a surjective morphism $p: \mathcal{S}_A(\Sigma_{g,n}) \to \mathcal{S}_A(\Sigma_{g,n-1})$ and an embedding $i:\widehat{\mathcal{X}}_{\SL_2}(\Sigma_{g,n-1}) \to \widehat{\mathcal{X}}_{\SL_2}(\Sigma_{g,n})$. If $x=([\rho], z_1, \ldots, z_n)$ is such that $z_1=-(q+q^{-1})$ ($z_1$ corresponds to the boundary of the disc we removed) then $x$ is in the image of $i$ so $x=i(x_0)$. Let $r_0: \mathcal{S}_A(\Sigma_{g,n-1}) \to \End(V)$ be an irreducible representation with classical shadow $x_0$. Then $\dim(V) \leq PI-deg(\mathcal{S}_A(\Sigma_{g,n-1}))$, so  $r:= r_0 \circ p : \mathcal{S}_A(\Sigma_{g,n}) \to \End(V)$ is a non-trivial representation of dimension $\dim(V)< PI-deg(\mathcal{S}_A(\Sigma_{g,n}))$ thus its classical shadow $x$ does not belong to the Azumaya locus. So $[\rho]\notin \mathcal{FAL}$ whenever $c_i=2$ for some $i$. If $n\geq 2$, we can always find $\chi\in \mathrm{H}^1(\Sigma_{g,n}; \mathbb{Z}/2\mathbb{Z})$ such that $\chi(\gamma_i)\equiv 1 \pmod{2}$. So Lemma \ref{lemma_X2_open} implies that $[\rho]\notin \mathcal{FAL}$ whenever $c_i=-2$ for some $i$ in this case.
\end{proof}

We will be interested in the loci of diagonal representations.

\begin{lemma}\label{lemma_symp_open} Let $\mathcal{X}^{(1)}_c \subset \mathcal{X}_{\SL_2}(\Sigma_{g,n}, c)$ be the locus of classes of diagonal non-central representations. Then $\mathcal{X}^{(1)}_c$ is a connected smooth irreducible symplectic subvariety. In particular, either $\mathcal{X}^{(1)}_c$ is included in the Azumaya locus or it does not intersect it. 
\end{lemma}

\begin{proof}
We proceed in the same manner than in the proof of Lemma \ref{lemma_X1}. Let $\mathcal{X}_{\mathbb{C}^*}(\Sigma_{g,n})=\mathrm{H}^1(\Sigma_{g,n}; \mathbb{C}^*) = \Specm( \mathbb{C}[\mathrm{H}_1(\Sigma_{g,n}; \mathbb{Z})] )$ equipped with the Poisson structure defined by $\{ X_{[\alpha]}, X_{[\beta]}\}=([\alpha], [\beta]) X_{[\alpha+\beta]}$. Since the kernel of the intersection form $(\cdot, \cdot)$ is generated by the classes $[\gamma_i]$, the symplectic leaves of  $\mathcal{X}_{\mathbb{C}^*}(\Sigma_{g,n})$ are the fibers of the map $f: \mathcal{X}_{\mathbb{C}^*}(\Sigma_{g,n}) \to (\mathbb{C}^*)^n$ sending $\chi$ to $(\chi([\gamma_1]), \ldots, \chi([\gamma_n]))$. Denote by $\mathbb{T}_{\mathbf{z}}:= f^{-1}(z_1, \ldots, z_n)$ such a fiber: it is a symplectic torus of dimension $2g$. 
Define $\Phi: \mathcal{X}_{\mathbb{C}^*}(\Sigma_{g,n}) \to \mathcal{X}^{(1)}\cup \mathcal{X}^{(2)}$ by $\Phi(\chi):=[\varphi \circ \chi]$. Like in the proof of Lemma \ref{lemma_X1}, using Goldman's formula (which also holds for open surfaces as proved in \cite{Lawton_PoissonGeomSL3, KojuTriangularCharVar}), we see that $\Phi$ is a Poisson double covering branched along the locus $\mathcal{X}^{(2)}$ of central representations. The intersection $\mathbb{T}_{\mathbf{z}}^{(2)}:= \Phi^{-1}(\mathcal{X}^{(2)})\cap \mathbb{T}_{\mathbf{z}}$ is finite so $\mathbb{T}_{\mathbf{z}} \setminus \mathbb{T}_{\mathbf{z}}^{(2)}$ is connected.
Thus  $\Phi$ induces a double regular Poisson covering $\Phi: \mathbb{T}_{\mathbf{z}} \setminus \mathbb{T}_{\mathbf{z}}^{(2)} \to  \mathcal{X}^{(1)}_c$ so $ \mathcal{X}^{(1)}_c$ is symplectic and connected.
This concludes the proof.
\end{proof}

\subsection{Relative skein modules}

Let $\mathcal{D}$ be a ribbon category and $M$ a compact oriented $3$-manifold. A \textit{marking} is a tuple $\mathcal{P}=\left( (p_1, \epsilon_1, V_1), \ldots, (p_n, \epsilon_n, V_n) \right)$ where $p_i$ is a marked point in $\partial M$, (i.e. a point in $\partial M$ equipped with a non-zero tangent vector in $T_p \partial M$), $\epsilon_i = \pm 1$  and $V_i \in \mathcal{D}$. The \textit{relative skein module} $\mathscr{S}^{\mathcal{D}}(M, \mathcal{P})$ is the quotient of the linear span of $\mathcal{D}$-colored ribbon graphs $\Gamma \subset M$ such that $\partial \Gamma = \mathcal{P}$ by skein relations (as defined in Definition \ref{def_skeinmodule}). Suppose that  $\mathcal{D}$ is $\mathcal{A}$-graded for $\mathcal{A}$ an abelian group. Like in Section \ref{sec_spaces}, for $H$ a handlebody and $\mathcal{P}$ a marking in $\partial H$,  the relative skein module $\mathscr{S}^{\mathcal{D}}(H, \mathcal{P})$ is $\mathrm{H}^1(H;\mathcal{A})$-graded and we denote by $\mathscr{S}^{\mathcal{D}}(H, \mathcal{P}, \omega_H)$ its $\omega_H\in \mathrm{H}^1(H; \mathcal{A})$-graded part. Let $G_{g,n}, G'_{g,n} \subset \mathbb{S}^3$ be the uni-trivalent ribbon graphs of Figure \ref{fig_graphs_open} and choose tubular neighborhoods $H_{g,n}=N(G_{g,n})$ and $H'_{g,n}=N(G'_{g,n})$ such that $G_{g,n}\cap G'_{g,n}=\{ p_1, \ldots, p_n\}$ is the set of vertices of valence $1$ of both $G_{g,n}$ and $G'_{g,n}$ and such that $H_{g,n} \cap H_{g,n}= \partial H_{g,n}=\partial H'_{g,n}$. Note again that for any edge $e$ of $G_{g,n}\cup G'_{g,n}$ then $\gamma_e$ is non-separating.

\begin{figure}[!h] 
\centerline{\includegraphics[width=12cm]{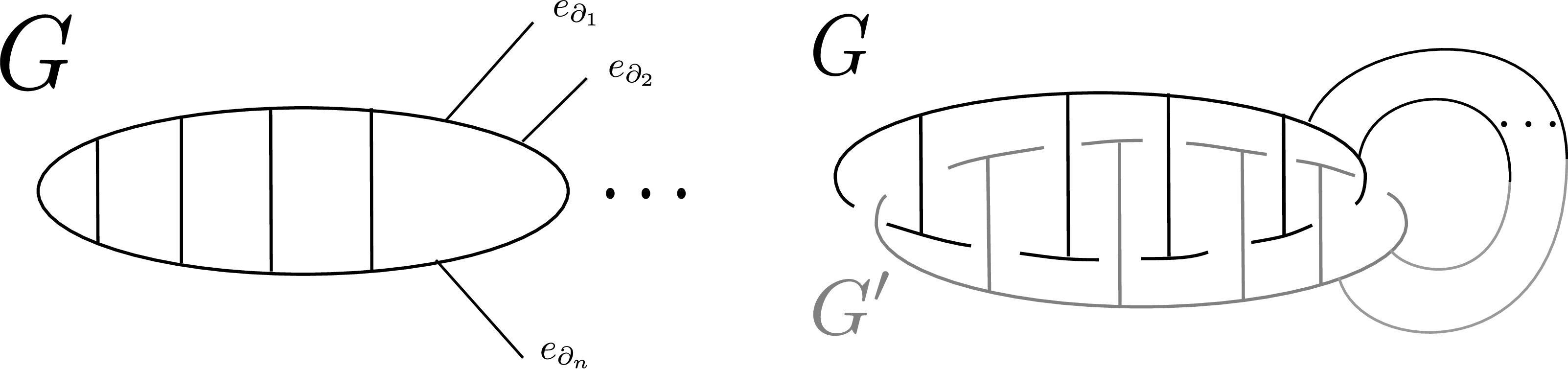} }
\caption{Two uni-trivalent graphs of genus $g$ with $n$ external edges.} 
\label{fig_graphs_open} 
\end{figure}

Let $\Sigma_{g,n}$ be obtained from $\partial H_{g,n}=\partial H'_{g,n}$ by removing a small open disc around each $p_i$.
Suppose that in $\mathcal{P}=\left( (p_i, \epsilon_i, V_i) \right)$, each $V_i$ is homogeneous of degree $a_i \in A$. 
Let $\omega \in \mathrm{H}^1(\Sigma_{g,n}; \mathcal{A})$ be a class such that $\omega([\gamma_i])=a_i$. By restriction, the class $\omega$ induces two classes $\omega_H \in \mathrm{H}^1(H_{g,n}; \mathcal{A})$ and $\omega_{H'}\in \mathrm{H}^1(H'_{g,n}; \mathcal{A})$ such that $\omega$ is completely determined by the tuple $(\omega_H, \omega_{H'}, \{a_i\}_i)$. 
 \par Note that  $\mathscr{S}^{\mathcal{D}}(H_{g,n}, \mathcal{P}, \omega_H)$ and $\mathscr{S}^{\mathcal{D}}(H'_{g,n}, \mathcal{P}, \omega_{H'})$ admit natural structures of left and right $\mathscr{S}^{\mathcal{D}}(\Sigma_{g,n})$-modules respectively and we have a Hopf pairing: 
$$ (\cdot, \cdot)_{\omega}^H : \mathscr{S}^{\mathcal{D}}(H'_{g,n}, \mathcal{P}, \omega_{H'}) \otimes \mathscr{S}^{\mathcal{D}}(H_{g,n}, \mathcal{P}, \omega_H) \to \mathscr{S}^{\mathcal{D}}(\mathbb{S}^3), \quad (\Gamma', \Gamma)_{\omega}^H := \Gamma'\cup \Gamma.$$

\subsection{WRT representations for open surfaces}

Consider the case where $\mathcal{D}=TL$ is the  Cauchy closure of the  Temperley-Lieb category with trivial grading: so as a braided category, $\mathcal{D}$ is equivalent to the category $\mathcal{C}^{small}$ equipped with a modified twist. Consider $\mathcal{P}=\left( (p_1, +, S_{i_1}), \ldots, (p_n, +, S_{i_n}) \right)$ where the punctures are colored by the (non-projective) simple modules $S_{i_k}$ with $i_k \in \{0, \ldots, \frac{N-3}{2}\}$ with non-vanishing q-dimension (as objects in $\mathcal{D}$ they correspond to the Jones-Wenzl idempotents). Using the natural identification $ \mathscr{S}^{TL}(\mathbb{S}^3)\cong \mathbb{C}$ (sending the empty link to $1$) we consider the Hopf pairing with values in $\mathbb{C}$ and define 
$$ V(\Sigma_{g,n}, \mathcal{P}):= \quotient{ \mathscr{S}^{TL}(H_{g,n}, \mathcal{P})}{\Rker (\cdot, \cdot)^H}.$$
The left action of $\mathcal{S}_A(\Sigma_{g,n})=\mathscr{S}^{TL}(\Sigma_{g,n})$ on  $\mathscr{S}^{TL}(H_{g,n}, \mathcal{P})$ induces a representation 
$$ r^{WRT}_{\mathcal{P}}: \mathcal{S}_A(\Sigma_{g,n})\to \End( V(\Sigma_{g,n}, \mathcal{P})).$$
Let $\col(G_{g,n})$ denote the set of maps $c: E(G_{g,n})\to \{0, \ldots, \frac{N-3}{2}\}$ such that $(1)$ $c(e_{\partial_k})=i_k$ where $e_{\partial_k}$ is the edge adjacent to $p_k$ and $(2)$ when $e_i, e_j, e_k$ are three adjacent edges then $c(e_i)+c(e_j)+c(e_k)$ is even and $\leq N-3$ and $c(e_i)+c(e_j)\leq c(e_k)$. For $c\in \col(G_{g,n})$, let $[G, c]\in V(\Sigma_{g,n}, \mathcal{P})$ be the class of the ribbon graph obtained by coloring each edge $e$ of $G_{g,n}$ by $S_{c(e)}$ and each trivalent vertex $v$ with adjacent edges $e_i, e_j, e_k$ by an arbitrary non-zero element of $\Hom(S_{c(e_i)}\otimes S_{c(e_j)}\otimes S_{c(e_k)}, \mathds{1})$. 

\begin{theorem}\label{theorem_BHMV}(\cite{Tu, BHMV2}) The set $\mathcal{B}:= \{ [G, c], c\in \col(G_{g,n})\}$ is a basis of $V(\Sigma_{g,n}, \mathcal{P})$.
\end{theorem}

\begin{lemma}\label{lemma_WRT_open}
\begin{enumerate}
\item $\dim(V(\Sigma_{g,n}, \mathcal{P})) < PI-deg (\mathcal{S}_A(\Sigma_{g,n}))$.
\item $r^{WRT}_{\mathcal{P}}$ is a central representation with classical shadow $([\rho^0], z_1, \ldots, z_n)$ where $\rho^0$ is central and $z_k:= (-1)^{w_S(\gamma_k)}(q^{i_k+1} + q^{-i_k-1})$.
\end{enumerate}
\end{lemma}

\begin{proof}
The first item follows from the inequality $| \col(G_{g,n}) | < N^{3g-3+n}$ together with Theorems \ref{theorem_FKL_open} and \ref{theorem_BHMV}. The proof of the second item is identical to the proof of Proposition \ref{prop_shadow} using the identity $S'( S_i, S_1)=q^{i+1}+q^{-(i+1)}$ of Lemma \ref{lemma_S}. 
\end{proof}

We can now prove the first item of Theorem \ref{theorem_opensurfaces}.

\begin{theorem}\label{theorem_WRT_open} If $x=([\rho], z_1, \ldots, z_n)$ is such that $[\rho]\in \mathcal{X}^{(2)}$ is central and $z_i\neq \pm 2$ for all $1\leq i \leq n$, then $x$ does not belong to the Azumaya locus of $\mathcal{S}_A(\Sigma_{g,n})$.
\end{theorem}

\begin{proof} 
 Let $i_k\in \{0, \ldots, \frac{N-3}{2}\}$ be the unique integer such that $z_k=\pm (q^{i_k+1} + q^{-(i_k+1)})$ and consider $\mathcal{P}=\left( (p_1, +, S_{i_1}), \ldots, (p_n, +, S_{i_n}) \right)$. On the one hand, there exists $\chi \in \mathrm{H}^1(\Sigma_{g,n}; \mathbb{Z}/2\mathbb{Z})$ (which depends on the choice of spin structure when $\varepsilon=+1$) such that $\chi\cdot x$ is the classical shadow of $r_{\mathcal{P}}^{WRT}$. On the other hand, by Lemma \ref{lemma_WRT_open}, $r_{\mathcal{P}}^{WRT}$ is a central representation with dimension strictly smaller than the PI-degree of $\mathcal{S}_A(\Sigma_{g,n})$, therefore its classical shadow does not belong to the Azumaya locus. So $x$ does not either.
\end{proof}

\subsection{BCGP representations for open surfaces}

Now consider the case where $\mathcal{D}=\mathcal{C}$ is the category of representations of the unrolled quantum group. Let $\omega \in \mathrm{H}^1(\Sigma_{g,n}; \mathbb{C}/\mathbb Z)$ and $\alpha_1, \ldots, \alpha_n \in \ddot{\mathbb{C}}$ such that $\overline{\alpha_i} = \omega([\gamma_i])$ and  $\omega(e)\notin (\frac{1}{2k}\mathbb{Z}) /\mathbb{Z}$ for all $e\in \mathring{E}(G)\cup \mathring{E}(G')$ (here $\mathring{E}(G):=E(G)\setminus \{e_{\partial_1}, \ldots, e_{\partial_n}\}$). Consider $\mathcal{P}:=\left( (p_1,+, V_{\alpha_1}), \ldots, (p_n, +, V_{\alpha_n})\right)$ and consider the Hopf pairing 
$$ (\cdot, \cdot)_{\omega}^H : \mathscr{S}^{\mathcal{C}}(H'_{g,n}, \mathcal{P}, \omega_{H'}) \otimes \mathscr{S}^{\mathcal{C}}(H_{g,n}, \mathcal{P}, \omega_H) \to \mathbb{C}, \quad (\Gamma', \Gamma)_{\omega}^H := \left<\Gamma'\cup \Gamma \right>, $$
where $\left<\Gamma'\cup \Gamma \right>$ denotes the CGP invariants of the ribbon graph $\Gamma'\cup\Gamma$. Since $\alpha_i \in \ddot{\mathbb{C}}$, $\Gamma'\cup\Gamma$ has one edge colored by a projective module  so the CGP invariant is well-defined. Set
$$ \mathbb{V}(\Sigma_{g,n}, \mathcal{P}, \omega):= \quotient{ \mathscr{S}^{\mathcal{C}}(H_{g,n}, \mathcal{P}, \omega_H)}{\Rker (\cdot, \cdot)_{\omega}^H}$$
and consider the representation 
$$ r_{\omega, \mathcal{P}}: \mathcal{S}_A(\Sigma_{g,n}) \xrightarrow{f} \mathscr{S}^{\mathcal{C}^{small}}(\Sigma_{g,n}) \to \End(\mathbb{V}(\Sigma_{g,n}, \mathcal{P}, \omega)).$$
Here the spin structure chosen to define $f$ is the same as the one used to identify $\mathcal{S}_{+1}(\Sigma_{g,n})$ with $\mathbb{C}[\mathcal{X}_{\SL_2}(\Sigma_{g,n})]$. Recall that we made the convention that $w_S([\gamma_i])=0$.
Set $z_i:= -\left( q^{ \alpha_i} + q^{- \alpha_i}\right)$ and define $\rho_{\omega}: \pi_1(\Sigma_{g,n}) \to \SL_2$ by the same formula than in Proposition \ref{prop_shadow}. 
Set
$$\col(G):= \{ c: E(G)\to \mathbb{C} \mbox{ such that } c(e_{\partial_i})=\alpha_i \mbox{ and } \overline{c(e)}=\omega(\gamma_e), 0\leq  \Re(c(e)) < N \mbox{ for all }e\in \mathring{E}(G)\} $$
and consider $\mathcal{B}:= \{ [G,c], c\in \col(G)\}$, where $[G,c]$ is defined by fixing a $\sigma$-decoration of $G\cup G'\subset S^3$ as before.

\begin{proposition}\label{prop_BCGP_open}
\begin{enumerate}
\item $\mathcal{B}$ is a basis of $\mathbb{V}(\Sigma_{g,n}, \mathcal{P}, \omega)$.
\item $r_{\omega, \mathcal{P}}$ is a central representation with classical shadow $([\rho_{\omega}], z_1, \ldots, z_n)$ and dimension $N^{3g-3+n}=PI-deg(\mathcal{S}_A(\Sigma_{g,n}))$.
\item if $g\geq 1$, $r_{\omega, \mathcal{P}}$ is irreducible.
\end{enumerate}
\end{proposition}

\begin{proof} The proof is a straightforward adaptation of the proofs of Theorem  \ref{theorem_basis}, Proposition \ref{prop_shadow} and Theorem \ref{theorem2}. We only detail the proof of the last item. Suppose $g\geq 1$, fix $\mathbf{z}=(z_1,\ldots, z_n)$ and 
set 
$$\mathcal{U}_{\mathbf{z}}:=\{ \omega \in \mathrm{H}^1(\Sigma_{g,n}; \mathbb{C}/\mathbb{Z}), \mbox{ such that } \omega(\gamma_i)=T_N(z_i), \omega(\gamma_e)\not\in (\frac{1}{2k}\mathbb{Z})/\mathbb{Z} \mbox{ for all }1\leq i \leq n, e\in \mathring{E}(G) \cup \mathring{E}(G')\}.$$
For every $[\omega] \in \mathcal{U}_{\mathbf{z}}$, then $([\rho_{\omega}], \mathbf{z})$ belongs to the symplectic core $\widehat{\mathcal{X}}^{(1)}\cap \widehat{\nu}^{-1}(\mathbf{z})$, therefore either every representation $r_{\omega, \mathcal{P}}$ with $\omega \in \mathcal{U}_{\mathbf{z}}$ is irreducible or none of them is. Since we made the assumptions $n\geq 1$ and $(g,n)\neq (0, i)$ for $i\leq 3$, the analytic set $\mathcal{U}_{\mathbf{z}}$ has dimension $\geq 1$. If $Z(\omega)\in \mathbb{C}$ is the product of all $6j$-symbols in Notations \ref{notations_Y}, then $Z: \mathcal{U}_{\mathbf{z}} \to \mathbb{C}$ is a non-vanishing analytic function by Proposition \ref{prop_6j}, so there exists $\omega \in \mathcal{U}_{\mathbf{z}}$ for which none of these $6j$-symbol vanishes. The proof of Lemma \ref{lemma_step2} extends word-by-word to the open case and proves that $r_{\omega, \mathcal{P}}$ is irreducible (note that this proof would not work if $g=0$ in which case the curves $\beta_1, \ldots, \beta_g$ used in the proof are not defined).
 Since it is true for one $\omega \in \mathcal{U}_{\mathbf{z}}$, it is true for all of them. This concludes the proof.
\end{proof}

We can now prove the fourth part of Theorem \ref{theorem_opensurfaces}.
\begin{theorem}\label{theorem_open2}
If $g\geq 1$ and $x=([\rho], z_1, \ldots, z_n)$ is such that $(1)$ $[\rho] \in \mathcal{X}^{(1)}$ and  $(2)$ for every $1\leq i \leq n$,  either $\tr( \rho(\gamma_i)) \neq \pm 2$ or $z_i=\pm 2$, then $x$ belongs to the Azumaya locus of $\mathcal{S}_A(\Sigma_{g,n})$.
\end{theorem}

\begin{proof}
Let  $x=([\rho], z_1, \ldots, z_n)$ satisfying the conditions $(1)$ and $(2)$ of the theorem. By $(2)$, for every $1\leq i \leq n$, there exists $\alpha_i\in \ddot{\mathbb{C}}$ such that $z_i= - \left( q^{ \alpha_i} + q^{- \alpha_i}\right)$. Let $\mathcal{P}$ denote the associated marked points. By $(1)$, $x$ belongs to the connected symplectic variety $\widehat{\mathcal{X}}^{(1)}_{\mathbf{z}}$. 
 By Proposition \ref{prop_BCGP_open}, there exists $x'\in \widehat{\mathcal{X}}^{(1)}_{\mathbf{z}}$ which is the classical shadow of a representation $r_{\omega, \mathcal{P}}$  which is irreducible and has dimension $N^{3g-3+n}$ by Proposition \ref{prop_BCGP_open}. Therefore $x'$ belongs to the Azumaya locus of $\mathcal{S}_A(\Sigma_{g,n})$ and so does $x$ by Theorem \ref{theorem_PO} and Lemma \ref{lemma_symp_open}.
\end{proof}

\begin{remark}
The definition we gave of the representations $r_{\omega, \mathcal{P}}$ still makes sense under the weaker hypothesis that either at least one puncture of $\mathcal{P}$ is colored by a projective module or that $\omega \notin \mathrm{H}^1(\Sigma_{g,n}; \frac{1}{2}\mathbb{Z}/2\mathbb{Z})$ (this condition ensures that the CGP invariants used in the Hopf pairing are well-defined). Therefore the BCGP representations could also be used in order to determine the intersection of the Azumaya locus with the whole set $\widehat{\mathcal{X}}^{(0)}\cup \widehat{\mathcal{X}}^{(1)}$. The reason we added hypotheses $(1)$ and $(2)$ in Theorem \ref{theorem_open2} is that without those hypotheses, the space $\mathbb{V}(\Sigma_{g,n}, \mathcal{P}, \omega)$ is no longer spanned by trivalent graphs colored by typical modules $V_{\alpha}, \alpha \in \ddot{\mathbb{C}}$ since we need to considerate graphs colored by modules $S_n$ and $P_n$ as well. However, in \cite[Theorem $5.1$]{BCGP_Bases}, the authors still found basis (in the non-closed case) for these spaces and it seems that our techniques might be generalized to this more general setting as well. 
\end{remark}

\subsection{Particular cases}

\subsubsection{One holed torus}\label{sec_holed_torus}

Let $\alpha_1, \alpha_2, \alpha_3 \subset \Sigma_{1,1}$ be the three curves of Figures \ref{fig_holed_torus} and $\gamma_p$ the peripheral curve around the puncture. Fricke showed in \cite{Fricke_CharVarHoledTorus} (see also \cite{Goldman_RelCharVar}) that the map sending $[\rho]$ to $\varphi([\rho]):=(\tr(\alpha_1), \tr(\alpha_2), \tr(\alpha_3))$ defines an isomorphism $\varphi: \mathcal{X}_{\SL_2}(\Sigma_{1,1}) \xrightarrow{\cong} \mathbb{C}^3$. Moreover, if $\varphi([\rho])=(t_1, t_2, t_3)$ then  $\tr(\rho(\gamma_p))= t_1^2 +t_2^2 +t_3^3 -t_1t_2t_3 -2$, so the sliced character variety writes 
$$ \mathcal{X}_{\SL_2}(\Sigma_{1,1}, c) \cong \{(t_1,t_2,t_3) | t_1^2 + t_2^2 +t_3^2 - t_1t_2t_3-2 =c \}.$$

 \begin{figure}[!h] 
\centerline{\includegraphics[width=7cm]{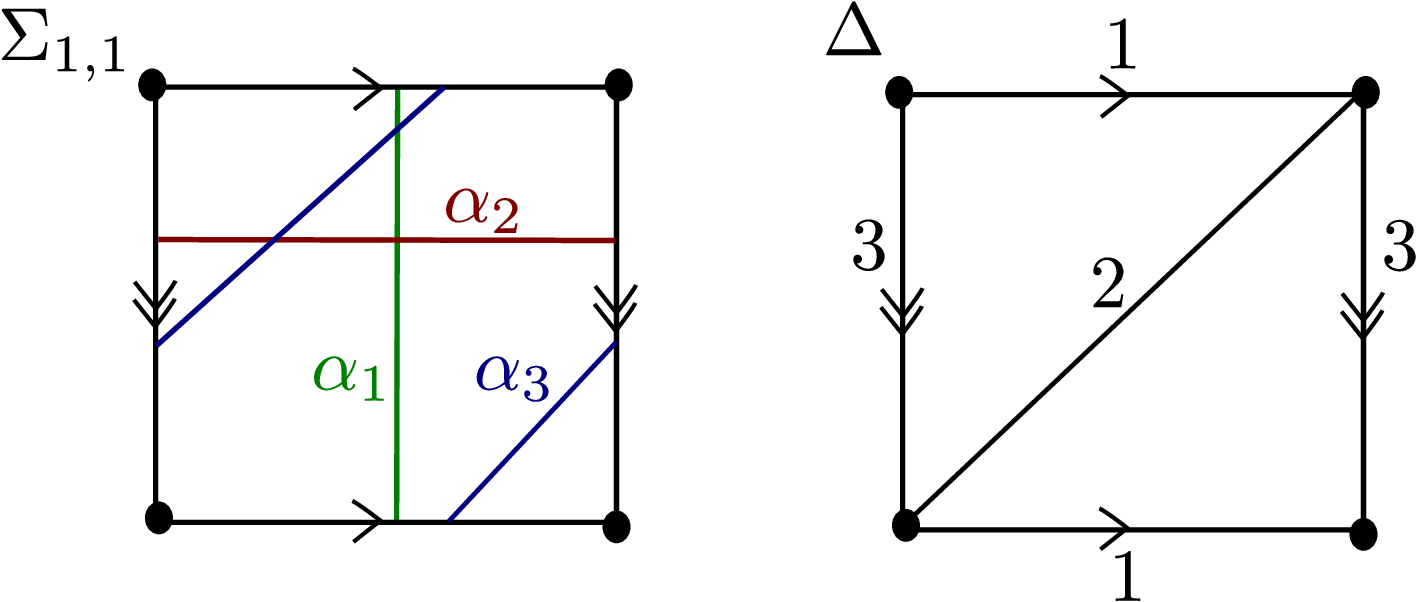} }
\caption{Three curves $\alpha_1, \alpha_2, \alpha_3$ in the holed torus $\Sigma_{1,1}$ and a triangulation $\Delta$.} 
\label{fig_holed_torus} 
\end{figure} 

The \textit{quaternionic representation} $\rho_0 : \pi_1(\Sigma_{1,1}) \to \SL_2$ is defined by 
$$ \rho_0(\alpha_1):=\mathbf{i}= \begin{pmatrix} i & 0 \\ 0 & -i \end{pmatrix}, \quad \rho_0(\alpha_2) := \mathbf{j}=\begin{pmatrix} 0 & 1 \\ -1 & 0 \end{pmatrix}, \quad \rho_0(\alpha_3):=\mathbf{k}=\begin{pmatrix} 0 & i \\ i & 0 \end{pmatrix}.$$
Note that $[\rho_0] \in \mathcal{X}_{\SL_2}(\Sigma_{1,1}, -2)$ has coordinate $\varphi([\rho_0])=(0,0,0)$. For $\varepsilon, \varepsilon'=\pm1$, let $\rho_{\varepsilon, \varepsilon'}$ be the central representation with coordinates $(\varepsilon 2, \varepsilon'2, \varepsilon \varepsilon'2)$.

\begin{theorem}\label{theorem_holed_torus} Let $([\rho], z) \in \widetilde{\mathcal{X}}_{\SL_2}(\Sigma_{1,1})$ be such that $([\rho], z) \neq ([\rho_{\varepsilon, \varepsilon'}], -2)$ for all $\varepsilon, \varepsilon'=\pm1$.
Then $([\rho],z)$ belongs to the  Azumaya locus of $\mathcal{S}_A(\Sigma_{1,1})$
if and only if either $[\rho]$ is  neither central, nor quaternionic, or $[\rho]=[\rho_0]$ and  $z=+ 2$. In particular, the fully Azumaya locus is the locus of classes of representations which are neither central nor quaternionic.
\end{theorem}

The techniques of the present paper are insufficient to determine whether the four points $([\rho_{\varepsilon, \varepsilon'}], -2)$ belong to the Azumaya locus or not (though either all of them or neither of them belong to this locus). 
Soon after the prepublication of the present paper, Yu proved in \cite{Yu_ALSkein_SmallSurfaces} that they do belong to the Azumaya locus.

\begin{lemma}\label{lemma_Torus1} Let $[\rho] \in \mathcal{X}_{\SL_2}(\Sigma_{1,1})$ and set $c:= \tr(\rho(\gamma_p))$. Then $[\rho]$ is a smooth point of $\mathcal{X}_{\SL_2}(\Sigma_{1,1}, c)$ if and only if $[\rho]$ is not central and $[\rho]\neq [\rho_0]$. 
\end{lemma}

\begin{proof} Let $f : \mathbb{C}^3 \to \mathbb{C}$ be defined by $f(t_1,t_2,t_3)= t_1^2 +t_2^2 +t_3^2 - t_1t_2t_3 -2$. Then $\mathcal{X}_{\SL_2}(\Sigma_{1,1}, c) \cong f^{-1}(c)$ so $[\rho]\in \mathcal{X}_{\SL_2}(\Sigma_{1,1}, c)$ with $\varphi([\rho])=(t_1,t_2,t_3)$ is smooth if and only if $D_{(t_1,t_2,t_3)}f \neq (0,0,0)$. We compute: 
$$ \left\{ \begin{array}{l} \partial_{t_1} f (t_1,t_2,t_3) =0 \\  \partial_{t_2} f (t_1,t_2,t_3) =0 \\  \partial_{t_3} f (t_1,t_2,t_3) =0 \end{array} \right.
\Leftrightarrow    \left\{ \begin{array}{l} 2t_1 = t_2t_3 \\ 2t_2=t_1t_3 \\ 2t_3=t_1t_2 \end{array} \right. \Leftrightarrow (t_1,t_2,t_3)= (0,0,0) \mbox{ or }(\varepsilon_1 2, \varepsilon_2 2, \varepsilon_1\varepsilon_22) \mbox{ where }\varepsilon_i=\pm1.$$
This concludes the proof.
\end{proof}

When $\rho$ is a central representation and $z=-q^n -q^{-n}$ for $n\in \{1, \ldots, N-1\}$, then $([\rho], z)$ is not in the Azumaya locus by Theorem \ref{theorem_WRT_open}. Therefore the only remaining cases are the cases $(1)$ of elements $([\rho_{\varepsilon, \varepsilon'}], -2)$  and $(2)$ the quaternionic cases $([\rho_0], +(q^n+q^{-n}))$ with $n=0,\ldots, N-1$. 

\par In \cite{BullockPrzytycki_00}, Bullock and Przytycki proved that $\mathcal{S}_A(\Sigma_{1,1})$ is finitely presented by the generators $\alpha_1, \alpha_2, \alpha_3$ and the following relations: 
\begin{equation}\label{eq_presentations_torus}
 A \alpha_i \alpha_{i+1} - A^{-1} \alpha_{i+1}\alpha_i = (q-q^{-1})\alpha_{i+2} \quad \mbox{,  for } i \in \mathbb{Z}/3\mathbb{Z}.
 \end{equation}
 Let $U'_q\mathfrak{so}_3$ be the algebra with generators $I_1, I_2, I_3$ and relations $A I_i I_{i+1} - A^{-1}I_{i+1}I_i = I_{i+2}$ for $i\in \mathbb{Z}/3\mathbb{Z}$ introduced in  \cite{HavlicekPosta_Uqso3}.
 Obviously, we have an isomorphism $\Psi: \mathcal{S}_A(\Sigma_{1,1}) \cong U_q'\mathfrak{so}_3$ sending $(\alpha_1, \alpha_2, \alpha_3)$ to $((q-q^{-1})I_1, (q-q^{-1})I_2, (q-q^{-1})I_3)$. 
 Large classes of irreducible representations of $U_q'\mathfrak{so}_3$ have been defined and studied by Havl\'i\v{c}ek-Po\v{s}ta
in \cite{HavlicekPosta_Uqso3} from which Takenov deduced in \cite[Theorem $15$]{Takenov_Azumaya} that if $[\rho]$ is neither central, nor quaternionic and  satisfies the additional hypothesis that 
$$(\tr(\rho(\alpha_i)), \tr(\rho(\alpha_{i+1})), \tr(\rho(\alpha_{i+2}))) \neq (\varepsilon_1 2, \varepsilon_2 \frac{2}{\sqrt{3}}, \varepsilon_1 \varepsilon_2 \frac{2}{\sqrt{3}}), \mbox{ for }i\in \mathbb{Z}/3\mathbb{Z}\mbox{ and }\varepsilon_1, \varepsilon_2= \pm 1, $$
 then $[\rho]$ is in the fully Azumaya locus. So  Takenov's result is a weaker form (of the particular case $(g,n)=(1,1)$) of Theorem \ref{theorem_FKL_open}. The following representations were introduced in \cite{HavlicekPosta_Uqso3}.
 
 \begin{definition}(Havl\'i\v{c}ek-Po\v{s}ta representations)
 For $1\leq r \leq N$, the representation $S_r$ has basis $\{v_0, \ldots, v_{r-1} \}$ and module structure: 
\begin{align*}
{}& \alpha_1v_j = -q^{-1/2}\frac{q-q^{-1}}{q^{-j+\frac{r-1}{2}}+q^{j- \frac{r-1}{2}}} ( q[j][j-r] v_{j-1} + v_{j+1}); \\
{}& \alpha_2 v_j = i q^{1/2} \frac{q-q^{-1}}{q^{-j+\frac{r-1}{2}} + q^{j-\frac{r-1}{2}}} ( q^{\frac{r}{2}-j} [j][j-r]v_{j-1} + q^{-\frac{r}{2}+j}v_{j+1} ); \\
{}& \alpha_3 v_j= -i (q^{-j +\frac{r-1}{2}}-q^{j-\frac{r-1}{2}})v_j;
\end{align*}
where we set $v_r:=v_{-1}=0$.
\end{definition}

The following lemma will be used in order to compute the classical shadow of $S_r$.
\begin{lemma}\label{lemma_HP_preliminary}
Let $N\geq 1$ be an odd integer and $M=(M_{ij})_{1\leq i,j \leq N}  \in \Mat_N(\mathbb{C})$ an $N\times N$ matrix. Suppose that 
\begin{enumerate}
\item there exists $t \in \mathbb{C}$ such that $T_N(M) = t\id $ and
\item we have $M_{ij} = 0$ whenever $| i- j | \neq 1$.
\end{enumerate}
Then $t=0$.
\end{lemma}

\begin{proof}
Let us first prove that the second hypothesis implies that $M$ is not invertible by induction on $N$. If $N=1$, the second hypothesis implies that $M=(0)$. Now suppose that the results is true for $N-2$ and let $M'$ be the $(N-2)\times (N-2)$ matrix obtained from $M$ by removing the two first lines and two first columns. Clearly $M'$ satisfies the second hypothesis so is not invertible. By developing the determinant $\det(M)$ first along the first line then along the first column, we obtain the identity $\det(M)= - M_{2,1}M_{1,2} \det(M') = 0$ so $M$ is not invertible. By induction, this is true for all $N$ odd.
\par Now using the first hypothesis and writing $t= x^N + x^{-N}$ for some $x\in \mathbb{C}^*$, the identity $T_N(M)= t$ implies that the spectrum of $M$ is included in the set $\{ q^n x + q^{-n}x^{-1}, n=0, \ldots, N-1\}$. Since $M$ is not invertible, $0$ belongs to this spectrum so there exists $n\geq 0$ such that 
$$ q^n x + q^{-n} x^{-1}= 0 \Leftrightarrow x = \pm i q^{-n} \Rightarrow t = \pm (i-i) = 0.$$
So $t=0$.
\end{proof}

\begin{lemma}\label{lemma_HP}
 $S_r$ is a well-defined irreducible representation with classical shadow $([\rho_0], q^{r} +q^{-r})$
\end{lemma}

\begin{proof} The fact that $S_r$ is a  well-defined irreducible representations follows from the fact that the induced $U'_q\mathfrak{so}_3$-modules obtained using $\Psi: \mathcal{S}_A(\Sigma_{1,1}) \cong U_q'\mathfrak{so}_3$ were proved to be well-defined and irreducible in \cite[Note $4(a)$]{HavlicekPosta_Uqso3} (compare with Equations $(22)$, $(23)$, $(24)$ in \cite{HavlicekPosta_Uqso3} where $v_j$ is denoted $x_{-j}$). Let $C\in U_q'\mathfrak{so}_3$ be the central element defined by $C:= q^2I_1^2 + I_2^2 +q^2I_3^2 - (q^{5/2}-q^{1/2})I_1I_2I_3$. Then it is proved in \cite{HavlicekPosta_Uqso3} that $C$ acts by the scalar $c=-q[\frac{r-1}{2}][\frac{r+1}{2}]$. Since $\gamma_p=-q^{-1}(q-q^{-1})^2\Psi^{-1}(C) +(q+q^{-1})$ we deduce that $\gamma_p$ acts by the scalar $z:=q^r + q^{-r}$. 
\par Suppose that $T_N(\alpha_i)$ acts by $t_i \id$ on $S_r$. One has: 
$$ T_N(\alpha_3) v_j = T_N\left( (-iq^{-j+\frac{r-1}{2}}) + (-iq^{-j+\frac{r-1}{2}})^{-1}\right) v_j = \left( (-iq^{-j+\frac{r-1}{2}})^N + (-iq^{-j+\frac{r-1}{2}})^{-N} \right) v_j=0.$$
So $t_3=0$. Applying Lemma \ref{lemma_HP_preliminary} to the matrices $\rho(\alpha_1)$ and $\rho(\alpha_2)$, we obtain that $t_1=t_2=0$.
So $S_r$ has classical shadow $([\rho_0], q^r +q^{-r})$. 

\end{proof}

\begin{proof}[Proof of Theorem \ref{theorem_holed_torus}]  If $[\rho]$ is not central nor quaternionic, then $[\rho]\in \mathcal{FAL}$ by Theorem \ref{theorem_FKL_open} and  Lemma \ref{lemma_Torus1}.   
When $[\rho]$ is central and $z\neq -2$, then $([\rho],z)$ does not belong to the Azumaya locus by Theorem \ref{theorem_WRT_open}.
When $[\rho]=[\rho_0]$ is  quaternionic and $r\in \{1, \ldots, N\}$  then by Lemma \ref{lemma_HP}, $([\rho_0], + (q^{r} + q^{-r}))$ is the classical shadow of the irreducible representation $S_r$ of dimension $r$ so it belongs to the Azumaya locus if and only if $r=N$. This concludes the proof.

\end{proof}

\subsection{Indecomposable semi-weight representations}

Recall that $Z_{\Sigma}$ is the center of $\mathcal{S}_A(\Sigma)$, that $Z_{\Sigma}^0\subset Z_{\Sigma}$ is the image of the Chebyshev morphism. We defined a weight module as a module which is semisimple over $Z_{\mathbf{\Sigma}}$. For open surfaces, we rather consider the following larger class.
\begin{definition}(Semi-weight modules) A module over $\mathcal{S}_A(\Sigma)$ is a \textit{semi-weight module} if it is semisimple as a module over $Z_{\mathbf{\Sigma}}^0$.
\end{definition}
So in semi-weight modules, the peripheral curves do not necessarily acts semi-simply. For instance the representations coming from non-semisimple TQFTs (\cite{BCGPTQFT}) where we chose a surface with punctures colored by the projectives $P_n$ are semi-weight modules which are not weight (because the Casimir does not act semi-simply on $P_n$).
 Recall that we identify $\Specm(Z_{\Sigma}^0)$ with the character variety $\mathcal{X}_{\SL_2}(\Sigma)$ so an indecomposable semi-weight module has a {classical shadow} which lies in $\mathcal{X}_{\SL_2}(\Sigma)$. For $[\rho] \in \mathcal{X}_{\SL_2}(\Sigma)$, we denote by $\mathfrak{m}_{[\rho]} \subset Z_{\Sigma}^0$ the corresponding maximal ideal and write: 
$$ Z([\rho]) := \quotient{Z_{\Sigma}}{\mathfrak{m}_{[\rho]}Z_{\Sigma}}, \quad \mathcal{S}_A(\Sigma)_{[\rho]}:= \quotient{\mathcal{S}_A(\Sigma)}{\mathfrak{m}_{[\rho]}\mathcal{S}_A(\Sigma)}.$$

\begin{theorem}(Brown-Gordon\cite[Corollary $2.7$]{BrownGordon_ramificationcenters})\label{theorem_BG_SW} If $[\rho]$ is in the fully Azumaya locus of $\mathcal{S}_A(\Sigma)$, then we have an algebra isomorphism $\mathcal{S}_A(\Sigma)_{[\rho]} \cong \Mat_D(Z([\rho]))$.
\end{theorem}

\begin{lemma}\label{lemma_SW} Let $[\rho]\in \mathcal{X}_{\SL_2}(\Sigma_{g,n})$. For $1\leq i \leq n$ write $c_i:= \tr(\rho(\gamma_i))$ and let $$\mathcal{A}_i:= \left\{ \begin{array}{ll} \mathbb{C}^{\oplus N} & \mbox{if }c_i \neq \pm 2; \\ \mathbb{C} \oplus \left(\quotient{\mathbb{C}[X]}{(X)^2} \right)^{\oplus (N-1)/2} & \mbox{if }c_i=\pm 2. \end{array} \right.$$
Then 
$$ Z([\rho]) \cong \mathcal{A}_1\otimes \ldots \otimes \mathcal{A}_n.$$
\end{lemma}

\begin{proof} 
Let us first make a simple remark. Let $P(X)\in\mathbb{C}[X]$ be a polynomial with decomposition $P(X)=\prod_{i=1}^d (X-\lambda_i)^{n_i}$ where $n_i \geq 1$ and the $\lambda_i$ are pairwise distinct. 
\par  \underline{Claim:} We have an isomorphism of algebras:
 $$ \quotient{\mathbb{C}[X]}{(P(X))} \cong \oplus_{i=1}^d \quotient{\mathbb{C}[X]}{ (X^{n_i})}.$$
 Indeed, since the ideals $(X-\lambda_i)^{n_i}$ are pairwise coprime, by the chinese reminder theorem we have an isomorphism $ \quotient{\mathbb{C}[X]}{(P(X))} \cong \oplus_{i=1}^d \quotient{\mathbb{C}[X]}{ \left((X-\lambda_i)^{n_i}\right)}$ and setting $Y:=X-\lambda_i$ one has $\quotient{\mathbb{C}[X]}{ (X-\lambda_i)^{n_i}}\cong \quotient{\mathbb{C}[Y]}{ (Y)^{n_i}}$.
\par 
Recall from Theorem \ref{theorem_FKL_open} that ${Z}_{\Sigma} \cong \quotient{ {Z}_{\Sigma}^0 [X_1, \ldots, X_n]}{(T_N(X_i)- \gamma_i )}$ from which we deduce that 
$$ Z([\rho]) \cong  \quotient{ \mathbb{C} [X_1, \ldots, X_n]}{(T_N(X_i)+c_i)}\cong \otimes_{i=1}^n \quotient{\mathbb{C}[X]}{(T_N(X)+c_i)}.$$
If $c_i\neq \pm 2$, then the polynomial $T_N(X)+c_i$ has its zeroes of multiplicity one, so $ \quotient{\mathbb{C}[X]}{(T_N(X)+c_i)}\cong \mathbb{C}^{\oplus N}$. If $c_i = \pm 2$, then $T_N(X) \pm 2 = (X \pm 1)\prod_{i=1}^{(N-1)/2}(X\pm (q^i +q^{-i}))^2$ so $\quotient{\mathbb{C}[X]}{(T_N(X)+c_i)} \cong \mathbb{C}\oplus \left(\quotient{\mathbb{C}[X]}{(X^2)} \right)^{\oplus (N-1)/2}$. This concludes the proof.
\end{proof}

\begin{corollary}\label{coro_dimension} If $[\rho] \in \mathcal{X}_{\SL_2}(\Sigma_{g,n})$ is in the fully Azumaya locus of $\mathcal{S}_A(\Sigma_{g,n})$ then we have 
$$\dim\left(  \mathcal{S}_A(\Sigma_{g,n})_{[\rho]} \right) = N^{6g-6+3n}.$$
\end{corollary}

\begin{proof}
Let $D=PI-deg(\mathcal{S}_A(\Sigma_{g,n}))$.
By Theorem \ref{theorem_BG_SW}, we have
$$ \mathcal{S}_A(\Sigma_{g,n})_{[\rho]} \cong \Mat_D(Z([\rho])) \cong \Mat_D(\mathbb{C})\otimes Z([\rho]).$$
By Lemma \ref{lemma_SW}, $\dim(Z([\rho]))=N^n$ so 
$$\dim\left(  \mathcal{S}_A(\Sigma_{g,n})_{[\rho]} \right) = N^{n} D^2= N^{n} (N^{3g-3+n})^2 = N^{6g-6+3n}.$$ 
\end{proof}

Theorem \ref{theorem_BG_SW} can be used to produce new indecomposable semi-weight $\mathcal{S}_A(\Sigma_{g,1})$-modules which are  not weight modules whenever we can find $[\rho]$ in the fully Azumaya locus such that $c:=\tr(\rho(\gamma_p))=-2$ (when $n> 1$, such a $[\rho] \in \mathcal{X}_{\SL_2}(\Sigma_{g,n})$ cannot exist in virtue of Lemmas \ref{lemma_X2_open} and \ref{lemma_z2}).  For instance consider the case $(g,n)=(1,1)$ and let $[\rho]\in \mathcal{X}_{\SL_2}(\Sigma_{1,1})$ be such that $c= \tr(\rho(\gamma_p))=-2$ and $[\rho]\neq [\rho_0]$. By Theorem \ref{theorem_holed_torus}, $[\rho]$ is in the fully Azumaya locus so using Theorem \ref{theorem_BG_SW} and Lemma \ref{lemma_SW}, for every $n\in \{1, \ldots, (N-1)/2\}$ one has: 
$$ \quotient{\mathcal{S}_A(\Sigma_{1,1})_{[\rho]}}{(\gamma_p- q^n -q ^{-n})^2\mathcal{S}_A(\Sigma_{1,1})_{[\rho]}} \cong \Mat_D \left(\quotient{\mathbb{C}[X]}{(X)^2} \right).$$
Let $W$ be the two dimensional representation of $\quotient{\mathbb{C}[X]}{(X^2)}$ with basis $\{x, y\}$ such that $X \cdot x= 0$ and $X \cdot y=x$ (so $W$ is free of rank $1$). Let $V$ be the $N$-dimensional simple module of $\Mat_N(\mathbb{C})$ and consider the representation
 $$\Mat_N\left( \quotient{\mathbb{C}[X]}{(X^2)}\right)\cong \Mat_N(\mathbb{C}) \otimes \left(\quotient{\mathbb{C}[X]}{(X^2)}\right) \to \End(V)\otimes \End(W) \cong\End(V\otimes W).$$
   We thus obtain an indecomposable representation of dimension $2N$:  
   $$ r_n : \mathcal{S}_A(\Sigma_{1,1}) \to  \quotient{\mathcal{S}_A(\Sigma_{1,1})_{[\rho]}}{(\gamma_p- q^n -q ^{-n})^2\mathcal{S}_A(\Sigma_{1,1})_{[\rho]}} \cong \Mat_D \left(\quotient{\mathbb{C}[X]}{(X^2)} \right) \to \End(V\otimes W).$$
   This representation is not a weight representation since $r_n(\gamma_p)$ is not semisimple and contains the Azumaya irreducible representation with classical shadow $([\rho], q^n +q^{-n})$ as a sub-representation.

\section{Applications}

\subsection{Representations of the Torelli group}\label{sec_MCG}
Suppose $g\geq 1$.
Let $\Mod(\Sigma_{g,n})$ be the mapping class group of $\Sigma_{g,n}$ and let $\mathcal{I}(\Sigma_{g,n}) \subset \Mod(\Sigma_{g,n})$ denote its Torelli subgroup. By definition, a mapping class belongs to $\mathcal{I}(\Sigma_{g,n})$ if it acts trivially in homology. Define a left action of $\Mod(\Sigma_{g,n})$ on $\mathcal{S}_A(\Sigma_{g,n})$ by $\phi\cdot [\gamma]:= [\phi(\gamma)]$ for $\phi$ a mapping class and $\gamma\subset \Sigma_g$ a multicurve. The mapping class group also acts on $\mathcal{X}_{\SL_2}(\Sigma_{g,n})$ on the right by the formula $\rho \cdot \phi (\gamma) = \rho (\phi(\gamma)) $ for $\rho: \pi_1(\Sigma_{g,n})\to \SL_2$ and $\phi \in \Mod(\Sigma_{g,n})$. Note that the composition 
$$ \mathcal{O}(\mathcal{X}_{\SL_2(\Sigma_{g,n})}) \cong \mathcal{S}_{\varepsilon}(\Sigma_{g,n}) \xrightarrow{ Ch_A} \mathcal{S}_A(\Sigma_{g,n})$$
is equivariant for the associated left $\Mod(\Sigma_{g,n})$ action. 
 Let $\omega \in \mathrm{H}^1(\Sigma_{g,n}; \mathbb{C}/ \mathbb{Z})\setminus \mathrm{H}^1(\Sigma_{g,n}; (\frac{1}{2k}\mathbb{Z})/\mathbb{Z})$ and $\mathbf{z}$ compatible as before, and consider the associated representation $\rho_{\omega}: \pi_1(\Sigma_{g,n}) \to \SL_2$ as before.  Since every element of $\mathcal{I}(\Sigma_{g,n})$ fixes $[\rho_{\omega}]$ (by definition of the Torelli group), the action of $\mathcal{I}(\Sigma_{g,n})$ on $\mathcal{S}_A(\Sigma_{g,n})$ passes to the quotient to an action on $\mathcal{S}_A(\Sigma_{g,n})_{([\rho_{\omega}], \mathbf{z})}$.  By Theorems \ref{main_theorem} and \ref{theorem_opensurfaces}, $([\rho_{\omega}], \mathbf{z})$ belongs to the Azumaya locus of $\mathcal{S}_A(\Sigma_{g,n})$ so $\mathcal{S}_A(\Sigma_{g,n})_{([\rho_{\omega}], \mathbf{z})} \xrightarrow[\cong]{r_{\omega, \mathcal{P}}}\End(\mathbb{V}(\Sigma_{g,n}, \mathcal{P}, \omega))$ is a matrix algebra and since every automorphism of a matrix algebra is inner, for every $\phi \in \mathcal{I}(\Sigma_{g,n})$, there exists an operator $\pi_{\omega, \mathcal{P}} (\phi) \in \End(\mathbb{V}(\Sigma_{g,n}, \mathcal{P}, \omega))$, unique up to multiplication by a non-zero scalar, such that 
 $$ \phi\cdot x = \pi_{\omega, \mathcal{P}}(\phi)^{-1} x \pi_{\omega, \mathcal{P}}(\phi), \quad \mbox{ for all } x\in \mathcal{S}_A(\Sigma_{g,n})_{([\rho_{\omega}], \mathbf{z})}.$$
 The assignation $\phi \mapsto \pi_{\omega, \mathcal{P}}(\phi)$ defines a projective representation
 $$ \pi_{\omega, \mathcal{P}} : \mathcal{I}(\Sigma_{g,n}) \to \PGL(\mathbb{V}(\Sigma_{g,n}, \mathcal{P},  \omega))$$
 of the Torelli group. This representation (or rather its variant where $q$ has order $2N$) first appeared in \cite{BCGPTQFT} where the whole structure of (non-semisimple) TQFT were used to construct it. The alternative definition we gave here, based on Theorems \ref{main_theorem} and \ref{theorem_opensurfaces}, is much simpler. 
 
 \subsection{Reduced skein modules of $3$-manifolds}\label{sec_onedim}
 
 Let $M$ be a compact connected $3$-manifold. As proved by L\^e in \cite{LeKauffmanBracket}, the Chebyshev-Frobenius morphism $Ch_A: \mathcal{S}_{\varepsilon}(M) \to \mathcal{S}_A(M)$ is still well-defined for $3$-manifolds (though not necessarily injective): it sends a framed link $L=\gamma_1\cup \ldots \cup \gamma_n \subset M$ to $Ch_A(L):= T_N(\gamma_1)\cup \ldots \cup T_N(\gamma_p)$. A framed link colored by $T_N$ is \textit{transparent} in the sense that if $L_1, L_2\subset M$ are two framed links, then the class $[Ch_A(L_1)\cup L_2]\in \mathcal{S}_A(M)$ does not depend on how $L_1$ and $L_2$ are entangled. So 
  $Ch_A$ endows $\mathcal{S}_A(M)$ with a structure of $\mathcal{S}_{\varepsilon}(M)$-module via $[L_1]\rhd [L_2]:=[Ch_A(L_1)\cup L_2]$. Fixing an identification $\Psi: \mathcal{O}[\mathcal{X}_{\SL_2}(M)] \cong \mathcal{S}_{\varepsilon}(M)$ as before, we thus obtain that $\mathcal{S}_A(M)$ defines a coherent sheaf $\mathscr{L}^M$ over $\mathcal{X}_{\SL_2}(M)$. The section of the fiber $\restriction{\mathscr{L}^M}{[\rho]}$ over a point $[\rho]$ will be called the \textit{reduced skein module}.
 
  \begin{definition} Let $[\rho]\in \mathcal{X}_{\SL_2}(M)$ be a closed point which corresponds to a maximal ideal $\mathfrak{m}_{[\rho]} \subset \mathcal{S}_{\varepsilon}(M)$ through $\Psi$. The \textit{reduced skein module} is the quotient
 $$ \mathcal{S}_A(M)_{[\rho]}:= \quotient{\mathcal{S}_A(M)}{Ch_A(\mathfrak{m}_{[\rho]})\mathcal{S}_A(M)}.$$
 \end{definition}

  \begin{theorem}\label{theorem_reduced} Let $[\rho]\in \mathcal{X}_{\SL_2}(M)$ be the class of a non-central representation and $M$ a closed $3$-manifold. Then the reduced skein module $ \mathcal{S}_A(M)_{[\rho]}$ is one-dimensional. 
\end{theorem}
Said differently, the restriction of $\mathscr{L}^M$ to the locus of non-central representations is a line bundle. 
Theorem \ref{theorem_reduced} was proved in \cite[Theorem $12.2$]{FKL_GeometricSkein} under the weaker hypothesis that $[\rho]$ is irreducible and mostly relied on the fact that for closed surfaces the Azumaya loci of skein algebras contain irreducible representations. Our proof is very similar to the proof of \cite[Theorem $12.2$]{FKL_GeometricSkein} where we use Theorem \ref{main_theorem}. Since the proof is short and enlightening, we repeat it for self-completeness. 
\par The following proposition is interesting on its own and is an extension of \cite[Theorem $12.1$]{FKL_GeometricSkein}. The inclusion $\partial M \subset M$ induces a morphism $\mathcal{X}_{\SL_2}(M) \to \mathcal{X}_{\SL_2}(\partial M)$ sending $[\rho]$ to, say, $[\rho_{\partial}]$.

\begin{proposition}\label{prop_reduced} Let $M$ be a $3$-manifold homeomorphic to a thickened surface $\Sigma \times [0,1]$. Let $[\rho] \in \mathcal{X}_{\SL_2}(M)$ such that the restriction of $\rho_{\partial}$ to every connected component of $\partial M$ is non-central. Then $\mathcal{S}_A(M)_{[\rho]}$ is simple as a module over $\mathcal{S}_A(\partial M)_{[\rho_{\partial}]}$. 
\end{proposition}

We will use the following extension of \cite[Proposition $12.3$]{FKL_GeometricSkein}.

\begin{lemma}\label{lemma_reduced}
Let $H_g$ be a genus $g$ handlebody and $\gamma_1, \ldots, \gamma_g$ some free generators of $\pi_1(H_g)$. Let $\rho: \pi_1(H_g)\to \SL_2$ a representation whose class $[\rho]\in \mathcal{X}_{\SL_2}(H_g)$ is non-central. Then there exists a homeomorphism $\phi\in \Mod(H_g)$ which satisfies one of the following conditions: 
\begin{enumerate}
\item either $\rho$ is irreducible or $g=1$  and $\tr(\rho (\phi(\gamma_i)))\neq \pm 2$ for all $1\leq i \leq g$ and $\tr(\rho(\phi(\gamma_1\ldots \gamma_g)))\neq \pm 2$; 
\item or $\rho$ is diagonal, $g\geq 2$ and $\tr(\rho (\phi(\gamma_i)))\neq \pm 2$ for all $1\leq i \leq g$; 
\item or $\rho$ is diagonal, $g\geq 2$ and $\tr(\rho (\phi(\gamma_i)))\neq \pm 2$ for all $2\leq i \leq g$ and $\tr(\rho(\phi(\gamma_1\ldots \gamma_g)))\neq \pm 2$. 
\end{enumerate}
\end{lemma}

\begin{proof} If $g=1$, the lemma is tautological. 
$(1)$ If $\rho: \pi_1(H_g)\to \SL_2$ is irreducible, then the lemma is proved in \cite[Proposition $12.3$]{FKL_GeometricSkein} so Assertion $(1)$ holds. 
\par $(2)$ Assume that $\rho$ is diagonal non-central, so $\rho(\gamma_i) =\begin{pmatrix} \lambda_i & 0 \\ 0 & \lambda_i^{-1} \end{pmatrix}$ for some $\lambda_i \in \mathbb{C}^*$.  Let $I= \{ i \in \{1, \ldots, g\} \mbox{ such that }\lambda_i=\pm 1\}$. Since $\rho$ is not central, there exists an index $i$ such that $i\notin I$. 
Up to permuting the indices, we can thus suppose that $1\notin I$. Recall that the map $\Mod(H_g) \to \Aut(\mathbb{F}_g)$ sending a mapping class to the corresponding automorphism of the free group $\mathbb{F}_g=\pi_1(H_g)=\left< \gamma_1, \ldots, \gamma_g\right>$ is surjective. 
Let $\phi\in \Mod(H_g)$ be a mapping class corresponding to the automorphism of $\mathbb{F}_g$ sending $\gamma_i$ to $\gamma_i \gamma_1$ if $i\in I$ and $\gamma_j$ to $\gamma_j$ if $j\notin I$. Then $\tr(\rho (\phi(\gamma_i)))\neq \pm 2$ for all $1\leq i \leq g$ and we have proved Assertion $(2)$. 
\par $(3)$ Eventually assume that $\rho$ is diagonal non-central, so by $(2)$ up to automorphism, we can suppose that $\rho(\gamma_i)\neq \pm \mathds{1}$ for $1\leq i \leq g$. If $\rho(\gamma_1\ldots \gamma_g)=\pm \mathds{1}_2$, i.e. if $\lambda_1\ldots \lambda_n=\pm 1$, then we precompose $\rho$ by the mapping class corresponding to the automorphism $\phi$ of $\mathbb{F}_g$ sending $\gamma_1$ to $\gamma_1\gamma_2$ and $\gamma_i$ to $\gamma_i$ for $i\neq 1$. 
Then $\psi \circ \phi$ satisfies the conclusion of Assertion $(3)$.
\end{proof}

\begin{lemma}\label{lemma_reduced2}
Let $H_g$ be a genus $g$ handlebody and $\rho: \pi_1(H_g)\to \SL_2$ a non-central representation. Then $\dim\left( \mathcal{S}_A(H_g)_{[\rho]}\right)= N^{3g-3}$. 
\end{lemma}

\begin{proof}
By Corollary \ref{coro_dimension}, it suffices to find an open surface $\Sigma_{g_0,n}$ and an homeomorphism $\phi: H_g\cong \Sigma_{g_0,n}\times [0,1]$ such that the class of the induced representation $\rho_{\Sigma}:=\rho\circ \phi_*: \pi_1(\Sigma_{g_0,n})\to \SL_2$ lies in the fully Azumaya locus of $\mathcal{S}_A(\Sigma_{g_0,n})$. Indeed, an homeomorphism $H_g\cong \Sigma_{g_0,n}\times [0,1]$ exists if and only if $2g_0+n-1=g$ in which case the conclusion of Corollary \ref{coro_dimension} writes:
$$ \dim(\mathcal{S}_A(H_g)_{[\rho]}) = \dim(\mathcal{S}_A(\Sigma_{g_0, n})_{[\rho_{\Sigma}]}) = N^{6g_0+3n-6}= N^{3g-3}.$$

 Note that  $\rho_{\Sigma}$ is non-central since $\rho$ is non-central. If $g=1$, we chose $\Sigma_{g_0,n}=\Sigma_1$ and since, by Corollary \ref{coro_genusOne},  any non-central representation of $\pi_1(\Sigma_1)$ is fully Azumaya, we have the desired equality. If $g\geq 2$ and $\rho$ is irreducible, we  choose $\Sigma_{g_0,n}=\Sigma_{0, g-1}$. By Lemma \ref{lemma_reduced} $(1)$, we can choose the homeomorphism $H_g\cong \Sigma_{0,g-1}\times [0,1]$ such that $\tr(\rho_{\Sigma}(\gamma_{p}))\neq \pm 2$ for all inner puncture $p$ so $\rho_{\Sigma}$ is fully Azumaya by Corollary \ref{coro_FKL_open}. If $g\geq 2$ and $\rho$ is diagonal non-central, we choose $\Sigma_{g_0,n}=\Sigma_{1,g-1}$. Again, by Lemma \ref{lemma_reduced} $(3)$, we can choose the homeomorphism $H_g\cong \Sigma_{1,g-1}\times [0,1]$ such that $\tr(\rho_{\Sigma}(\gamma_{p}))\neq \pm 2$ for all inner puncture $p$ so $\rho_{\Sigma}$ is fully Azumaya by Theorem \ref{theorem_open2}. We conclude using the fact that in the GIT quotient $\mathcal{X}_{\SL_2}(H_g)$, any non-central representation $\rho$ has the same class that a representation which is either diagonal or irreducible. 

\end{proof}

\begin{proof}[Proof of Proposition \ref{prop_reduced}]
Without loss of generality, we can suppose that $\Sigma$ is connected. Let $D_{\Sigma}$ and $D_{\partial M}$ the PI-degrees of $\mathcal{S}_A(\Sigma)$ and $\mathcal{S}_A(\partial M)$.  By hypothesis, $[\rho_{\partial}]$ is in the Azumaya locus of $\mathcal{S}_A(\partial M)$ so it suffices to prove that $\dim\left( \mathcal{S}_A(M)_{[\rho]} \right) = D_{\partial M}$. 
The identification $M\cong \Sigma \times [0,1]$ induces an isomorphism $\mathcal{X}_{\SL_2}(M)\cong \mathcal{X}_{\SL_2}(\Sigma)$.  If $\Sigma$ is closed, then $[\rho]$ is in the Azumaya locus of $\mathcal{S}_A(\Sigma)$ so $\mathcal{S}_A(M)_{[\rho]}\cong \Mat_{D}(\mathbb{C})$ has dimension $D_{\Sigma}^2$. Since $\partial M =- \Sigma \sqcup \Sigma$ in that case, we have $D_{\partial M}= D_{\Sigma}^2$ as well. 
\par Now if $\Sigma=\Sigma_{g_0,n}$ for $n\geq 1$, set $g:=2g_0+n-1$.  Then $M \cong \Sigma_{g_0,n}\times [0,1] \cong H_{g}$ and $\partial M \cong \Sigma_{g}$ so $D_{\partial M}= N^{3g-3}$. 
 Thus by Lemma  \ref{lemma_reduced2} we have $\dim (\mathcal{S}_A(M)_{[\rho]})= D_{\partial M}$.
This concludes the proof. 

\end{proof}
 
 \begin{proof}[Proof of Theorem \ref{theorem_reduced}]
  Let $M$ be a closed oriented $3$ manifold and consider a genus $g$ Heegaard splitting $M=H_g^{(1)} \cup_{\Sigma_g} H_g^{(2)}$ where $H_g^{(i)}$ are genus $g$ handlebodies glued along some oriented homeomorphisms $\phi_1: \Sigma_g \cong -\partial H_g^{(1)}$ and $\phi_2: \Sigma_g \cong \partial H_g^{(2)}$. The attaching maps endow $\mathcal{S}_A(H_g^{(1)})$ and $\mathcal{S}_A(H_g^{(2)})$ with some structures of right and left $\mathcal{S}_A(\Sigma_g)$-modules respectively and a classical argument of Hoste and Przyticki (detailed in \cite{Przytycki_skein}) shows that the inclusion morphisms $\mathcal{S}_A(H_g^{(i)}) \to M$ induce an isomorphism 
 \begin{equation}\label{eq_HP}
  \mathcal{S}_A(M) \cong \mathcal{S}_A(H_g^{(1)}) \otimes_{\mathcal{S}_A({\Sigma}_g)} \mathcal{S}_A(H_g^{(2)}).
  \end{equation}
Similarly, at the classical level, the Van Kampen theorem asserts that we have a push-out square of groups
 $$ \begin{tikzcd}
  \pi_1(\Sigma_g, v)
   \ar[r, "(\phi_1)_*"] \ar[d, "(\phi_2)_*"] 
   & 
   \pi_1(H_g^{(1)}, v) \ar[d] \\
  \pi_1(H_g^{(2)},v ) \ar[r] & \pi_1(M, v)
  \end{tikzcd}
  $$
which induces an isomorphism 
$$ \mathcal{O}[\mathcal{X}_{\SL_2}(M)]\cong  \mathcal{O}[\mathcal{X}_{\SL_2}(H_g^{(1)})]\otimes_{ \mathcal{O}[\mathcal{X}_{\SL_2}(\Sigma_g)]}  \mathcal{O}[\mathcal{X}_{\SL_2}(H_g^{(2)})]$$
as proved, for instance, in \cite[Proposition $2.10$]{MarcheCharVarSkein}. Through this isomorphism, the isomorphism of \eqref{eq_HP} becomes an isomorphism of $ \mathcal{O}[\mathcal{X}_{\SL_2}(M)]$-modules.
 Let  $[\rho_i] \in \mathcal{X}_{\SL_2}(H_g^{(i)})$ and $[\rho_{\Sigma}] \in \mathcal{X}_{\SL_2}(\Sigma_g)$ the classes of the representations obtained by precomposing  $\rho$ with the group morphisms $\pi_1(H_g^{(i)}, v) \to \pi_1(M, v)$ and $\pi_1(\Sigma_g, v)\to \pi_1(M,v)$. Since $[\rho]$ is non-central, so are $[\rho_{\Sigma}]$ and $[\rho_i]$. By Theorem \ref{main_theorem}, $[\rho]$ is in the Azumaya locus of $\mathcal{S}_A(\Sigma_g)$. Therefore $\mathcal{S}_A(\Sigma_g)_{[\rho_{\Sigma}]}\cong \Mat_D(\mathbb{C})$ and $\mathcal{S}_A(H_g^{(i)})_{[\rho_i]}$ is isomorphic to the standard simple (left or right) module of $\Mat_D(\mathbb{C})$ by Proposition \ref{prop_reduced}. Therefore
 one obtains isomorphisms:
 $$  \mathcal{S}_A(M)_{[\rho]} \cong \mathcal{S}_A(H_g^{(1)})_{[\rho_1]} \otimes_{\mathcal{S}_A(\Sigma_g)_{[\rho_{\Sigma}]}} \mathcal{S}_A(H_g^{(2)})_{[\rho_2]} \cong \mathbb{C}^D \otimes_{\Mat_D(\mathbb{C})} \mathbb{C}^D \cong \mathbb{C}.$$
 This concludes the proof.

 \end{proof}
 
  \subsection{Skein modules reduced at central representations}
  
  Let $\mathcal{S}_v(M)$ be the skein module over the ring $K:=\mathbb{Q}[v^{\pm 1}]$, so here $v$ is a generic formal parameter and $\mathcal{S}_A(M)=\mathcal{S}_v(M)\otimes_K \mathbb{C}_A$ (here $\mathbb{C}_A$ is $\mathbb{C}$ with the $K$-module structure obtained by replacing $v$ by $A$). Following  \cite{DetcherryKalfagianniSikora_SkeinDim}, we say that the skein module $\mathcal{S}_v(M)$ is $A$-\textit{tame} if is a direct sum of a cyclic $K$-modules which does not contain any summand of the form $\quotient{K}{(\phi_N(v))}$ if $\varepsilon = +1$ or of the form $\quotient{K}{(\phi_{2N}(v))}$ if $\varepsilon=-1$. Here $\phi_N(X)$ is the cyclotomic polynomial and the summands we exclude are the ones which do not vanish once tensored with $\mathbb{C}_A$. Adapting the arguments in \cite{DetcherryKalfagianniSikora_SkeinDim} to our purpose, we prove the following
  
  \begin{theorem}\label{theorem_central}
 Let $M$ be a closed $3$-manifold such that $(1)$ the character variety $\mathcal{X}_{\SL_2}(M)$ consists in a finite number of points and is reduced and $(2)$ the skein module $\mathcal{S}_v(M)$ is $A$-tame. Then for every central representation $[\rho_0]\in \mathcal{X}_{\SL_2}(M)$ then $ \mathcal{S}_A(M)_{[\rho]}$ is one-dimensional. So $\mathscr{L}^M \to \mathcal{X}_{\SL_2}(M)$ is a line bundle.
 \end{theorem}
 
 The following is a straightforward adaptation of \cite[Lemma $2.6$]{DetcherryKalfagianniSikora_SkeinDim}.
 
 \begin{lemma}\label{lemma_central_rep} Let $M$ be a closed $3$-manifold such that  the character variety $\mathcal{X}_{\SL_2}(M)$ has a finite number of points. Then the linear morphism
 $$ \Phi: \mathcal{S}_A(M) \to \oplus_{[\rho] \in \mathcal{X}_{\SL_2}(M)} \mathcal{S}_A(M)_{[\rho]}, $$ 
 obtained by taking the direct sum of the quotient maps, is surjective.
 \end{lemma}
 
 \begin{proof}
 For each $[\rho]\in \mathcal{X}_{\SL_2}(M)$, let $f_{[\rho]}\in \mathcal{S}_{\varepsilon}(M)$ be skein element corresponding under $\Psi$ to the regular function $g_{[\rho]}=\Psi(f_{[\rho]})$ defined by $g_{[\rho]}([\rho']):= \delta_{[\rho], [\rho']}$.
 Let $(w_{1, [\rho]}, \ldots, w_{d_{[\rho]}, [\rho]})$ be a basis of $\mathcal{S}_A(M)_{[\rho]}$, so $(w_{i,[\rho]})_{i, [\rho]}$ is a basis of $ \oplus_{[\rho] \in \mathcal{X}_{\SL_2}(M)} \mathcal{S}_A(M)_{[\rho]}$. Since the quotient map $\mathcal{S}_A(M) \to \mathcal{S}_A(M)_{[\rho]}$ is surjective, we can find $v_{i, [\rho]}\in \mathcal{S}_A(M)$ which is sent to $w_{i, [\rho]}$ by the quotient map. Then the set $\{ Ch_A(f_{[\rho]}) v_{i,[\rho]} \} \subset \mathcal{S}_A(M)$ is sent to the basis $(w_{i,[\rho]})_{i, [\rho]}$ by $\Psi$, so $\Psi$ is surjective.
 \end{proof}
 
 \begin{proof}[Proof of Theorem \ref{theorem_central}]
 By Theorem \ref{theorem_reduced}, $\dim(\mathcal{S}_A(M)_{[\rho]})=1$ whenever $[\rho]$ is not central. By Lemma \ref{lemma_X2_open}, when $[\rho]$ is central then $d:= \dim(\mathcal{S}_A(M)_{[\rho]})$ is the same for all central representations. It follows from the non-trivialness of the Witten-Reshetikhin-Turaev invariants, proved in \cite[Proposition $2.4$]{DetcherryKalfagianniSikora_SkeinDim},  that $d\geq 1$. 
 By Lemma \ref{lemma_central_rep} we have $\dim ( \mathcal{S}_A(M) ) \geq |\mathcal{X}^{(0)}\cup \mathcal{X}^{(1)} | + d |\mathcal{X}^{(2)}|$ so it suffices to prove that $\dim(\mathcal{S}_A(M))\leq |\mathcal{X}_{\SL_2}(M)|$ to conclude that $d=1$. We argue like in the proof of \cite[Theorem $3.1$]{DetcherryKalfagianniSikora_SkeinDim}. By the $A$-tameness assumption, we have 
 $$\mathcal{S}_v(M) \cong F \oplus \oplus_{i=1}^n \quotient{\mathbb{Q}[v^{\pm 1}]}{(P_i(v))}$$
 where $F$ is a free $K$-module and the cyclotomic polynomial $\phi_N(v)$ (resp. $\phi_{2N}(v)$) does not divide $P_i(v)$ if $\varepsilon =+1$ (resp. if $\varepsilon =-1$). So, while tensoring with $\mathbb{C}_A$ all torsion modules vanishes and $\mathcal{S}_A(M)\cong \mathcal{S}_v(M) \otimes_K \mathbb{C}_A \cong F\otimes_K \mathbb{C}$ has dimension $\dim_K(F)$. On the other hand, we have
  $$\mathcal{O}[\mathcal{X}_{\SL_2}(M)] \cong \mathcal{S}_{\varepsilon}(M)\cong \mathcal{S}_v(M)\otimes_K \mathbb{C}_{\varepsilon} \cong (F\otimes_K \mathbb{C}) \oplus \mathbb{C}^n $$
  where $n$ is the number of elements with torsion $(v- \varepsilon 1)^n$. Since we assumed that $\mathcal{X}_{\SL_2}(M)$ is reduced (i.e. that its nilradical is trivial) then $\dim(\mathcal{O}[\mathcal{X}_{\SL_2}(M)])=|\mathcal{X}_{\SL_2}(M)|$, so 
  $$ |\mathcal{X}_{\SL_2}(M)| = \dim(\mathcal{S}_{\varepsilon}(M)) \geq \dim_K(F) = \dim (\mathcal{S}_A(M)).$$
  This implies that $\dim(\mathcal{S}_A(M))= |\mathcal{X}_{\SL_2}(M)|$ so $d=1$. This concludes the proof.
 \end{proof}
 
 \subsection{Concluding remarks}\label{sec_conclusion}
 
 Let $M$ be a closed oriented $3$ manifold. The coherent sheaf $\mathscr{L}^M \to \mathcal{X}_{\SL_2}(M)$ is expected to be a line bundle over the whole character variety. 
 The Costantino-Geer-Patureau Mirand invariants of \cite{CGPInvariants} form a section $\sigma^{CGP} \in \Gamma(\mathcal{X}^{(1)}, \mathcal{L}^*)$ of the restriction of the dual of $\mathcal{L}$ to the locus of diagonal non-central representations (which depends on choice of signature for a bounding $4$-manifold $W$ with $\partial W=M$). 
 The fiber of $\sigma^{CGP}$ over $[\rho_{\omega}]$ is the linear map $\mathcal{S}_A(M)_{[\rho_{\omega}]}\to \mathbb{C}$ sending $[\gamma]$ to the CGP invariant $\left< M, [\gamma], \omega\right>$. Similarly, the Witten-Reshetikhin-Turaev invariants form a section of the dual of $\mathscr{L}^M$ over the locus of central representations.
  An important problem is to extend these sections   to the whole character variety  $\mathcal{X}_{\SL_2}(M)$, that is to define invariants $\left< M, [\gamma], [\rho] \right>$ for $[\rho]\in \mathcal{X}^{(2)}$. It is conjectured that such invariants fit into a $\SL_2$-HQFT which would refine the work of Baseilhac-Benedetti \cite{BaseilhacBenedettiHTQFT} and generalize the non-semisimple TQFTs of \cite{BCGPTQFT}. The link invariants of this conjectural HQFT were defined independently in \cite{BGPR_Biquandle} and \cite{KojuQgroupsBraidings}. The construction of these conjectural HQFTs is the main motivations for the authors to consider Problem \ref{problem_classification} and for the present paper.

\appendix

\section{Proof of Proposition \ref{prop_appendixA}}\label{sec_appendixA}

First, the facts that $w_{\alpha}$ is an equivariant isomorphism and the identities of Equation \eqref{eq_Cap} follow from easy computations left to the reader. To prove the other identities, we will adapt the work of Costantino and Murakami in \cite{CostantinoMurakami_6j} to our context. The non-triviality of this adaptation comes from the facts that
\begin{enumerate}
\item we work with $q$ a root of unity of order $N$, whereas it has order $2N$ in \cite{CostantinoMurakami_6j}, 
\item we work with the unrolled quantum group with the additional generator $H$ which is not considered in \cite{CostantinoMurakami_6j}, 
\item we followed the conventions of \cite{CGPInvariants, CGP_unrolledQG, BCGPTQFT, BCGP_Bases} which are completely different from the ones in \cite{CostantinoMurakami_6j}.
\end{enumerate}
To remedy the third problem, let us first establish a dictionary between our notations and the ones in \cite{CostantinoMurakami_6j}. Let $\mathbf{U}$ be the quantum group considered in \cite{CostantinoMurakami_6j} and denote by $\mathbf{E}$, $\mathbf{F}$, $\mathbf{K}^{\pm 1}$ its generators. By definition, they satisfy the relations
$$ \mathbf{K}\mathbf{E}=q\mathbf{E}\mathbf{K}, \quad \mathbf{K}\mathbf{F}=q^{-1}\mathbf{F}\mathbf{K}, \quad [\mathbf{E},  \mathbf{F}]=\frac{\mathbf{K}^2-\mathbf{K}^{-2}}{q-q^{-1}}$$
with the coproduct
$$\Delta(\mathbf{K})=\mathbf{K}\otimes \mathbf{K}, \quad  \Delta(\mathbf{E})=\mathbf{E}\otimes \mathbf{K}+\mathbf{K}^{-1}\otimes \mathbf{E}, \quad  \Delta(\mathbf{F})=\mathbf{F}\otimes \mathbf{K} +\mathbf{K}^{-1}\otimes \mathbf{F}.$$
Therefore, we establish a morphism between the quantum group $\mathbf{U}$ of \cite{CostantinoMurakami_6j} and our quantum group $U^H_q\mathfrak{sl}_2$ by sending $\mathbf{K}$ to $K^{1/2}$, $\mathbf{E}$ to $K^{-1/2}E$ and $\mathbf{F}$ to $FK^{1/2}$, where $K^{1/2}$ is the element acting as $q^{H/2}$. The authors consider highest weight $\mathbf{U}$-modules $V^a$ with the basis $(e_i^a)_{i=0}^{N-1}$ defined by 
$$ \mathbf{E} e_i^a = [i] e_{i-1}^a, \quad \mathbf{F} e_i^a = [2a-i] e_{i+1}^a, \quad \mathbf{K} e_i^a = q^{a-i}e_i^a \quad (e_{-1}^a =e_N^a=0).$$
Setting $2a = \alpha+N-1$, the $\mathbf{U}$-module $V^a$ is isomorphic to the $U_q^H\mathfrak{sl}_2$-module $V_{\alpha}$ via the isomorphism identifying $v_i$ with $\kappa_i^{\alpha}e_i^a$. 
\par Let $\alpha, \beta, \gamma$ and consider $a,b,c$ such that $2a=\alpha+N-1$, $2b=\beta+N-1$ and $2c=\gamma+N-1$. The identity $\alpha+\beta-\gamma\in H_N$ corresponds to the identity $a+b-c \in \{0, \ldots, N-1\}$ and the authors consider $Y_c^{a,b}: V^c\to V^a\otimes V^b$ defined by $Y_c^{a,b} (e_n^c)=\sum_{i,j} \sqrt{-1}^{-c+a+b}C_{i,j,n}^{\alpha, \beta, \gamma}e_i^a\otimes e_j^b$ which, under the above isomorphism corresponds to our operator $Y_{\gamma}^{\alpha, \beta}$ by definition up to the multiplication by the factor $\sqrt{-1}^{-c+a+b}$. Its absence in our formulas will be explained below.
 Similarly, under the same correspondence, our operators $Y_{\alpha, \beta}^{\gamma}$, $\cap_{\alpha}$ and $\cup_{\alpha}$ correspond to the operators $Y_{a,b}^c$, $\cap_{a, N-1-a}$ and $\cup_{a, N-1-a}$ in \cite{CostantinoMurakami_6j} (up to a factor $\sqrt{-1}^{-c+a+b}$ in $Y_{a,b}^c$). To prove Proposition \ref{prop_appendixA}, we need to prove that 
 \begin{itemize}
 \item[(i)] $Y_{\gamma}^{\alpha, \beta}$ is equivariant by verifying that the computations in Appendix $A.1$ of \cite{CostantinoMurakami_6j} still hold when $q$ has order $N$; 
 \item[(ii)] Equation \eqref{eq_Y} holds by verifying that the computations in Appendix $A.2$ of \cite{CostantinoMurakami_6j} still hold when $q$ has order $N$; 
 \item[(iii)] Equation \eqref{eq_3j} holds by verifying that the computations in Appendix $A.3$ of \cite{CostantinoMurakami_6j} still hold when $q$ has order $N$: this is where the fact that we removed the factor $\sqrt{-1}^{-i-j+n}$ from the formula of $C_{i,j,n}^{\alpha, \beta, \gamma}$ is important.
 \end{itemize}

\par $(i)$ First, it is a trivial verification (left to the reader) that the computations done in  \cite[Appendix $A.1$]{CostantinoMurakami_6j} hold word-by-word in the case where $q$ has order $N$ instead of $2N$.
\par $(ii)$ The fact that the equality of \cite[Lemma $1.8$]{CostantinoMurakami_6j} still holds when $q$ has order $N$ is a non-trivial fact; let us detail the computation. Let $\alpha, \beta, \gamma$ and $a,b,c$ as before and write $m:= a+b-c$. The operators $V_{\alpha}\otimes V_{\beta} \to V_{\gamma}$ of the middle and right hand side of Equation \eqref{eq_Y} both send $v_0\otimes v_m$ to a vector proportional to $v_0$. Equating the two proportionality factors reduces to prove that 
$$ C_{N-1, 0, m}^{-\alpha, \gamma, \beta} q^{-a(N-1)} = C_{0, N-1-m,0}^{\gamma, -\beta, \alpha}q^{b(N-1) +m}.$$
Denote by LHS (resp. RHS) the left-hand-side (resp. right-hand-side) of this equality. Developing the formulas in Notations \ref{notations_formulas}, one finds that 
$$
 LHS = \left( q^{(N-1)a} q^{\frac{m}{2}(2b-m+(N-1))}(-1)^m \qbinom{N-1-2a+m}{N-1-2a} \right) q^{-a(N-1)} 
= (-1)^m q^{bm -\frac{m(m+1)}{2}} \qbinom{N-1-2a+m}{N-1-2a}.
$$
Similarly
$$
 RHS = \left(  q^{\frac{(N-1-m)(-2b-(N-1)-(N-1-m))}{2}} \qbinom{2a}{2a-m} \right) q^{b(N-1)+m} 
= q^{bm -\frac{m(m+1)}{2}}  \qbinom{2a}{2a-m}.
$$
And we derived the equality $LHS=RHS$ from the equality
$$ \qbinom{N-1-2a+m}{N-1-2a} = (-1)^m  \qbinom{2a}{2a-m}.$$

\par $(iii)$ Again, that the computations of  \cite[Appendix $A.3$]{CostantinoMurakami_6j} still hold for $q$ of order $N$ is non-trivial and must be checked by a careful analysis. The only difference with the case where $q$ has order $2N$ are that: 
\begin{enumerate}
\item while passing from the last line of page $26$ of  \cite[Appendix $A.3$]{CostantinoMurakami_6j} to the first line of page $27$, the authors used the identity $\qbinom{N-1-y}{N-1-x}=\qbinom{x}{y}$ three times. When $q$ has order $N$, this identity is replaced by  $\qbinom{N-1-y}{N-1-x}=(-1)^{x-y}\qbinom{x}{y}$. It results that the identity is multiplied by three new factors $(-1)^v$, $(-1)^{m}$ and $(-1)^{m+t}$ where $m=a+b-c$ as before and $v+t=N-1-m$, so the formula in the first line of page $27$ gets multiplied by a factor $(-1)^m$.
\item In the last equality of the proof of Lemma $1.10$ in  \cite[Appendix $A.3$]{CostantinoMurakami_6j}, the authors used a  formula of the form $\qbinom{x}{y}=(-1)^{x-y}\qbinom{x-N}{y-N}$ which becomes $\qbinom{x}{y}=\qbinom{x-N}{y-N}$ when $q$ has order $N$. It results that the identity gets multiplied by an additional factor $(-1)^m$ which compensates the previous one.
\end{enumerate}

\section{Proof of Proposition \ref{prop_6j}} \label{sec_appendixB}
 
 Recall that $6S(\alpha, \beta, \gamma; \varepsilon_1, \varepsilon_2)$ is the scalar defined by the skein relation
 $$\adjustbox{valign=c}{\includegraphics[width=8cm]{6jRel.eps}}$$
Without loss of generality, we can suppose that $\alpha+\beta-\gamma \in H_N$.
Set $m:= \frac{\alpha + \beta - \gamma}{2}+\frac{N-1}{2}\in \{0, \ldots, N-1\}$. By evaluating each operator of the above equality in the vector $v_0$, we obtain a vector proportional to $v_0\otimes v_m$. Equating both proportionality scalars we obtain the equality
$$ D_{0, m, 0}^{\alpha, \beta, \gamma} 6S(\alpha, \beta, \gamma; \varepsilon_1, \varepsilon_2)= \sum_{a,b=0, \ldots, N-1} \sum_{n=0,1} D_{a, b, 0}^{\alpha+ \varepsilon_1, \beta+\varepsilon_2, \gamma}
D_{0, N-2+n, a}^{\alpha, k(N-2), \alpha+\varepsilon_1k} E_{N-2+n, b, m}^{(N-2), \beta+\varepsilon_2, \beta}.$$
Set $\theta_1:= \frac{\varepsilon_1-1}{2}$ and $\theta_2:= m + \frac{\varepsilon_2 +1}{2}$.
The summand in the right hand side of the equality vanishes if  $(a,b)\neq (\theta_1 +n, \theta_2-n)$. 
 Indeed, by definition, both $C_{i,j,k}^{\alpha, \beta, \gamma}$ and $E_{i,j,k}^{\alpha, \beta, \gamma}$ vanish if  $i+j-k \neq  \frac{\alpha+\beta-\gamma + (N-1)}{2}$. We deduce the implications $D_{a,b,0}^{\alpha+ \varepsilon_1, \beta + \varepsilon_2, \gamma}\neq 0 \Rightarrow  a+b=m + \frac{\varepsilon_1+\varepsilon_2}{2}$  and $E_{N-2+n, b, m}^{N-2, \beta+ \varepsilon_2, \beta}\neq 0 \Rightarrow b= \theta_2 - n$; so both inequalities can happen only if  $(a,b)=(\theta_1 +n, \theta_2-n)$.

\par 
Therefore one has 
\begin{multline}\label{eq_6j_formula}
\left( D_{0, m, 0}^{\alpha, \beta, \gamma} \right) 6S(\alpha, \beta, \gamma; \varepsilon_1, \varepsilon_2) = 
D_{\theta_1, \theta_2, 0}^{\alpha+\varepsilon_1, \beta+\varepsilon_2, \gamma}
D_{0, N-2, \theta_1}^{\alpha, N-2, \alpha + \varepsilon_1}
E_{N-2, \theta_2, m}^{N-2, \beta+ \varepsilon_2, \beta}
\\ + 
D_{\theta_1 +1, \theta_2-1, 0}^{\alpha+\varepsilon_1, \beta+\varepsilon_2, \gamma}
D_{0, N-1, \theta_1+1}^{\alpha, N-2, \alpha + \varepsilon_1}
E_{N-1, \theta_2-1, m}^{N-2, \beta+ \varepsilon_2, \beta}.
\end{multline}
 
 Write $K:= \mathbb{Q}(A)(X_1, X_2,X_3)$ and let us define a rational function $\mathbb{D}_{i,j,n}(X_1, X_2, X_3) \in  K$ such that $D_{i,j,n}^{\alpha, \beta, \gamma}=\mathbb{D}_{i,j,n}(q^{\frac{\alpha}{2}}, q^{\frac{\beta}{2}}, q^{\frac{\gamma}{2}})$. Set 
 $$ {\mathds{k}}_n (X):= X^{-1}q^{\frac{n(n-1)}{2}}(X^2q^{-1}-X^{-2}q)(X^2q^{-2}-X^{-2}q^2)\ldots (X^2q^{-n}-X^{-2}q^n)\in \mathbb{Q}(A)(X)$$
 so that $\kappa^{\alpha}_n = \mathds{k}_n(q^{\frac{\alpha}{2}})$.  Also set
 $$ f_{a,b}(X) := \frac{1}{ \{b\}!} (X^2 q^{a-1} - X^{-2}q^{1-a}) (X^2q^{a-2} - X^{-2}q^{2-a}) \ldots (X^2 q^{a-b}-X^{-2}q^{b-a})\in \mathbb{Q}(A)(X)$$
 so that $\qbinom{\alpha + N-1+a}{\alpha+N-1+a-b} = f_{a,b}(q^{\frac{\alpha}{2}})$. Writing $m:=  \frac{\alpha + \beta - \gamma}{2}+\frac{N-1}{2}$ as before, we set 
 \begin{multline*}
  \mathbb{C}_{i,j,n}:= (-1)^{j-n} q^{\frac{i-j}{2}}X_1^{-i}X_2^j (f_{0,n}(X_3))^{-1}f_{0,m}(X_3) \\
   \sum_{z+w=n} (-1)^z q^{-\frac{n}{2}(2z-n)}X_3^{2z-n} \qbinom{i+j-n}{i-z}f_{u+z,z}(X_1)f_{w-j,w}(X_2).
 \end{multline*}
 so that $C_{i,j,n}^{\alpha, \beta, \gamma} = \mathbb{C}_{i,j,n}(q^{\frac{\alpha}{2}}, q^{\frac{\beta}{2}}, q^{\frac{\gamma}{2}})$.  Set 
 $$ \mathbb{D}_{i,j,n}(X_1,X_2,X_3):= \frac{\mathds{k}_n(X_3)}{\mathds{k}_i(X_1)\mathds{k}_j(X_2)}  \mathbb{C}_{i,j,n}(X_1,X_2,X_3)$$
 and 
 $$
\mathbb{E}_{i,j,n}:= \frac{\mathds{k}_i(X_1)\mathds{k}_j(X_2)}{\mathds{k}_n(X_3)} \mathbb{C}_{N-1-j, N-1-i, N-1-n}(X_2^{-1}, X_1^{-1}, X_3^{-1})$$
 so that $D_{i,j,n}^{\alpha, \beta, \gamma}= \mathbb{D}_{i,j,n}(q^{\frac{\alpha}{2}}, q^{\frac{\beta}{2}}, q^{\frac{\gamma}{2}})$ and $E_{i,j,n}^{\alpha, \beta, \gamma}= \mathbb{E}_{i,j,n}(q^{\frac{\alpha}{2}}, q^{\frac{\beta}{2}}, q^{\frac{\gamma}{2}})$. Eventually set 
 $$\mathbb{F}_1(X_1,X_2,X_3):=  \mathbb{D}_{\theta_1,\theta_2, 0}(X_1A^{\varepsilon_1}, X_2A^{\varepsilon_2}, X_3) 
 \mathbb{D}_{0, N-2, \theta_1}(X_1, q^{-1}, X_1A^{\varepsilon_1})
 \mathbb{E}_{N-2, \theta_2, m}(q^{-1}, X_2A^{\varepsilon_2}, X_2)$$
  and 
  $$ \mathbb{F}_2(X_1,X_2,X_3):=   \mathbb{D}_{\theta_1+1,\theta_2-1, 0}(X_1A^{\varepsilon_1}, X_2A^{\varepsilon_2}, X_3) 
 \mathbb{D}_{0, N-1, \theta_1+1}(X_1, q^{-1}, X_1A^{\varepsilon_1})
 \mathbb{E}_{N-1, \theta_2-1, m}(q^{-1}, X_2A^{\varepsilon_2}, X_2).
 $$
 
 \begin{definition} Let $\alpha, \beta, \gamma$ such that $\alpha+\beta-\gamma \in H_N$ and $\varepsilon_1, \varepsilon_2 = \pm 1$. The rational function $R \in K$ is
$$ R(X_1,X_2,X_3):= \left( \mathbb{D}_{0,m,0}(X_1,X_2,X_3) \right)^{-1} \left( \mathbb{F}_1(X_1,X_2,X_3) + \mathbb{F}_2(X_1, X_2, X_3) \right).$$
 \end{definition}
 It follows from Equation \eqref{eq_6j_formula} that 
 $ 6S(\alpha, \beta, \gamma; \varepsilon_1, \varepsilon_2)= R(q^{\frac{\alpha}{2}}, q^{\frac{\beta}{2}}, q^{\frac{\gamma}{2}})$, so it remains to prove that $R\neq 0$. Note that, since $R_{i,j,n}^{\alpha, \beta, \gamma} \neq 0$ and $E_{i,j,n}^{\alpha, \beta, \gamma}\neq 0$, then $\mathbb{D}_{i,j,n}\neq 0$ and $\mathbb{E}_{i,j,n} \neq 0$ so $\mathbb{F}_1\neq 0$ and $\mathbb{F}_2 \neq 0$. We need to show that $\mathbb{F}_1\neq -\mathbb{F}_2$.
 \par For $L$ a field, define a discrete valuation $v: L(X) \to \mathbb{Z}\cup \{ \infty \}$ by setting $v(0)=\infty$ and for $f(X)=X^n\frac{P(X)}{Q(X)}$ with $P(0)Q(0)\neq 0$, set $v(f):=n$. For instance, when $L=\mathbb{Q}(A)$, one has: 
 $$ v( \mathds{k}_i(X))= -1 -2i \quad \mbox{ and } \quad v( f_{a,b}(X) ) = -2b.$$
 Setting $L= \mathbb{Q}(A)(X_2,X_3)$, and by precomposing with the natural embedding $K\hookrightarrow L(X_1)$, we get a discrete valuation $v$ on $K$ (for instance $v(X_1^aX_2^bX_3^c)=a$). Using the facts that $v(fg)=v(f)+v(g)$ and that $v(f+g)\geq \min(v(f), v(g))$ with equality when $v(f)\neq v(g)$, we easily compute the following
 $$ v \left( \mathbb{D}_{\theta_1, \theta_2, 0}(X_1A^{\varepsilon_1}, X_2 A^{\varepsilon_2}, X_3) \right) = 2 \theta_1, 
 \quad v \left( \mathbb{D}_{0,N-2,\theta_1}(X_1, q^{-1}, X_1A^{\varepsilon_1}) \right) =5 \theta_1-2N+5, $$
 so 
 $$ v(\mathbb{F}_1)= 7 \theta_1-2N+5.$$
 Similarly
  $$ v \left( \mathbb{D}_{\theta_1+1, \theta_2-1, 0}(X_1A^{\varepsilon_1}, X_2 A^{\varepsilon_2}, X_3) \right) = 2 \theta_1 +2, 
 \quad v \left( \mathbb{D}_{0,N-1,\theta_1+1}(X_1, q^{-1}, X_1A^{\varepsilon_1}) \right) =5 \theta_1-2N+8, $$
 so 
 $$ v(\mathbb{F}_2)= 7 \theta_1-2N+10.$$
 Therefore $v(\mathbb{F}_1)\neq v(\mathbb{F}_2)$ so $\mathbb{F}_1+\mathbb{F}_2 \neq 0$. This implies that $R\neq 0$ and concludes the proof.

\bibliographystyle{amsalpha}
\bibliography{biblio}

\end{document}